\documentclass[10pt]{article}
\usepackage{latexsym}
\usepackage[ansinew]{inputenc}

\usepackage{amsthm}
\usepackage{amsmath}
\usepackage{amssymb}
\usepackage{amsfonts}
\usepackage{color}
\usepackage{textcomp}
\usepackage{latexsym}

\usepackage{mathrsfs}

\usepackage{array}
\usepackage{float}
\usepackage{caption}

\usepackage{graphicx}
\usepackage{anysize}

\usepackage{tabularx}
\setcaptionwidth{0.95\textwidth}%

\newtheorem{lemma}{Lemma}
\newtheorem{proposition}{Proposition}
\newtheorem{theorem}{Theorem}
\newtheorem{corollary}{Corollary}
\newtheorem{definition}{Definition}
\newtheorem{remark}{Remark}
\newtheorem{example}{Example}
\newcommand{\Q}{\mathbb{Q}}
\newcommand{\Z}{\mathbb{Z}}
\newcommand{\N}{\mathbb{N}}
\newcommand{\R}{\mathbb{R}}

\newcommand{\B}{\mathcal{B}}
\newcommand{\M}{\mathcal{M}}
\newcommand{\D}{{\rm Diff}}
\def\Diff{\mathrm{Diff}}
\def\R{\mathbb R}
\def\Z{\mathbb Z}
\def\diff{\mathfrak{diff}}

\def\le{\leqslant}
\def\ge{\geqslant}

\def\kappa{\varkappa}
\def\epsilon{\varepsilon}
\def\cN{\mathcal{N}}

\newfont{\weird}{cmff10}

\normalsize 

\title{{Diffeomorphisms groups of tame Cantor sets and Thompson-like groups}}
\date{}
\author{\begin{tabular}{cc}      Louis Funar &  Yurii Neretin\\      
\small \em Institute Fourier, UMR 5582        &\small \em Math.Dept., University of Vienna, Nordbergstrasse 15 \\  
\small \em Department of Mathematics &\small \em  Vienna, Austria  \\      
\small \em University Grenoble Alpes  &\small \em Institute for Theoretical and Experimental Physics  \\      
\small \em CS40700 &\small\em Mech. Math. Department, Moscow State University
\\ \small \em  38058 Grenoble cedex 9 &\small \em  Kharkevich Institute for Information Transmission Problems\\  
\small \em France &  \small \em Moscow, Russia \\ 
\small \em e-mail: {\tt louis.funar@univ-grenoble-alpes.fr}      
& \small \em e-mail: {\tt neretin@mccme.ru } \\      
\end{tabular}      
}

\begin{document}

\maketitle

\begin{abstract}
The group of $\mathcal C^1$-diffeomorphisms  of any sparse Cantor subset 
of a manifold is countable and discrete (possibly trivial). 
Thompson's groups come out of this construction when we consider 
central ternary Cantor subsets of an interval. Brin's higher dimensional 
generalizations $nV$ of Thompson's group $V$ arise when we consider 
products of central ternary Cantor sets. 
We derive that the $\mathcal C^2$-smooth mapping class group of a 
sparse Cantor sphere pair is a discrete  countable group and 
produce this way versions of the braided Thompson groups.
\end{abstract}




\section{Introduction}
Differentiable structures on Cantor sets have first been considered by Sullivan 
in \cite{Sul}. Our aim is to consider groups of diffeomorphisms of Cantor sets, 
mapping class groups of Cantor punctured spheres and their relations with Thompson-like 
groups. In particular, the usual Thompson groups 
(see \cite{CFP}) can be retrieved as diffeomorphisms 
groups of Cantor subsets of suitable spaces (a line, a circle or a 2-sphere).

Let $M$ be a compact manifold and $C\subset M$ be a Cantor set, namely a {\em compact 
totally disconnected subset without isolated points}. 
Any two Cantor sets are homeomorphic as topological spaces. But if $M$ has dimension $m\geq 3$ 
there exist Cantor sets $C_1$ and $C_2$ embedded into $M$ so that 
there is no ambient homeomorphism of $M$ carrying $C_1$ into $C_2$. 
One says that $C_1$ and $C_2$ are not {\em topologically equivalent}  Cantor set embeddings. 

A Cantor subset of $\R^m$ is {\em   tame} if  there is a homeomorphism of $\R^m$ 
which sends it within a coordinates axis. 
All Cantor sets in $\R^m$, for $m\leq 2$ are tame, but there exist 
uncountably many {\em wild} (i.e. not tame) Cantor sets in $\R^m$, for every $m\geq 3$ (see \cite{Bla}).

One defines similarly {\em smooth equivalence} and {\em smoothly tame} Cantor sets. 
The analogous story for diffeomorphisms is already interesting for $m= 1$, as 
Cantor subsets of $\R$ might be differentiably non-equivalent. 
Our main concern is the image of the group of 
diffeomorphisms of $M$ which preserve a Cantor set $C$ into the 
automorphisms group of $C$. Under fairly general conditions we are able to prove 
that this is a countable group, thereby providing an interesting class of discrete groups. 
For Cantor sets obtained from a topological iterated function system the associated groups 
are non-trivial, while for many self-similar Cantor sets these are versions of Thompson's groups.

\section*{A. General countability statements}
\subsection{Pure mapping class groups}

\begin{definition}
Let $M$ be a compact orientable manifold and $C\subset M$ a Cantor subset. 
We denote by $\D^k(M,C)$ the group of diffeomorphisms 
of class $\mathcal C^k$ of $M$ sending $C$ to itself, by 
 $\D^{k,+}(M,C)$ the subgroup of orientation preserving diffeomorphisms and 
 by $P\D^{k,+}(M,C)$ the subgroup of {\em pure} orientation preserving diffeomorphisms, i.e.  
pointwise preserving  $C$. 

The {\em $\mathcal C^k$-mapping class group} $\M^{k,+}(M,C)$  is the group $\pi_0(\D^{k,+}(M,C))$ of 
$\mathcal C^k$-isotopy classes of orientation preserving diffeomorphisms rel $C$ (i.e. which are identity on $C$) 
of class $\mathcal C^k$.  
The  {\em pure $\mathcal C^k$-mapping class group}  $P\M^{k,+}(M,C)$ is the group $\pi_0(P\D^{k,+}(M,C))$ of 
$\mathcal C^k$-isotopy classes of pure orientation preserving  $\mathcal C^k$-diffeomorphisms rel $C$.
\end{definition}

In a similar vein but a different context, the group of homeomorphisms $\D^0(M,A)$ associated to 
a manifold $M$ and a countable dense set $A\subset M$ 
was studied recently in \cite{DvM}. The authors 
proved there that  $\D^0(M,A)$ is either isomorphic to 
a countably infinite product of copies of $\Q$, when $M$ is 1-dimensional,  
or the Erd\"os subgroup of $l^2$ elements, otherwise. 
In the present setting, when $A$ is closed  and the 
smoothness is at least $\mathcal C^1$ the situation 
is fundamentally different.

If we write $C=\cap_{j=1}^{\infty} C_j$, where 
each $C_j$ is  a compact submanifold of $M$ and $C_{j+1}\subset {\rm int}(C_{j})$ for all $j$, 
then the sequence  $\{C_j\}$ is called a {\em defining sequence} 
for $C$. It is known that $C$ is a tame Cantor set if and only if we can choose $C_j$ to be finite unions of 
disjoint disks. 

\begin{definition}
The class of $\varphi$ in  $P\M^{k,+}(M,C)$ is {\em compactly supported} 
if there exists some defining sequence  $\{C_j\}$ of $C$ and some 
$n$  for which the restriction of $\varphi$ to $C_n$ is 
$\mathcal C^k$-isotopic to identity rel $C$.
\end{definition}

Note that the property of being compactly supported is  independent of the choice of the 
defining sequence.

Our first result is the following:

\begin{theorem}\label{compact}
Let $C$ be a  $\mathcal C^k$-tame Cantor set, namely a Cantor subset of a closed interval  $\mathcal C^k$-embedded in a compact orientable manifold $M$ of dimension at least 2. If $k\geq 2$, then all classes in the 
group $P\M^{k,+}(M,C)$  are compactly supported. In particular, 
the group $P\M^{k,+}(M,C)$ is countable.
\end{theorem}
In contrast, the  topological mapping class group  $P\M^{0,+}(S^2,C)$ is uncountable.  
We might expect $P\M^{k,+}(M,C)$ be countable for $k\geq 2$ even when  $C$ is only a $\mathcal C^0$-tame Cantor 
subset of $M$. 

The following precises the second statement in Theorem \ref{compact}: 
\begin{corollary}\label{inductive}
Let  $C$ be a $\mathcal C^k$-tame Cantor subset of a compact orientable 
surface $M$ and  $\{C_j\}$ be a defining sequence  for $C$
consisting of finite unions of disjoint disks.  
If $k\geq 2$, then  $P\M^{k,+}(M,C)$ coincides with the 
inductive limit $\lim_{j\to \infty} P\M^{k,+}(M-{\rm int}(C_j))$ of 
pure mapping class groups of compact subsurfaces.  
\end{corollary}

Note that, when $N$ is a compact surface  the isomorphism type of $P\M^{k,+}(N)$ is independent of $k$.

\subsection{$\mathcal C^1$-diffeomorphisms groups of Cantor sets}

We now turn to the full mapping class groups. Several  groups  which 
arose recently in the literature could be thought to 
play the role of the mapping class groups for some infinite type surfaces, 
for instance the group $\B$  from \cite{FK1} and its version 
$BV$, which was defined by Brin \cite{brin2} and Dehornoy \cite{dehornoy2}, 
independently. These two groups are braidings of the Thompson group 
$V$ (see \cite{CFP}). Geometric constructions of the same sort permitted 
the authors of \cite{FK2} to derive two  braidings $T^*$ and 
$T^{\sharp}$  of the Thompson group $T$. 

Our next goal is to show that these groups are indeed smooth 
mapping class groups in the usual sense and that most (if not all) 
smooth mapping class groups of Cantor sets are related to some Thompson-like 
groups.

Let assume for the moment that $C\subset M$ is  smoothly tame. 
Set then $\diff^k_M(C)$ and $\diff^{k,+}_M(C)$ for the groups of induced 
transformations of $C$ arising as restrictions of elements 
of $\D^{k}(M,C)$ and $\D^{k,+}(M,C)$, respectively. The $\mathcal C^{k}$ topology on  $\D^{k}(M,C)$
induces a topology on $\diff^k_M(C)$. 

Notice now that we have the exact sequence: 

\begin{equation}\label{fdt}
1\to P\M^{k,+}(M,C)\to \M^{k,+}(M,C)\to \diff^{k,+}_{M}(C)\to 1.
\end{equation}

By Theorem \ref{compact} the group $\M^{k,+}(M,C)$ is discrete countable if and only if 
$\diff^{k,+}_M(C)$ does, when $k\geq 2$ and $M$ is compact (or the interior of a compact manifold). 

Classical Thompson groups can be realized as groups 
of  dyadic piecewise linear homeomorphisms (or bijections) of  an 
interval, circle or a Cantor set (see \cite{CFP,GS}) or  as groups of automorphisms 
at infinity of graphs (respecting or not the planarity), as in \cite{Ner}.  
Notice that the more involved  construction from \cite{GS} provides 
embeddings of Thompson groups into the group of diffeomorphisms of the circle,  
admitting invariant minimal Cantor sets. In particular, Ghys and Sergiescu obtained  
embeddings as discrete subgroups of the group of diffeomorphisms (see \cite{GS}, Thm. 2.3).  

In our setting we see that whenever it is discrete and countable 
the group $\M^{k,+}(M,C)$ is the braiding of 
$\diff^{k,+}_{M}(C)$ according to Corollary \ref{inductive}, as in the cases studied 
in \cite{brin2,dehornoy2,FK1,FK2}.
This pops out the question whether $\diff^{k,+}_{M}(C)$ is a Thompson-like group, in general. 
We were not able to solve this question in full generality and actually 
when $C$ is a generic Cantor set of the interval we expect the 
group  $\diff^{k,+}_{M}(C)$ be trivial. 
To this purpose we introduce the following property of Cantor sets. 

\begin{definition}
The Cantor subset $C$ of an interval is {\em $\sigma$-sparse} if,  
for any $a, b\in C$, there is a complementary
interval $(\alpha,\beta)\subset (a,b)\cap \R\setminus C$ such that 
\begin{equation}
\beta-\alpha\ge \sigma (b-a).
\end{equation}
Moreover $C$ is {\em sparse} if it is $\sigma$-sparse for some $\sigma >0$. 
\end{definition} 

Set $\diff^k(C)=\diff^k_{\R}(C)$, $\diff^{k,+}(C)=\diff^{k,+}_{\R}(C)$ for the sake of notational simplicity. 

\begin{theorem}\label{th:diff-countable}
If $C$ is a sparse Cantor subset of $\R$, 
then the group $\diff^1(C)$ is countable. If $C$ is a sparse Cantor set in $S^1=\R/\Z$, then 
$\diff_{S^1}^1(C)$ is countable. 
\end{theorem}

Theorem \ref{th:diff-countable} cannot be extended to all Cantor 
sets $C$, without additional assumptions, as we can see from 
the examples given in section \ref{examples}.

We have the following more general version of the previous result: 
\begin{theorem}\label{V-type}
If $C$ is a sparse Cantor subset of an interval 
$\mathcal C^1$-embedded into a compact orientable manifold $M$,  
then the group $\diff^1_{M}(C)$ is countable and discrete.
\end{theorem}

Although we only considered smoothly tame Cantor subsets above, 
there is a large supply of topologically tame Cantor subsets in any dimensions 
for which we can prove the countability:

\begin{theorem}\label{products}
Let $C_i$ be sparse Cantor sets in $\R$ and  $C=C_1\times C_2\times\cdots C_n\subset \R^n$.
Then the group $\diff_{\R^n}^1(C)$ is countable.
\end{theorem}

Observe that the Lebesgue measure of a sparse Cantor set is zero. In this direction, notice that 
Deroin, Kleptsyn and Navas recently proved  that invariant Cantor sets of groups of 
real-analytic circle diffeomorphisms have zero Lebesgue measure (see \cite{DKN2}, Cor. 1.17). 
The result cannot be extended to ${\mathcal C}^1$-diffeomorphisms, due to the Denjoy 
counter-examples, but it might hold more generally for all 
${\mathcal C}^2$-diffeomorphisms of the circle according to a conjecture of 
Hector (see the discussion in \cite{DKN2} and \cite{DKN3}, Conj.1.11).

The key point is to show that the stabilizer of a point in this group is a finitely generated 
abelian group (see Lemma \ref{l:stabilizer},  Proposition \ref{stabilizers}). The discreteness of the stabilizers 
seems to be the counterpart to the following unpublished theorem of G. Hector (see \cite{N}):
If the subgroup $G$ of the group $\Diff^{\omega}(S^1)$ 
of analytic diffeomorphisms of the circle  has an exceptional minimal set,  
then the stabilizer  $G_a$ of any point $a$ of the circle 
in $G$ is either trivial or $\Z$. As a corollary every subgroup of $\Diff^{\omega}(S^1)$ 
having a minimal Cantor set is countable.  This is, of course,  not true for subgroups of $\Diff^{\infty}(S^1)$.
Nevertheless, the stabilizer $G_a$ of a subgroup $G\subset\Diff^{2}(S^1)$ with an exceptional minimal set 
cannot contain two germs whose logarithm of their derivatives are rationally independent, according to a classical result of 
Sacksteder (\cite{Sa},   Thm. 2). Sacksteder's result cannot be extended to ${\mathcal C}^1$-diffeomorphisms. 
The proof of our key result is related to Thurston's generalization of Reeb's stability theorem from \cite{Th} and uses 
in an essential way the fact that the Cantor set is sparse while allowing only ${\mathcal C}^1$-smoothness of the diffeomorphisms.

\begin{remark}
Note that mapping class groups $\mathcal M^1(M,C)$ depend essentially on the ambient manifold $M$ (see \cite{AF}). 
On the contrary, if $C\subset {\rm int}(D^n)$ is fixed, then for any embedding of the disk $D^n$ in the interior of some orientable 
$n$-manifold $M^{n}$ the groups $\diff^1_M(C)$ are isomorphic. Moreover, these groups stabilize with respect to the  
standard embeddings $D^n\subset D^{n+k}$, for large enough $k$.  
One could however vary the groups $\diff^1_M(C)$ by allowing 
$C$ to intersect the boundary of $M$ in different patterns (isolated points, Cantor subsets etc).   
\end{remark}

\section*{B. Specific families of Cantor sets}
\subsection{Iterated functions systems}
\begin{definition}\label{contIFS}
A {\em contractive iterated function system} (abbreviated  contractive {\em IFS}) is a finite family 
$\Phi=(\phi_0,\phi_1,\ldots,\phi_n)$ of  contractive maps 
$\phi_j:\R^d\to \R^d$. Recall that a map $\phi$ is {\em contractive} if 
its Lipschitz constant is smaller to unit, namely: 
\[ \sup_{x,y\in \R^d} \frac{d(\phi(x),\phi(y))}{d(x,y)} < 1.\]
\end{definition}

According to Hutchinson (see \cite{Hut}) there exists a unique 
non-empty compact $C=C_{\Phi}\subset \R^d$, called the {\em attractor} of 
the IFS $\Phi$, such that $C=\cup_{j=0}^n \phi_j(C)$. 

\begin{example}\label{centralcantor}
The  central Cantor set $C_{\lambda}$, with $\lambda >2$, is the attractor of the 
IFS $\Phi=(\phi_0,\phi_1)$ on $\R$ given by 
\[ \phi_0(x)=\frac{1}{\lambda}x, \; \phi_1(x)=\frac{1}{\lambda}x +\frac{\lambda-1}{\lambda}.\]
Although the IFS makes sense also when $1< \lambda \leq 2$, in 
this case the attractor is not a Cantor set but the whole interval $[0,1]$. 
\end{example}

Consider now the following type of IFS of topological nature. 
\begin{definition}\label{topIFS} 
Let $U$ be an orientable manifold (possibly non-compact) and 
$\phi_j:U\to U$ be finitely many  orientation preserving homeomorphisms on their image. 
We say that $\Phi=(\phi_0,\phi_1,\ldots,\phi_n)$ has 
a  {\em strict attractive basin} $M$ if $M$ is a compact orientable 
submanifold $M\subset U$ with the following properties:  
\begin{enumerate}
\item $\phi_j(M)\subset {\rm int}(M)$, for all $j\in\{0,1,\ldots,n\}$; 
\item $\phi_i(M)\cap \phi_j(M)=\emptyset$, for any $j\neq i\in\{0,1,\ldots,n\}$. 
\end{enumerate}
We say that the pair $(\Phi, M)$ is an {\em invertible} IFS if $M$ is 
a strict attractive basin for $\Phi$. 
If moreover, $\phi_j$ are $\mathcal C^k$-diffeomorphisms on their image, then we say that the IFS is of class $\mathcal C^k$. 
\end{definition}

The existence of an attractive basin is a topological version of 
uniform contractivity of $\phi_j$. 
There exists then a unique invariant non-empty compact $C_{\Phi}\subset M$ 
with the property that $C_{\Phi}=\cup_{i=0}^n\phi_i(C_{\Phi})$.

\begin{theorem}\label{thompIFS}
Consider a $\mathcal C^1$  contractive  invertible IFS $(\Phi,M)$, $\Phi=(\phi_0,\phi_1,\ldots,\phi_n)$,  
whose strict attractive basin $M$ is diffeomorphic 
to a $d$-dimensional ball.  Then,  
the group $\diff^1_M(C_{\Phi})$ contains the Thompson group 
$F_{n+1}$, when  $M$ is of  dimension $d=1$ and the Thompson group 
$V_{n+1}$, when $d\geq 2$, respectively.  
\end{theorem}

In particular, the groups  $\diff^1_M(C_{\Phi})$ are (highly) 
nontrivial. 

For a clear introduction to the classical Thompson groups
$F,T,V$ we refer to \cite{CFP}. The generalized versions 
$F_n,T_n,V_n$ were considered by Higman (\cite{Hig} and 
further extended and studied by Brown and Stein (see 
\cite{Stein}), Bieri and Strebel (see \cite{Bieri-Strebel}) and Laget  \cite{Laget}. We will recall their definitions in 
section \ref{introdthomp}.

The result of the theorem does not hold when the attractive basin $M$ is not 
a ball. For instance, when $M$ is a 3-dimensional solid torus, by taking   
nontrivial (linked)  embeddings $\cup_{i=0}^n\phi_i(M)\subset M$ we can 
provide examples of wild Cantor sets, some of them 
being  topologically rigid, in which case 
the group $\diff^1_M(C)$ is trivial (see \cite{Shil,Wright}).

\subsection{Self-similar Cantor subsets of the line}
The second part of this paper is devoted to concrete 
examples of groups arising by these constructions, for particular choices 
of Cantor sets. 
We will be concerned in this section  with 
self-similar Cantor sets, namely  attractors 
of  IFS which consist only of similitudes. The typical example 
is  the central ternary Cantor set $C_{\lambda}\subset [0,1]$ of parameter 
$\lambda >2$ from Example \ref{centralcantor}.

Let $\Phi=(\phi_0,\phi_1, \ldots, \phi_n)$ be an IFS of affine 
transformations of $[0,1]$, given by: 
\[
\phi_j(x)= \lambda_jx+ a_j, 
\]
where 
\[
0=a_0< \lambda_0 < a_1 <\lambda_1+a_1 <a_2 <\dots < \lambda_{n-1}+a_{n-1}
<a_n< \lambda_n+a_n=1.
\] 
The last condition means  that the segments 
$\phi_j([0,1])$ are mutually disjoint, so that 
the attractor $C=C_{\Phi}$ is a sparse Cantor subset of $[0,1]$. 
The positive reals $g_j=a_{j+1}-\lambda_j-a_j$ are the initial {\em gaps} 
as they represent the distance between consecutive 
intervals $\phi_j([0,1])$ and $\phi_{j+1}([0,1])$.  The image of $[0,1]$ by 
the elements of the monoid generated by $\Phi$ are called {\em standard intervals}. 

We consider the groups $F_{C}$  and $T_{C}$ defined as follows.
Let $PL(\R,C)$ and $PL(S^1,C)$ be the groups  
of orientation preserving piecewise linear homeomorphisms of $\R$ and $S^1=\R/\Z$ respectively, 
keeping invariant $C$, i.e.  of those homeomorphisms $\varphi$ for which there exists a finite covering of $C$ by standard disjoint intervals 
$\{I_j\}$, ${\mathbf k}_j\in \Z^{n+1}$ and $a_j,b_j\in C$,  
such that 
 \begin{equation}\label{op-affinemap}
 \varphi(x)=b_j+ \Lambda_{\mathbf k_j}(x-a_j), \qquad {\rm for \quad any }\quad x\in I_j,  
 \end{equation} 
where $\Lambda_{\mathbf k}= \prod_{i=0}^n\lambda_i^{k_i}$, for each multi-index $\mathbf k=(k_0,k_1,\ldots,k_n)\in \Z^{n+1}$.
Eventually  $F_{C}$ and $T_{C}$ are the images of 
$PL(\R,C)$ and $PL(S^1,C)$, respectively, in the group of 
homeomorphisms of $C$.  Similarly we have the group 
of piecewise affine exchanges $PE(C)$ which are (not necessarily orientation preserving) 
left continuous bijections of $S^1$ preserving $C$, i.e. 
of those  (not necessarily continuous)  maps $\varphi$ for which there exists a finite covering of $C$ by standard disjoint intervals 
$\{I_j\}$, ${\mathbf k}_j\in \Z^{n+1}$ and $a_j,b_j\in C$,  
such that 
 \begin{equation}
 \varphi(x)=b_j\pm \Lambda_{\mathbf k_j}(x-a_j), \qquad {\rm for \quad any }\quad x\in I_j,  
 \end{equation}

We denote by $V^{\pm}_{C}$ its image into 
the group of homeomorphisms of $C$.  We denote by $V_C\subset V^{\pm}_C$ the subgroup obtained 
by requiring the restrictions of $\varphi$ to each  standard interval $I_j$  be orientation-preserving, as in  (\ref{op-affinemap}).

\begin{definition}\label{genericity}
The self-similar Cantor set $C\subset [0,1]$ satisfies the 
genericity condition (C) if  
\begin{enumerate}
\item either all homothety ratios $\lambda_i$ are equal and all 
initial generation gaps $g_{\alpha}$ are equal;
\item  or the factors $\lambda_i$ and the gaps $g_{\alpha}$ are incommensurable,  
in the following sense: 
\begin{enumerate}
\item $\Lambda_{\mathbf k}g_{\alpha}=g_{\beta}$ implies that $\mathbf k=0$ and $\alpha=\beta$; 
\item there exists no permutation $\sigma$ different from identity 
and $\mathbf k, \mathbf  k_{\alpha}\in\Z_+^{n+1}$ such that for all $\beta$ we have: 
\[   \frac{g_{\sigma(\beta)}}{g_{\beta}}= \Lambda_{-\mathbf k_{\beta} +\frac{1}{n}\sum_{\alpha=1}^n\mathbf k_{\alpha}}.\]
\end{enumerate}
\end{enumerate}
\end{definition}

\begin{theorem}\label{thompgen}
Let $C\subset [0,1]$ be a self-similar Cantor set satisfying the genericity condition (C). 
Then for every $\varphi\in\diff^{1,+}(C)$ 
we can find a covering of $C$ by a 
finite collection of disjoint standard intervals $\{I_j\}$, whose images are also standard intervals, 
integers ${\mathbf k}_j\in \Z^{n+1}$ and $a_j,b_j\in C$,  
such that the restriction of the map $\varphi$ has the form
 \begin{equation}\label{affine}
 \varphi(x)=b_j+ \Lambda_{{\mathbf k}_j}(x-a_j), \qquad {\rm for \quad any }\quad x\in I_j\cap C.  
 \end{equation} 
In particular, $\diff^{1,+}(C)$ is isomorphic to $F_{C}$, 
$\diff^{1,+}_{S^1}(C)$ is isomorphic to $T_{C}$ and  
$\diff^{1,+}_{S^2}(C)$ is isomorphic to $V^{\pm}_{C}$. Moreover, these are isomorphic to 
the Thompson groups $F_{n+1}, T_{n+1}$ and the signed Thompson group $V^{\pm}_{n+1}$, respectively.   
 \end{theorem}

The main points in the statement of the theorem are the finiteness of the covering and the fact that 
the intervals are standard. 
\begin{remark}
If the self-similar Cantor set does not satisfy the genericity condition $(C)$, then the same proof provides 
for every $\varphi\in\diff^{1,+}(C)$ a covering of $C$ by a 
finite collection of disjoint intervals $\{I_j\}$, integers ${\mathbf k}_j\in \Z^{n+1}$ and $a_j,b_j\in C$,  
such that the restriction of the map $\varphi$ is given by the formula (\ref{nongeneric}), a slight generalization
of (\ref{affine}) above. However, it is not clear that one could assume that the images of $\{I_j\}$ are standard intervals and, in particular, that $\diff^{1,+}(C)$ is isomorphic to  some Thompson group.  
\end{remark}


We derive easily now the following interpretation for the Thompson groups
and their braided versions: 

\begin{corollary}\label{thomp2}
\begin{enumerate}
\item Let $C$ be the image of the 
standard ternary Cantor subset into the equatorial circle of the 
sphere $S^2$ and $k\geq 2$. 
\begin{enumerate}
\item The smooth  
mapping class group $\M^{k,+}(D_+^2,C)$ is the Thompson 
group $T$, where $D_+^2$ denotes the upper hemisphere; 
\item The smooth  
mapping class group $\M^{k,+}(S^2,C)$ is the group of half-twists $\mathcal B^{1/2}$ from \cite{FN}. 
\end{enumerate}
\item Let $C$ be the standard ternary Cantor subset of 
an interval contained in the interior of a 2-disk $D^2$ and $k\geq 2$. 
Then $\M^{k,+}(D^2,C)$ is the group of half-twists  of the punctured disk (see \cite{AF}). 
\end{enumerate}
\end{corollary}

\begin{remark}
The group of half twists $\mathcal B^{1/2}$ is an extension of the signed Thompson group $V^{\pm}$ by the compactly supported 
pure mapping class group of $S^2-C$.  It is similar to the braided Thompson group $\mathcal B$ 
from \cite{FK1} (see section \ref{thommcg}), which is an extension of $V$ by the same pure mapping class group, in particular it is 
also finitely presented (see \cite{AF}).   
\end{remark}

\begin{remark}\label{Hausdorff}
The central ternary Cantor sets $C_{\lambda}$ are pairwise non-diffeomorphic, i.e. 
there is no $\mathcal C^1$ diffeomorphism of $\R$ 
sending $C_{\lambda}$ into $C_{\lambda'}$ for $\lambda\neq \lambda'$.
Indeed, if it were such a diffeomorphism then the Hausdorff dimensions
of the two Cantor sets would agree, while the Hausdorff dimension of $C_{\lambda}$ 
is $\frac{\log 2}{\log \lambda}$ (see \cite{Falc}, Thm. 1.14).  Nevertheless, 
the groups $\diff^{1,+}(C_{\lambda})$ are all isomorphic, for $\lambda >2$, according to 
Theorem \ref{thompgen}.    
\end{remark}

We notice that a weaker version of our Theorem \ref{thompgen} concerning the
form of  ${\mathcal C}^1$-diffeomorphisms of the central Cantor sets $C_\lambda$, was already 
obtained in (\cite{BMPV}, Proposition 1).

A case which attracted considerable interest is that of 
bi-Lipschitz homeomorphisms of Cantor sets (see \cite{CP,FM} and the recent \cite{RRY,XX}). 
In particular, the results of Falconer and Marsh \cite{FM} imply that every bi-Lipschitz homeomorphism of a Cantor set is given by a 
pair of possibly infinite coverings of the Cantor set by disjoint intervals 
and affine homeomorphisms between the corresponding intervals.  
Notice that any countable subgroup of $\Diff^0(S^1)$ (or $\Diff^0([0,1])$ can be conjugated 
(by a homeomorphism) into the corresponding group of bi-Lipschitz homeomorphisms (see \cite{DKN}, Thm. D).

\subsection{Self-similar Cantor dusts}
The next step is to go to higher dimensions.
Examples of Blankenship (see \cite{Bla}) show that there exist 
wild Cantor sets in $\R^n$, for every $n\geq 3$. 
A   Cantor set $C$ is tame if and only if 
for every $\varepsilon >0$  there exist finitely many  disjoint piecewise linear cells 
of diameter smaller than $\varepsilon$ whose interiors cover $C$. 
In particular, products of tame Cantor sets are tame. More generally, 
the product of a Cantor subset of $\R^n$ with any compact $0$ dimensional 
subset $Z\subset \R^m$ is a tame Cantor subset of $\R^{m+n}$ (see \cite{McM}, Cor.2). 
   
In order to emphasize the role of the embedding we will consider 
now  the simplest Cantor subsets, which although tame they are not 
smoothly tame.  
Let $C_{\lambda}^n\subset \R^n$ be the Cartesian product of 
$n$ copies of $C_{\lambda}$, where $n\geq 2$ and $\lambda >2$, which is itself a Cantor set.  
\begin{theorem}\label{highthomp}
Let $\varphi\in\diff_{\R^n}^{1,+}(C_{\lambda}^n)$, where  $\lambda >2$.
Then there is a covering of $C_{\lambda}^n$ by a 
finite collection of disjoint standard parallelepipeds $\{I_j\}$, 
integers $k_{j,i}\in \Z$ and $a_{j,i},b_{j,i}\in C_{\lambda}$,  
such that:
 \begin{equation}
 \varphi(x)=(b_{j,i} + \lambda^{k_j}(x_i-a_{j,i}))_{i=1,n}\circ S_{j}, \qquad {\rm for \quad any }\quad x\in I_j\cap C_{\lambda}^n. 
 \end{equation} 
where $S_j$ is an orientation preserving symmetry of the cube.  
In particular, $\diff_{\R^n}^{1,+}(C_{\lambda}^n)$ is isomorphic to 
the $n$-dimensional Brin group  $nV^{sym}$ decorated by the group $D_n$ 
of positive symmetries of the cube (see section \ref{introdthomp}).  
 \end{theorem}

Notice that in a series of papers (see \cite{brin1,brin3,BL,HM})  by 
Brin, Bleak and Lanoue, Hennig and Matucci the authors proved 
that the higher dimensional Thompson groups $nV$  defined by Brin 
are pairwise non-isomorphic finitely presented simple groups (see also 
\cite{Rubin1,Rubin2}).

\begin{remark}
Note that the group $\diff_{[0,1]^n}^{1}(C_{\lambda}^n)$ is a proper subgroup 
of $\diff_{\R^n}^{1}(C_{\lambda}^n)$.  
\end{remark}

{\bf Acknowledgements.} 
The authors are grateful to B. Deroin, L. Guillou, P. Haissinsky, S.Hurtado,  I. Liousse, V. Sergiescu and M.Triestino for useful discussions 
and to the referees for having thoroughly read this paper, for their corrections and comments. 
The first author was supported by the ANR 2011 BS 01 020 01 ModGroup and 
the second author  by the FWF  grant  P25142.
Part of this work was done during authors' visit at 
the Erwin Schr\"odinger Institute, whose hospitality and 
support are acknowledged.

\section{Definition of Thompson-like groups}\label{introdthomp}
The standard reference for the classical 
Thompson groups is \cite{CFP}. 
For the sake of completeness we provide here the basic 
definitions from several different perspectives, which lead naturally 
the path to the generalizations considered by Brown and Stein and further 
to the high dimensional Brin groups.

\subsection{Groups of piecewise affine homeomorphisms/bijections}

{\it Thompson's group  $F$} 
is the group of piecewise dyadic affine homeomorphisms of the interval $[0,1]$. Namely, for each $f\in F$, there exist two dyadic subdivisions of $[0,1]$, $a_0=0<a_1<\ldots<a_n=1$ and $b_0=0<b_1\ldots<b_n$, with $n\in \N^*$,  i.e. such that $a_{i+1}-a_i$ and $b_{i+1}-b_i$ belong to $\{\frac{n}{2^k},\; n, k\in \N\}$, so that 
 the restriction of $f$ to $[a_i,a_{i+1}]$ is the unique increasing affine map  onto  $[b_i, b_{i+1}]$.

Therefore, an element of $F$ is completely determined by the data of two dyadic subdivisions of  $[0,1]$ having the same cardinality.

Let us identify the circle to the quotient space $[0,1]/0\sim 1$. {\it Thompson's group $T$}  
 is the group of piecewise dyadic affine orientation preserving homeomorphisms   of the circle.
In other words, for each $g\in T$, there exist two  dyadic  subdivisions of $[0,1]$, $a_0=0<a_1<\ldots<a_n=1$ and $b_0=0<b_1\ldots<b_n$, with  $n\in \N^*$, and $i_0\in \{1,\ldots, n\}$, such that, for 
each $i\in \{0,\ldots, n-1\}$, the restriction of $g$ to $[a_i,a_{i+1}]$ is the unique increasing map onto $[b_{i+i_0}, b_{i+i_0+1}]$. The indices must be understood  modulo $n$.

Therefore, an element of $T$ is completely determined by the data of two dyadic subdivisions of $[0,1]$ having the same cardinality, 
say $n\in \N^*$, plus an integer  $i_{0}$ mod $n$.

Finally, {\it Thompson's group  $V$}  is the group of bijections of $[0,1[$, which are right-continuous at each point, piecewise nondecreasing and dyadic affine. In other words, for each $h\in V$, there exist two dyadic subdivisions of $[0,1]$, $a_0=0<a_1<\ldots<a_n=1$ and $b_0=0<b_1\ldots<b_n$, with $n\in \N^*$, 
and a permutation $\sigma\in \mathfrak{S}_n$, such that, for each  $i\in\{1,\ldots,n\}$, the restriction of  $h$ to $[a_{i-1},a_i[$ is the unique  nondecreasing affine map onto $[b_{\sigma(i)-1}, b_{\sigma(i)}[$.
It follows that an element $h$ of  $V$ is completely determined by the data of two dyadic subdivisions of $[0,1]$ having the same cardinality, say $n\in \N^*$, plus a permutation $\sigma\in \mathfrak{S}_n$. Denoting $I_{i}=[a_{i-1},a_i]$ and $J_i= [b_{i-1}, b_i]$, these data can be summarized into a triple $((J_i)_{1\leq i\leq n},(I_i)_{1\leq i\leq n},\sigma\in \mathfrak{S}_n)$.

The {\em signed Thompson group} $V^{\pm}$  is the
group of right-continuous bijections of the unit circle that map the set of
dyadic rationals to itself, are differentiable except at finitely many points,
and such that, on every interval of differentiability, they are affine maps
whose derivatives are (positive or negative) powers of $2$. 

We have obvious inclusions  $F\subset T\subset V\subset V^{\pm}$. R.J.  Thompson proved in 1965 that $F,T$ and $V$ are finitely presented groups and that $T$ and $V$ are simple  (cf. \cite{CFP}). The group $V^{\pm}$ is also finitely presented and simple (see \cite{AF}). The group $F$ is not perfect, as $F/[F,F]$ is isomorphic to $\Z^ 2$, but $F'=[F,F]$  is simple. However, $F'$ is not finitely generated (this is related to the fact that an element $f$ of  $F$ lies in $F'$ if and only if its support is included  in $]0,1[$).


\subsection{Groups of diagrams of finite binary trees}

A {\it finite binary rooted planar tree} \index{tree} is a finite planar tree having a unique 2-valent vertex, called the {\it root}, a set of monovalent vertices called the {\it leaves}, and whose other vertices are 3-valent. The planarity of the tree provides a canonical labelling of its leaves, in the following way. Assuming that the plane is oriented, the leaves are labelled from  1 to $n$, from left to right, the root being at the top and the leaves at the bottom.

There exists a bijection between the set of dyadic subdivisions of $[0,1]$ and the set of  finite binary rooted planar trees. Indeed,  given such a tree, one may label its  vertices by dyadic intervals in the following way. First, the root is labelled by $[0,1]$. Suppose that a vertex is labelled by $I=[\frac{k}{2^n}, \frac{k+1}{2^n}]$, then its two descendant vertices 
are labelled by the two halves  $I$: $[\frac{k}{2^n}, \frac{2k+1}{2^{n+1}}]$ for the left 
one and  $[\frac{2k+1}{2^{n+1}},\frac{k+1}{2^n}]$ for the right one. 
Finally, the dyadic subdivision associated to the tree is 
the sequence of intervals which label its  leaves.

Thus, an element $h$ of $V$ is represented by a triple $(\tau_1,\tau_0,\sigma)$, where $\tau_0$ and $\tau_1$ have the same number of leaves $n\in \N^*$, and  $\sigma\in\mathfrak{S}_n$. Such a triple will be called a {\it symbol} for $h$.
It is convenient to interpret the permutation $\sigma$ as the bijection $\varphi_{\sigma}$ which
maps the $i$-th leaf of the source tree  $\tau_0$ to the $\sigma(i)$-th 
leaf of the target tree $\tau_1$.
When  $h$ belongs to  $F$, the permutation  $\sigma$ is  identity 
and the symbol reduces to a pair of trees $(\tau_{1},\tau_{0})$.

Now,  two symbols are equivalent if they represent the same 
element of $V$ and one denotes by $[\tau_1,\tau_0,\sigma]$ the equivalence class. 
The composition law of piecewise dyadic 
affine bijections is pushed out on the set of equivalence classes of symbols in the following way. In order to define 
 $[\tau_1',\tau_0',\sigma']\cdot[\tau_1,\tau_0,\sigma]$, one may suppose, at the price of refining both symbols, that
  the tree   $\tau_1$ coincides with the tree $\tau_0'$. Then the product of the two symbols is 
\[ [\tau_1',\tau_1,\sigma']\cdot[\tau_1,\tau_0,\sigma]= [\tau_1',\tau_0,\sigma'\circ\sigma]. \]
    It follows that $V$ is isomorphic to the group of  equivalence classes of symbols endowed with this internal law.
Now, every element of $V^{\pm}$ is encoded 
by an {\em enhanced symbol} $(T,T',\sigma,\varepsilon)$, where $T, T'$ are admissible trees, $\sigma:\partial T\to \partial T'$ is a bijection and $\varepsilon \in (\Z/2\Z)^{\partial T'}$, up to equivalence.

 \subsection{Partial automorphisms of trees} 

The beginning of the article \cite{GreS} formalizes a change of point of view, consisting in considering, not the finite binary trees, 
but their complements in the infinite binary tree.

Let ${\cal T}_2$ be the infinite binary rooted planar tree (all its vertices other than the root are 
3-valent). Each finite binary rooted planar tree $\tau$ can be embedded in a unique way into ${\cal T}_2$, assuming that the embedding maps the root of $\tau$ onto the root of ${\cal T}_2$, and respects the orientation.  Therefore, $\tau$ may be identified with 
a subtree  of ${\cal T}_2$, whose root coincides with that of  ${\cal T}_2$. 

\begin{definition}[ cf. \cite{ka-se}]
A {\em partial isomorphism} \index{partial isomorphism} of ${\cal T}_2$ consists of the data of two finite binary rooted subtrees $\tau_0$ and $\tau_1$ of ${\cal T}_{2}$ having the same number of 
leaves $n\in\N^*$, and an isomorphism $q: {\cal T}_2\setminus \tau_0\rightarrow {\cal T}_2\setminus \tau_1$. The complements 
of  $\tau_0$ and $\tau_1$ have $n$ components, each one isomorphic to  ${\cal T}_2$, 
which are enumerated from 1 to $n$ according to the labeling of the leaves 
of the trees  $\tau_0$ and $\tau_1$. Thus, 
$ {\cal T}_2\setminus \tau_0=T^1_0\cup\ldots\cup T^n_0$ and $ {\cal T}_2\setminus \tau_1=T^1_1\cup\ldots\cup T^n_1$ where the $T^i_j$'s are the 
connected components. 
Equivalently, the partial isomorphism of ${\cal T}_2$
is given by a permutation $\sigma\in \mathfrak{S}_n$ and,  for $i=1,\ldots,n$, an  isomorphism
 $q_i: T^i_0\rightarrow T^{\sigma(i)}_1$.

Two partial automorphisms $q$ and $r$ can be composed if and only if the target of $r$ coincides with the source of $r$. One gets the partial automorphism $q\circ r$.
The composition provides a structure of inverse monoid on the  set of partial automorphisms. 


\end{definition}

 Let $\partial {\cal T}_2$ be the boundary of  ${\cal T}_2$ (also called the set of  ``ends'' of ${\cal T}_2$) 
endowed with its usual topology, for which it is a Cantor set.
Although a partial automorphism does not act (globally) on the tree, it does  act on its boundary.
 One has therefore a morphism from the monoid of partial 
 isomorphism into the homeomorphisms of  $\partial{\cal T}_2$, 
 whose image  $N$ is the {\it spheromorphisms group of  Neretin} (see \cite{Ner}). 

Thompson's group  $V$ can be viewed as the subgroup of  $N$ which is the image of those
partial automorphisms which respect the local orientation of the edges.

\subsection{Generalizations following Brown and Stein, Bieri and Strebel}
Brown considered in \cite{Br} similar groups $F_{n,r}\subset T_{n,r}\subset V_{n,r}$, extending previous work of Higman,  
which were defined as in the last two constructions above but using instead of binary trees forests of $r$ copies of $n$-ary trees so that $F, T, V$ correspond to $n=2$ and $r=1$. 
The isomorphism type of $V_{n,r}$ and $T_{n,r}$ only depends on $r$ (mod $n$) while 
$F_{n,r}$ depends only on $n$.  We drop the subscript $r$ when $r=1$. 
These groups are finitely presented and of type $FP_{\infty}$  according  to \cite{Br-Geo} for the case of 
$F$ and $T$ and then  (\cite{Br}, thm. 4.17)  for its extension to all other groups from this family.
Moreover, Higman have proved (see \cite{Hig}) that 
$V_{n,r}$ has a simple subgroup of index ${\rm g.c.d}(2,n-1)$, and this was extended by Brown who showed 
that $F_{n}$ have simple commutator and $T_{n,r}$  have simple double commutator groups (see \cite{Br} for more details and refinements).

 One can obtain these groups also by considering $n$-adic piecewise affine 
homeomorphisms (or bijections) of $[0,r]$ (with identified endpoints for $T_{n,r}$) i.e.  having 
singularities in $\Z\left[\frac{1}{n}\right]$ and derivatives in $\{n^a, a\in \Z\}$. 
This point of view was taken further  by Bieri, Strebel and Stein in \cite{Bieri-Strebel,Stein}.  Specifically, given a multiplicative subgroup $P\subset \R$, a $\Z[P]$-submodule $A\subset \R$ satisfying $P\cdot A = A$, and a positive $r\in A$, one can consider 
 the group $F_{A,P,r}$ of those  ${\rm PL}$ homeomorphisms of $[0,r]$ with finite singular set in $A$ and all slopes in $P$. 
   There are similar families $T_{A,P,r}$ and $V_{A,P,r}$. 
   Brown and Stein proved that $F_{ \Z\left[ \frac{1}{n_1n_2\cdots n_k}\right], \langle n_1,n_2, \ldots, n_k\rangle, r}$ 
   is finitely presented of $FP_{\infty}$ type.  Furthermore  $F_{A,P,r}$ and $V_{A,P,r}$ have simple commutator subgroups,
    while $T_{A,P,r}$ have  simple second commutator subgroup.  
 
The signed version $V_{n,r}^{\pm}$ of  $V_{n,r}$ is defined as above, by allowing both orientation preserving and orientation reversing piecewise affine homeomorphisms.

\subsection{Mapping class groups of infinite surfaces and braided Thompson groups}\label{thommcg}

Let  $\mathscr{S}_{0,\infty}$ be the oriented surface of genus zero, 
which is the following  inductive limit of compact oriented genus zero surfaces with boundary 
$\mathscr{S}_{n}$. Starting with a cylinder $\mathscr{S}_{1}$, one gets $\mathscr{S}_{n+1}$ from 
 $\mathscr{S}_{n}$ by gluing a pair of pants  (i.e. a three-holed sphere) along 
 each boundary circle of $\mathscr{S}_{n}$. This construction yields, for each $n\geq 1$, an embedding 
  $\mathscr{S}_{n}\hookrightarrow \mathscr{S}_{n+1}$, 
  with an orientation on $\mathscr{S}_{n+1}$ compatible with that of  $\mathscr{S}_{n}$. 
  The resulting inductive limit (in the topological category) of the $\mathscr{S}_{n}$'s is the surface
   $\mathscr{S}_{0,\infty}={\displaystyle \lim_{\stackrel{\rightarrow}{n}} \mathscr{S}_{n}}$. 
  
By the above construction, the surface $\mathscr{S}_{0,\infty}$ is the union of a cylinder and of countably many 
pairs of pants.  This topological decomposition of  $\mathscr{S}_{0,\infty}$
will be called the {\em canonical pair of pants decomposition}.

The set of isotopy classes of  orientation-preserving  homeomorphisms  of $\mathscr{S}_{0,\infty}$ 
is an {\em uncountable} group. 
By restricting to a certain type of homeomorphisms (called asymptotically rigid),
we shall obtain countable subgroups (see \cite{FK1,FK2}).

Any connected and compact subsurface of $\mathscr{S}_{0,\infty}$ 
which is the union of the cylinder and finitely many pairs of pants of the canonical decomposition will be called an {\it  admissible subsurface}
\index{admissible subsurface} of $\mathscr{S}_{0,\infty}$.
The {\it type} of such a subsurface $S$ is the number of connected components in its boundary.

A {\em rigid structure} on $\mathscr{S}_{0,\infty}$ is  given by a pants decomposition together with 
a set of disjoint proper arcs joining distinct ends such that each pair of pants intersects essentially only three arcs 
which join different boundary circles. One component of complement of the union of arcs is called the {\em visible side}.  
We fix a rigid structure on the surface underlying the canonical pants decomposition.

\begin{definition}[following  \cite{ka-se,FK1}]\label{asy}
A homeomorphism $\varphi$ of $\mathscr{S}_{0,\infty}$ is {\em asymptotically rigid}  if there exist two 
admissible  subsurfaces $S_0$ and $S_1$ having the same type, such that 
$\varphi(S_{0})=S_{1}$ and whose restriction 
$\mathscr{S}_{0,\infty}\setminus{S_0}\rightarrow \mathscr{S}_{0,\infty}\setminus{S_1}$ is rigid, 
meaning that it maps each pants (of the canonical pants decomposition) onto a pants and the visible side onto the visible side. 
If we drop the last requirement we say that the homeomorphism $\varphi$ is {\em asymptotically
quasi-rigid}. 
  
\vspace{0.1cm}\noindent 
The {\em asymptotically rigid}  and {\em quasi-rigid mapping class groups} 
\index{asymptotically rigid mapping class group} of $\mathscr{S}_{0,\infty}$ 
are the groups of isotopy classes of asymptotically rigid  and quasi-rigid homeomorphisms, respectively. 
\end{definition}

The {\em asymptotically rigid mapping class group}  $\mathcal B$ and the {\em quasi-rigid mapping class group} 
$\mathcal B^{1/2}$ of $\mathscr{S}_{0,\infty}$ are  finitely presented groups (see \cite{FK1,AF}) which fit into the exact sequences:
\[ 1 \to P\mathcal M (\mathscr{S}_{0,\infty}) \to \mathcal B \to V\to 1,\]
\[ 1 \to P\mathcal M (\mathscr{S}_{0,\infty}) \to \mathcal B^{1/2} \to V^{\pm}\to 1.\]
Some very similar versions of the same group (using a Cantor disk instead of a Cantor sphere or  a more combinatorial framework) were obtained independently by Brin (\cite{brin2}) and Dehornoy (\cite{dehornoy2}). 
We will call any version of them as {\em braided Thompson groups}.

\subsection{Brin's groups $nV$ and their decorated versions}

A rather different direction was taken in the seminal paper \cite{brin1} of Brin, where the author constructed 
a family of countable  groups $nV$ acting as homeomorphisms of the product of $n$-copies of the standard triadic 
Cantor, generalizing the group $V$ which occurs for $n=1$. 

Let  $I^n\subset \R^n$ denote the unit cube. A {\em numbered pattern} is a finite dyadic partition 
of $I^n$ into parallelepipeds along with a numbering. A dyadic partition is obtained from the cube 
by dividing at each step of the process one parallelepiped into two equal halves  by a  cutting hyperplane 
parallel to one of the coordinates hyperplane. 

One definition of $nV$ is as the group of piecewise affine (not continuous!) transformations associated to pairs of 
numbered patterns. Given the numbered patterns $P=(L_1,L_2,\ldots, L_n)$ and 
$Q=(R_1,R_2,\ldots,R_n)$, we set $\varphi_{P,Q}$ for the unique piecewise affine 
transformation of the cube sending  affinely each  $L_i$ into $R_i$ and preserving the coordinates hyperplanes. 
Thus $nV$ is the group of piecewise affine transformations of the form $\varphi_{P,Q}$, with $P,Q$ running over the set of all possible dyadic partitions.

Another description is as a group of homeomorphisms of the product $C^n$ of the standard triadic Cantor set $C$. 
Parallelepipeds in a dyadic partition correspond to closed and open (clopen) subsets of $C^n$. 
Every dyadic cutting hyperplane $H$ subdividing some parallelepiped $R$ into two halves determines a 
parallel shadow  (open) parallelepiped in $R$  whose width is one third of the width of $R$ in the direction orthogonal 
to $H$. 
 Notice then that the complement of the union of all shadow parallelepipeds is $C^n$.
Every pattern $P=(R_1,R_2,\ldots,R_n)$ determines a numbered collection of parallelepipeds 
$X_P=(X(R_1),X(R_2),\ldots, X(R_n))$ whose 
complementary   is the set of shadows parallelepipeds of those cutting hyperplanes used to built $P$.  
Then $A(R_i)=X(R_i)\cap C^n$ form a clopen partition of $C^n$. 
For a pair of patterns $P,Q$ we define the homeomorphism 
$h_{P,Q}$ of $C^n$ as the unique homeomorphism which sends affinely $A(L_i)$ into $A(R_i)$ 
and preserves the orientation in each coordinate. This amounts to say that 
$h_{P,Q}$ is the restriction to $C^n$ of the piecewise affine transformation 
sending affinely $X(L_i)$ into $X(R_i)$ and preserving the coordinates hyperplanes.

The groups $nV$ are simple (see e.g. \cite{brin1,brin3}) and finitely presented (see \cite{HM}). 
The stabilizer at some $a\in C^n$ of the (germs of) homeomorphisms in $nV$ 
is isomorphic to $\Z^{r(a)}$, where $r(a)$ is the number of rational coordinates of $a$. 
This implies that the groups $nV$ are pairwise non-isomorphic (see \cite{BL} for details).

We could of course extend this construction to arbitrary  products 
of central Cantor sets $C_{\lambda}$ in the spirit of Brown and Stein, Bieri and Strebel as above. 

As in the case of groups $V_{n,r}$ there exists a decorated version $nV^{sym}$ of $nV$ by allowing piecewise affine 
transformations $h_{P,Q}$  to be arbitrary affine isomorphisms between  $A(L_i)$ into $A(R_i)$, not necessarily 
preserving each coordinate hyperplane. We will say that  $nV^{sym}$ is the {\em $n$-dimensional Brin group decorated by 
$D_n$}, where $D_n$ denotes the group of orientation preserving symmetries  of the cube. Its elements correspond to numbered patterns $P=(L_1,L_2,\ldots, L_n)$ and $Q=(R_1,R_2,\ldots,R_n)$, along with a $n$-tuple $\Sigma=(\sigma_1,\sigma_2,\ldots,\sigma_n)$ of orientation preserving symmetries of the $n$-cube. The map $\varphi_{P,Q,\Sigma}$ consists of  the unique piecewise affine transformation sending $\sigma_i(L_i)$ into $R_i$.  Recall that $D_n$ is the group 
of orthogonal $n\times n$ matrices with integer entries and unit determinant. One defines in the same way the 
{\em $n$-dimensional Brin group  $nV^{\pm sym}$ decorated by $O_n$}, where $O_n$ denotes the hyperoctahedral group of all 
symmetries  of the cube, which  is the higher dimensional generalization of $V^{\pm}$.

\section{Proof of general countability statements}
\subsection{Proof of Theorem \ref{compact}}
We parameterize  the interval $E$ containing the Cantor set $C$ by the 
$\mathcal C^k$-curve $\gamma:[0,1]\to M$ and denote by 
$A\subset [0,1]$ the preimage of $C$, which is still a Cantor set. 
We may assume that $\{0,1\}\subset A$.  For the sake of simplicity we suppose that  the interval $E$ lies in the interior of $M$. The proof works  in general, with only minor modifications.  
Let $\varphi\in \D^{k,+}(M,C)$ and denote by 
$\xi(t)=\varphi\circ\gamma(t)$. 
Consider a $\mathcal C^k$-coordinates chart $U\subset M$ containing $E$, such that $U$ is 
identified with an open disk, while  $\gamma$ is  now linear and parameterized by arc length, namely that 
$\parallel \dot{\gamma}\parallel=1$ and  $\ddot{\gamma}=0$.  The norm $\parallel \; \parallel$ is 
associated to the standard scalar product $\langle , \rangle$ on $U$ induced from $\R^n$. 

The strategy of the proof is as follows. 
We define a subset $I_{\varepsilon}\subset [0,1]$ consisting of finitely many intervals which contains $A$.  
At first one straightens out $\xi$ in the complementary of  $I_{\varepsilon}$. To this 
purpose we modify $\varphi$  by composing with a convenient compactly supported diffeomorphism. 
Further we show that there is an isotopy rel $A$ which straightens out the remaining arcs of $\xi$.   
Eventually, one proves that a diffeomorphism preserving the orientation of the surface which fixes the arc $E$  
is, up to isotopy, supported outside a disk neighborhood of $E$. This will show that $\varphi$ has a 
compactly supported class. 

Assume for the moment that $A$ is just an infinite set without isolated points.
The set of those $t$ for which $\gamma(t)=\xi(t)$ is a closed 
subset of $[0,1]$ containing $A$ and hence its closure 
$ \overline{A}$. Let now $t_0\in \overline{A}$. 
Then, since $\gamma$ and $\xi$ are differentiable at $t_0$ we have: 
\begin{equation} \dot{\gamma}(t_0)=\lim_{t\in A, t\to t_0}\frac{\gamma(t)-\gamma(t_0)}{t-t_0}=
\lim_{t\in A, t\to t_0}\frac{\xi(t)-\xi(t_0)}{t-t_0} =\dot{\xi}(t_0). 
\end{equation}
If $\varphi$ is twice differentiable then the same argument shows that:
\begin{equation} \ddot{\gamma}(t_0)=\ddot{\xi}(t_0).\end{equation}

Since $\varphi$ is of class $\mathcal C^2$, for every $\varepsilon >0$, 
there exists $\delta(\varepsilon) >0$ such that whenever  $s_1,s_2\in A$, with 
$|s_1-s_2|< \delta(\varepsilon)$ we have: 
\begin{equation}\label{smalldot}  1-\varepsilon <  \langle \dot{\gamma}(t),\dot{\xi}(t)\rangle \leq 1, 
{\rm for \; all\; } t\in[s_1,s_2],\end{equation}

\begin{equation}\label{smallddot}  |\ddot{\xi}(t)| <\varepsilon, 
{\rm for \; all\; } t\in[s_1,s_2].\end{equation}

We assume now that $A=\cap_{j=1}^{\infty}A_j$ 
is the infinite nested intersection of the closed finite unions of intervals 
$A_j\supset A_{j+1}\supset \cdots$.

We  denote $I_{\varepsilon}=
\cup_{s_1,s_2\in A; |s_1-s_2| \leq\frac{\delta(\varepsilon)}{2}}[s_1,s_2]\subset [0,1]$. 
We choose $\varepsilon >0$ small enough such that the image 
of $\xi|_{I_{\varepsilon}}$ is contained within the coordinates disk $U$. 

Set further $\gamma_s(t)= (1-s) \gamma(t)+ s\; \xi(t)$, for $t\in[s_1,s_2]\subset I_{\varepsilon}$ and 
$s\in[0,1]$.

\begin{lemma}\label{local}
Fix $\varepsilon <1$ as above. Let $s_1,s_2\in A$, such that $|s_1-s_2| \leq \delta(\varepsilon)/2$. 
Then $\gamma_s|_{[s_1,s_2]}$ provides a $\mathcal C^k$-isotopy between the 
restrictions $\gamma|_{[s_1,s_2]}$ and $\xi|_{[s_1,s_2]}$ to the interval $[s_1,s_2]$. 
In particular,  the image $\xi(I_{\varepsilon})$ is contained within the union of 
orthogonal strips   $(\gamma(I_{\varepsilon})\times \R)\cap U$.  
\end{lemma}
\begin{proof}
We have to prove that for  any $s\in[0,1]$ the curve $\gamma_s|_{[s_1,s_2]}$ is simple.
This follow immediately from  the fact that whenever $\varepsilon <1$ we have: 
\begin{equation}
\langle \dot{\gamma}_s(t), \dot{\gamma}(t) \rangle \geq 
1-s+s\langle \dot{\xi}(t), \dot{\gamma}(t) \rangle\geq 1-\varepsilon s >0
\end{equation}  
for any $t\in [0,1]$, $s\in [0,1]$. 
Further, note that the curve $\gamma_s(t)$, for $s\in [0,1]$ and fixed $t\in I_{\varepsilon}$ is a segment 
joining $\xi(t)$  with its orthogonal projection  
onto $\gamma(I_{\varepsilon})$.  
\end{proof}

We set 
\begin{equation}
\eta(t)=\left\{\begin{array}{cc}
\xi(t), &  {\rm if }\; t\in I_{\varepsilon};\\
\gamma(t), &  {\rm if }\; t\not\in I_{\varepsilon}.\\
\end{array}\right.
\end{equation}

\begin{lemma}
There exists a compactly supported diffeomorphism $\psi\in P\D^{k,+}(M,C)$ such that 
$\psi(\xi)$ and $\eta$ are isotopic rel $A$. 
\end{lemma}
\begin{proof}
Lemma \ref{local} shows that the image of $\eta$ is a simple curve, as $\xi(I_{\varepsilon})$ is contained within the union of orthogonal strips   $(\gamma(I_{\varepsilon})\times \R)\cap U$, and thus it cannot intersect $\gamma([0,1]\setminus I_{\varepsilon})$. 

Note  that $A\subset I_{\varepsilon}$, since  $A$ has no isolated points. 
The endpoints of a maximal complementary interval should belong to $A$, by maximality. 
In particular, its length should be greater than $\frac{\delta(\varepsilon)}{2}$, and hence there are only finitely many 
maximal complementary intervals say $J_1,J_2,\ldots, J_p$.   
Then $\xi(J_i)$ are pairwise disjoint smooth arcs  whose  interiors  are $\xi({\rm int} (J_i))\subset M-C$, 
each such arc joining two distinct points of  $C$. Moreover, as $\xi(t)=\gamma(t)$, for $t\in \cup_{i=1}^p \partial J_i$, 
we can straighten out  the half-arcs of $\xi$ around these points. Namely,  
there exists a small neighborhood  $N(I_{\varepsilon})$ of $I_{\varepsilon}$ within $[0,1]$ such that after 
perturbing $\xi$ by an isotopy supported in $N(I_{\varepsilon})\cap (\cup_{i=1}^pJ_i)$ 
we have $\xi(t)=\gamma(t)$, for $t\in N(I_{\varepsilon})\cap (\cup_{i=1}^pJ_i)$. 

Now the arcs   $\xi(I_{\varepsilon})$ are disjoint both from $\xi(J_i\setminus N(I_{\varepsilon}))$ and 
$\gamma(J_i\setminus N(I_{\varepsilon}))$. 
There exists then a small enough open neighborhood $V$ of $\xi(I_{\varepsilon})$ within $U$ 
which is disjoint from  both $\xi(J_i\setminus N(I_{\varepsilon}))$ and 
$\gamma(J_i\setminus N(I_{\varepsilon}))$. Therefore there exists an orientation preserving  diffeomorphism 
$\psi$ supported on $M - V$, thus compactly supported,  such that 
$\psi(\xi(J_i\setminus N(I_{\varepsilon})))=\gamma(J_i\setminus N(I_{\varepsilon}))$,  and hence 
$\psi(\xi(J_i))=\eta(J_i)$.  Thus $\psi(\xi)$ and $\eta$ are isotopic rel $A$, as claimed. 
\end{proof}

\begin{lemma}\label{global}
The curves $\gamma$ and $\eta$ are isotopic rel $A$.
\end{lemma}
\begin{proof}
We will prove that the family 
\begin{equation}
\eta_s(t)=\left\{\begin{array}{cc}
\gamma_s(t), &  {\rm if }\; t\in I_{\varepsilon};\\
\gamma(t), &  {\rm if }\; t\not\in I_{\varepsilon}\\
\end{array}\right.
\end{equation}
is the desired isotopy. From Lemma \ref{local} it suffices to show that 
there are not intersections between the segments of curves 
$\gamma_s|_{[s_1,s_2]}$ and $\gamma_s|_{[s_3,s_4]}$, when 
$s_i\in A$ and $[s_1,s_2], [s_3,s_4]\subset I_{\varepsilon}$ are disjoint.

Let $p=\gamma_s|_{[s_1,s_2]}(t_0)$ be a point on the first curve segment. 
We want to estimate the angle $\beta$  of the Euclidean triangle 
with vertices $p, \gamma(s_1), \gamma(s_2)$ at $\gamma_s(s_1)$.  
We can write then: 
\begin{equation}
\langle \gamma_s(t_0)-\gamma_s(s_1), \dot{\gamma}(0)\rangle=
\int_{0}^{t_0-s_1}\langle \dot{\gamma}_s(s_1+x), \dot{\gamma}(0)\rangle \; dx=
\int_{0}^{t_0-s_1}1-s +s\langle \dot{\xi}(s_1+x), \dot{\gamma}(0)\rangle \; dx.
\end{equation}
Then (\ref{smalldot}) implies: 
\begin{equation}\label{angle}
\parallel\gamma_s(t_0)-\gamma_s(s_1)\parallel\cos(\beta)=\langle \gamma_s(t_0)-\gamma_s(s_1), \dot{\gamma}(0)\rangle \geq 
(t_0-s_1)(1-s\varepsilon).
\end{equation}

On the other hand from (\ref{smallddot}) we derive 
\begin{equation}
\parallel\dot{\xi}(x)-\dot{\xi}(s_1)\parallel\leq \varepsilon(x-s_1),
\end{equation}
and then: 
\begin{equation}
\parallel\gamma_s(t_0)-\gamma_s(s_1)\parallel\leq   \int_{0}^{t_0-s_1}\parallel\dot{\gamma}_s(x)\parallel dx \leq \int_{0}^{t_0-s_1}
(s \parallel \dot{\xi}(x)\parallel +(1-s)) dx \leq t_0-s_1+\frac{\varepsilon}{2}(t_0-s_1)^2.
\end{equation}
From (\ref{angle}) we obtain 
\begin{equation}\label{aangle}
\cos(\beta)\geq \frac{1-s\varepsilon}{1+\frac{\varepsilon}{2}(t_0+s_1)}\geq \frac{1-\varepsilon}{1+\varepsilon}.
\end{equation}

If we choose $\varepsilon\leq \frac{1}{3}$ then 
$\beta\in [-\frac{\pi}{3},\frac{\pi}{3}]$.  

Assume now the contrary of our claim, namely that there exists 
some intersection point $p$ between 
$\gamma_s|_{[s_1,s_2]}$ and $\gamma_s|_{[s_3,s_4]}$. 
Up to a symmetry of indices we can assume that the 
Euclidean  triangle with vertices 
at $p$, $\gamma_s(s_1)$ and $\gamma_s(s_2)$ has the angle $\beta$ 
at  $\gamma_s(s_1)$ within the interval $[\frac{\pi}{2}, \pi)$.
This contradicts our estimates  (\ref{aangle}) for $\beta$. 
\end{proof}

The last ingredient of the proof of Theorem \ref{compact} is the following: 

\begin{lemma}
Assume that there exists an isotopy of class $\mathcal C^k$ 
between $\gamma$ and $\eta=\psi(\varphi(\gamma))$ rel  $A$. 
Then $\varphi$ is $\mathcal C^k$-isotopic to a compactly supported 
diffeomorphism  from  $P\D^{k,+}(M,C)$.  
\end{lemma}
\begin{proof}
We can assume that the  disk $D$ of diameter $E$ is contained in $U$. 

If the dimension of $M$ is 2, 
the endpoints of $E$ separate the circle $\partial D$ into two arcs, say 
$F^+$ and $F^-$. The circular order of the three arcs $F^+, E, F^-$  around an endpoint of $E$ 
is preserved by $\psi\circ\varphi\in \D^{k,+}(M,C)$. Note that $E$ is fixed 
by  $\psi\circ \varphi$. Thus there exists a $\mathcal C^k$-isotopy 
which is identity on $E$, sends $\psi(\varphi(F^+))$ to $F^+$ and  $\psi(\varphi(F^-))$ to $F^-$. 
Therefore $\psi\circ \varphi$ is isotopic to a diffeomorphism supported on 
the complement of $D$ and hence its class is compactly supported. The claim follows now, because $\psi$ is
equally  compactly supported.  

If the dimension of $M$ is at least 3, there exists an isotopy of $M$ 
sending $\psi(\varphi(\partial D))$ to $\partial D$, which is identity on $E$, because 
$\psi\circ\varphi\in P\D^{k,+}(M,C)$ and we conclude as above. 
\end{proof}

\subsection{Sparse sets and proofs of Theorems \ref{th:diff-countable}, \ref{V-type} and \ref{products}}

\subsubsection{Preliminaries}
Let $\cN_\epsilon(a)$ denote the $\epsilon$-neighborhood $|x-a|<\epsilon$ 
of $a$ in $\R$, $\cN^\pm_\epsilon(a)$
the punctured right and left semi-neighborhoods of $a$, 
i.e., $a < x<a+\epsilon$ and $a-\epsilon<x < a$, respectively. 

We say that $a\in C$ is a {\it left point} of $C$ if there is a 
left semi-neighborhood 
$\cN^-(a)$ such that $\cN^-(a)\cap C=\varnothing$. 
In the same way we define {\it right points}.

For $a\in C$  denote by $\Diff^k_a$ the stabilizer of $a$ in $\Diff^k(\R,C)$,
and by $\diff^k_a$ the group of $k$-germs of elements of the 
stabilizer of $a$ in $\diff^k(C)$. The superscript $+$ in $\Diff^{k,+}_a$
and  $\diff^{k,+}_a$ means that we only consider those diffeomorphisms 
that preserve the orientation of the interval, i.e. increasing.

Let $\varphi$ be a diffeomorphism with $\varphi(a)=a$. 
We say that $\varphi$ is $N$-{\it flat} at $a$ if: 

\begin{equation}
\varphi(x)- x= o\left((x-a)^N\right), \qquad {\rm as} \quad x\to a. 
\end{equation}

\begin{lemma}
\label{flat}
Assume  that $C$ is a $\sigma$-sparse subset of $\R$.  
Let $\varphi \in \Diff^1_a$ be 1-flat at $a\in C$. Then $\varphi\Bigr|_C$
 is the identity in a small neighborhood of $a$.
\end{lemma}
\begin{proof}
Observe first that $\varphi \in \Diff^{1,+}_a$, since $\varphi'(a)=1$ and hence $\varphi$ must be increasing.  
We can assume without loss of generality 
that $a$ is not a right point of $C$. Suppose that  
$\varphi$ is nontrivial on $\cN^+_\delta(a)\cap C$ for any $\delta>0$.

We first claim that fixed points of $\varphi$ accumulate 
from the right to $a$. Otherwise, there exists some $\delta$ 
such that $\varphi(x)-x$ keeps constant sign for all 
$x\in \cN^+_\delta(a)$. Assume that this sign is positive 
and choose $b\in \cN^+_\delta(a)\cap C$.
Let $(\alpha,\beta)\subset (a,b)$ be a maximal complementary interval of length 
at least $\sigma(b-a)$. By maximality $\alpha\in C$. 
Since $\varphi(\alpha)\in C$ and $\varphi(\alpha) >\alpha$ 
we have $\varphi(\alpha)\geq \beta$, so that:

\begin{equation}\label{meanvalue}
\frac{\varphi(\alpha)-a}{\alpha -a}=\frac{\varphi(\alpha)-\alpha}{\alpha -a} +1
\geq \frac{\beta-\alpha}{\alpha-a}+1\geq \frac{\sigma (b-a)}{\alpha -a}+1
\geq 1+\sigma.
\end{equation}
By the mean value theorem  there exists $\xi\in (a,\alpha)$ such that: 
\[ \varphi'(\xi)=\frac{\varphi(\alpha)-a}{\alpha -a}\geq 1+\sigma.\]
But this inequality contradicts   
the 1-flatness condition for small $\delta$, as taking the limit when 
$\delta\to 0$ we would obtain $\varphi'(a)\geq 1+\sigma$.

When the sign of $\varphi(x)-x$ is negative 
we reach the same conclusion by considering $\varphi(\beta)-\beta$. 
This proves the claim.

Therefore there is a decreasing sequence $u_k$ accumulating at $a$, 
such that $\varphi(u_k)=u_k$. As $\varphi\Bigr|_{C\cap\cN^+_\delta(a)}$ 
is not identity for any $\delta >0$ 
there exists a decreasing sequence $v_k\in C$ accumulating on $a$, such 
that all $\varphi(v_k)-v_k$ are of the same sign, say positive. Therefore, 
up to passing to a subsequence, we obtain a sequence of  
disjoint intervals $(\alpha_j,\beta_j)$ such that $\beta_{j+1}\le \alpha_{j}$, 
$\varphi(\alpha_j)=\alpha_j$, $\varphi(\beta_j)=\beta_j$,
and $v_j\in (\alpha_j,\beta_j)$.

Since $\varphi$ is monotone, it has to be monotone increasing, by above.
Thus $\varphi^{k}(v_j)\in [\alpha_j,\beta_j]$, for any $k\in\Z$, 
where $\varphi^{k}$ denotes the $k$-th iterate of $\varphi$. 
The bi-infinite sequence  $\varphi^{k}(v_j)$ is increasing and so:
\begin{equation}
\alpha_j \leq \lim_{k\to -\infty}\varphi^{k}(v_j) <  
\lim_{k\to \infty}\varphi^{k}(v_j) \leq \beta_j.  
\end{equation}
Now $\lim_{k\to -\infty}\varphi^{k}(v_j)$ and  
$\lim_{k\to \infty}\varphi^{k}(v_j)$ are fixed points of $\varphi$ 
and we can assume, without loss of generality that our choice  
of intervals is such that 
$\alpha_j = \lim_{k\to -\infty}\varphi^{k}(v_j)$,   
$\lim_{k\to \infty}\varphi^{k}(v_j) =\beta_j$. 
In particular $\alpha_j,\beta_j\in C$. 
  
As $C$ is $\sigma$-sparse there is a complementary interval
$(\gamma_j,\delta_j) \subset (\alpha_j,\beta_j)$ 
of length  at least $\sigma (\beta_j-\alpha_j)$.
The interval $(\gamma_j,\delta_j)$ cannot contain any 
point $\varphi^k(v_j)$ and thus there exists some $k_j\in \Z$ 
such that 
\begin{equation}
\varphi^{k_j}(v_j)\leq \gamma_j <\delta_j \leq \varphi^{k_j+1}(v_j).
\end{equation}

Denote $\varphi^{k_j}(v_j)=\eta_j$. We have then 
\begin{equation}
\frac{\varphi(\eta_j)-\varphi(\alpha_j)}{\eta_j-\alpha_j}-1
= \frac{\varphi(\eta_j)-\eta_j}{\eta_j-\alpha_j}\geq \frac{\sigma (\beta_j-\alpha_j)}{\eta_j-\alpha_j}\geq \sigma.
\end{equation}
By the mean value theorem there exists 
$\xi_j\in(\alpha_j,\eta_j)$ such that 
\begin{equation}
\frac{\varphi(\eta_j)-\varphi(\alpha_j)}{\eta_j-\alpha_j}=\varphi'(\xi_j).   
\end{equation}
and thus such that $\varphi'(\xi_j)\geq 1+\sigma$. 
As $\varphi'$ is continuous at $a$, by letting $j$ go to infinity 
we derive $\varphi'(a)\geq 1+\sigma$ which contradicts the 1-flatness. 
\end{proof}

\begin{lemma}
\label{nonflat}
If $C$ is $\sigma$-sparse and 
$\varphi\in \Diff^1_a$ is not 1-flat then 
\begin{equation}
|\varphi'(a)-1| \geq \sigma. 
\end{equation}
\end{lemma}
\begin{proof}
Let  $\varphi\in \Diff^1_a$ not 1-flat, 
so that $\varphi'(a)\neq 1$. Let us further suppose that 
$\varphi'(a) >1$, the other situation being similar. 
For any $\delta>0$ we can choose $b\in \cN^+_\delta(a)\cap C$.
There is then a maximal complementary interval $(\alpha,\beta)\subset (a,b)$ 
of length at least $\sigma (b-a)$.  
By maximality $\alpha\in C$. 

We claim that 
for small enough $\delta$ we have $\varphi(\alpha) >\alpha$. 
Assume the contrary. By the mean value theorem 
there exists $\xi\in(a,\alpha)\subset (a,b)$ such that 
\begin{equation}
\varphi'(\xi)=1+\frac{\varphi(\alpha)-\alpha}{\alpha-a}\leq 1
\end{equation}
and letting $\delta$ go to $0$ we would obtain $\varphi'(a)\leq 1$, 
contradicting our assumptions. 
Thus $\varphi(\alpha) >\alpha$, and hence 
$\varphi(\alpha)\geq \beta$. 
As above, the mean value theorem provides us $\xi\in (a,\alpha)$ 
so that  
\begin{equation}
\varphi'(\xi)=1+\frac{\varphi(\alpha)-\alpha}{\alpha-a}\geq 1+\sigma.
\end{equation} 
Letting $\delta$ go to zero we obtain $\varphi'(a)\geq 1+\sigma$. 
When $\varphi'(a) <1$ we can use similar methods or pass to 
$\varphi^{-1}$ in order to obtain  $\varphi'(a)\leq 1-\sigma$. 
\end{proof}

\begin{lemma}
\label{l:stabilizer}
Let  $C$ sparse and $a\in C$. Then one of the following holds:
\begin{enumerate}
\item  either for any $\varphi\in\Diff^{1,+}_a$, the restriction 
$\varphi\Bigr|_C$ is identity in a small neighborhood  of $a$, so that 
$\diff^{1,+}_a=1$;
\item or else, there is $\psi_a\in \Diff^{1,+}_a$ such that for any 
$\varphi\in \Diff^{1,+}_a$ the restriction of $\varphi$
to a small neighborhood $\cN_\delta(a)\cap C$ coincides
with the iterate $\psi_a^k\Bigr|_C$ for some $k\in\Z\setminus\{0\}$.
Moreover, any such $\psi_a$ is of the form:  
\begin{equation}\label{form}
\psi_a(x)=a+p(x-a)+o(x-a),\qquad {\rm as} \quad x\to a, 
\end{equation}
where $|p-1|\geq \sigma$. Thus $\diff^{1,+}_a=\Z$. 
\end{enumerate}
\end{lemma}
\begin{proof}
If the first alternative doesn't hold, by  Lemma \ref{flat}  we can assume that  
there exists some $\varphi\in \Diff^{1,+}_a$ which is not 1-flat. 

The map $\chi:\Diff^1_a\to \R^*$ given by 
$\chi(\varphi)=\varphi'(a)$ is easily seen to be 
a group homomorphism. By Lemma \ref{nonflat} 
the subgroup $\chi(\Diff^{1,+}_a)$ of $\R_+^*$ 
is discrete and non-trivial and thus it is 
isomorphic to $\Z$. Let $\psi_a\in \Diff^{1,+}_a$ be a germ whose image 
$\chi(\psi_a)$ is a generator of $\chi(\Diff^{1,+}_a)$.
Then $\psi_a$ is not 1-flat and thus, by Lemma \ref{nonflat},  it satisfies the equation 
(\ref{form}). 

If $\varphi\in  \Diff^{1,+}_a$, then we can write 
$\varphi=\psi_a^k\theta$, for some $k\in \Z\setminus\{0\}$ and $\theta\in \ker \chi$. 
But the kernel of $\chi$ consists of those $\theta\in \Diff^{1,+}_a$ 
which are 1-flat. By Lemma \ref{flat}  the restriction
of $\theta$ to some neighborhood $\cN_\delta(a)\cap C$ is identity. This proves that 
$\varphi=\psi_a^k$ in a neighborhood $\cN_\delta(a)\cap C$, as claimed.
\end{proof}

\begin{remark}\label{stab} 
If $C=C_{\lambda}$ is the ternary central Cantor set in $\R$, then 
$\diff^{1,+}_a(C_{\lambda})$ is not always $\Z$. An element 
$a$ of $C_{\lambda}$ is called  $\lambda$-rational if 
it has an eventually periodic development
\[ a=\sum_{i=1}^{\infty}a_i \lambda^i,\]
where $a_i\in \{0,\lambda-1\}$. 
Therefore $\diff^{1,+}_a(C_{\lambda})$ is  $\Z$ if and only if 
$a$ is $\lambda$-rational and trivial, otherwise. 
\end{remark}
\begin{remark}\label{decrease}
Since the subgroup $\chi(\Diff^1_a)\subset \R^*$ is discrete there exists $\lambda\geq 1$ such that 
$\chi(\Diff^1_a)$ is of the form $\langle \lambda\rangle$, $\langle -\lambda\rangle$ or 
$\langle\pm \lambda\rangle$. Here $\langle x\rangle$ denotes the subgroup of $\R^*$ generated by $x$. 
In particular $\diff^1_a$ is isomorphic to  either $1$, $\Z/2\Z$, $\Z$, or else $\Z\oplus \Z/2\Z$. 

However, if $a$ is a left (or right) point of $C$ then there is no decreasing homeomorphism of $(\R,C)$ 
fixing $a$. Thus  $\Diff^1_a=\Diff^{1,+}_a$, and the result of Lemma \ref{form} holds more generally for 
 $\Diff^1_a$. 
\end{remark}

\subsubsection{Proof of Theorem \ref{th:diff-countable}}
We need to show that the identity is an isolated point 
of the group $\diff^1(C)$, if $C$ is $\sigma$-sparse. To this purpose 
consider an element $\diff^1(C)$ having a representative
$\psi\in \Diff^1(\R,C)$ such that 
\begin{equation}
1-\sigma<\psi'(x)<1+\sigma, \; {\rm for \; any } \; x\in C. 
\end{equation}

There is no loss of generality 
in assuming that $\psi\in \Diff^{1,+}(\R,C)$, i.e. that $\psi$ 
is monotone increasing.  The minimal element $\min C$ of 
$C$ should therefore be fixed by any 
element of $\Diff^{1,+}(\R,C)$, in particular by $\psi$. 
By Lemma \ref{nonflat}, $\psi\in \Diff^1_{\min C}(\R,C)$ must be 1-flat at $\min C$.

Consider the set 
\begin{equation}
U=\{x\in C; \psi(z)=z, \; {\rm for \; any} \; z\in C\cap (-\infty,x]\}.
\end{equation}
The set $U$ is nonempty, as $\min C\in U$. 
Let $\xi=\sup U$. 

Assume first that  $\xi$ is not a right point of $C$. 
Since $\psi$ is continuous, $\xi\in U$ so that 
$\psi\in\Diff^1_{\xi}$. From Lemma \ref{nonflat} 
$\psi'(\xi)=1$ and $\psi$ is 1-flat at $\xi$. 
According to Lemma \ref{flat} there is 
some $\delta >0$ such that the restriction 
$\psi\Bigr|_{C\cap \cN^+_{\delta}(\xi)}$ is  identity, 
which contradicts the maximality of $\xi$.  

If $\xi$ is a right point of $C$, then there is some 
maximal complementary interval $(\xi,\eta)\subset \R\setminus C$.
Since $\psi\Bigr|_{C\cap [\min C, \xi]}$ is identity it follows 
that  $\psi( C\cap [\xi, \infty))\subset C\cap [\xi, \infty)$. 
As $\eta$ is the minimal element of $C\cap (\xi, \infty)$ 
it should be a fixed point of  $\psi\Bigr|_{[\xi, \infty)}$ and so $\eta\in U$.   
This contradicts the maximality of $\xi$. 
Hence $\psi$ is identity on $C$.

\begin{remark}\label{trivial}
The same arguments show that if $C\subset [0,1]$ is a sparse Cantor set and 
$\diff^{1,+}_0(C)=1$, then $\diff^{1,+}(C)=1$. 
\end{remark}

For the second claim of the theorem let $V_{\delta}$ be the set of those elements 
in $\diff^1_{S^1}(C)$ having a representative 
$\psi\in\Diff^1(S^1,C)$ such that 
\begin{equation}
1-\delta<\psi'(x)<1+\delta,  \; {\rm for \; any } \; x\in C.
\label{eq:okr}
\end{equation}
Here elements of $\Diff^1(S^1)$ are identified with real 
periodic functions on $\R$.  
We choose  $\delta<\min(\sigma, 0.3)$.
It is enough to prove that $V_{\delta}$ is finite. 

Consider a complementary interval $J\subset S^1-C$ of maximal 
possible length, say $|J|$. Consider its right end $\eta$, with respect to the cyclic orientation. 
If $\psi\in V_{\delta}$ is such that $\psi(\eta)=\eta$, 
then the arguments from the proof of Theorem \ref{th:diff-countable} 
show that $\psi(x)=x$ when $x\in C$.

We claim that the set of intervals of the 
form $\psi(J)$, for $\psi\in V_{\delta}$ is finite.  
Each $\psi(J)$ is a maximal complementary interval, because 
if it were contained in a larger interval $J'$, then 
$\psi^{-1}(J)$ would be a complementary interval strictly larger than 
$J$. This shows that any two such intervals $\psi(J)$ and $\varphi(J)$ 
are either disjoint or they coincide, for otherwise their union 
would contradict their maximality. 
Further, each $\psi(J)$ has length at least $(1-\delta)|J|$. 
This shows that the set of intervals is a finite set 
$\{J_1,J_2,\ldots,J_k\}$. 

Assume that $\psi(J)=\varphi(J)$ and both $\psi$ and $\varphi$ 
preserve the orientation of the circle. If the right end 
of $J$ is $\eta$, with respect to the cyclic orientation, then 
$\varphi\circ\psi^{-1}$ sends $J$ to $J$ and hence fixes $\eta$.  
Then the arguments from the proof of Theorem \ref{th:diff-countable} 
show that $\varphi\circ\psi^{-1}(x)=x$ when $x\in C$.
It follows that there are at most $2k$ elements in $V_{\delta}$, 
finishing the proof of the first part.

\subsubsection{Proof of Theorem \ref{V-type}}\label{section-V-type}
Let $C$ be a Cantor set contained within a  $\mathcal C^1$-embedded simple closed 
curve $L$ on the orientable manifold $M$. For the sake of simplicity 
we will suppose from now on that $M$ is a surface, but 
the proof goes on without essential modifications in higher dimensions. 
Let $\varphi$ be a diffeomorphism of $M$ sending
$C$ into $C$. Fix a parameterization of a 
collar $N$ such that $(N,L)$ 
is identified with $(L\times [-1,1], L\times\{0\})$.  
Denote by $\pi:N\to L$ the projection on the first 
factor and by $h:N\to [-1,1]$ the projection on the second factor.

There exists an open neighborhood $U$ of $C$ in $L$ so that 
$\varphi(U)\subset N$. In particular, the closure $\overline{U}$ is a 
finite union of closed intervals. 
The map $\varphi:\overline{U}\to N=L\times[-1,1]$ has the property 
$h\circ \varphi(a)=0$, for each $a\in C$.  Therefore the differential
$D_a\, (h\circ\varphi)=0$, for each $a\in C$. 
Since $\varphi$ is a diffeomorphism $D_a (\pi\circ\varphi) \neq 0$, for 
every $a\in C$.

For each $a\in C$ consider an open interval neighborhood $U_a$ 
within $L$, so that $D_x (\pi\circ\varphi) \neq 0$ and 
$\parallel D_x\, (h\circ\varphi)\parallel<1$, for every $x\in \overline{U_a}$. 
We obtain an open covering $\{U_a; a\in C\}$ 
of $C$. As $C$ is compact there exists a  
finite subcovering by intervals $\{U_1,U_2,\ldots,U_n\}$. 
Without loss of generality one can suppose that $U_j\subset U$, for all $j$. 
We consider such a covering having the minimal number of elements.
This implies that $\overline{U_j}$ are disjoint intervals. 
 
For every $j$ the map $\pi\Bigr|_{\varphi(\overline{U_j})}:\varphi(\overline{U_j})\to 
\pi(\varphi(\overline{U_j}))\subset L$ is a diffeomorphism on its 
image, since $\varphi(\overline{U_j})$ is connected and 
$D_x (\pi\circ\varphi) \neq 0$, for any $x\in\overline{U_j}$. 
 
Consider a slightly smaller closed interval 
$I_j\subset U_j$ such that $I_j\cap C=U_j\cap C$. 

Let $\mu$ be a positive smooth  function on 
$\sqcup_{j=1}^n\overline{U_j}$ such $\mu(t)$  equals 1 near 
the boundary points and vanishes on $\sqcup_{j=1}^nI_j$.
Define $\phi_s:\sqcup_{j=1}^n\overline{U_j}\to N$ by:
\begin{equation}
\phi_s(x)=(\pi\circ\varphi(x), (s\mu(x)+1-s)\cdot h\circ\varphi (x)). 
\end{equation}
Then $\phi_0(x)=\varphi(x)$ and for each $s\in [0,1]$ we have 
$\phi_s(x)=\varphi(x)$, for $x$ near 
the boundary points of $\sqcup_{j=1}^n\overline{U_j}$. Furthermore   
$\phi_1(x)=\pi\circ\varphi(x)\in L$, when $x\in \sqcup_{j=1}^nI_j$.
One should also notice that $\phi_s(x)=\varphi(x)$, for each $x\in C$ and $s\in [0,1]$. 

Let now denote $J_j=\pi\circ\varphi(I_j)$. 
It is clear that $C=\varphi(C)\subset \cup_{j=1}^nJ_j$. 
We claim that we can assume that $J_j$ are disjoint. 
Indeed, since $\varphi$ is bijective we have 
$\varphi(I_j\cap C)\cap \varphi(I_k\cap C)=\emptyset$, for any $j\neq k$. 
Since $\varphi(I_j\cap C)$ are closed subsets of $L$ there exists
$\varepsilon >0$ so that 
$d(\varphi(I_j\cap C), \varphi(I_k\cap C))\geq \varepsilon$, for 
$j\neq k$, where $d$ is a metric on $L$. 
Since $\phi_1(I_j\cap C)=\varphi(I_j\cap C)$, we have 
$d(\phi_1(I_j\cap C), \phi_1(I_k\cap C))\geq \varepsilon$, for 
$j\neq k$. Thus there exist some open neighborhoods 
$J_j'$ of $\phi_1(I_j\cap C)$ within $L$ so that 
$d(J_j', J_k')\geq \frac{1}{2}\varepsilon$, for all 
$j\neq k$. As $\phi_1$ is a diffeomorphism there exist
open neighborhoods $I_j'$ of $I_j\cap C$ with the property that 
$\phi_1(I_j')\subset J_j'$, for all $j$. 
Now $I_j'$ and $J_j'$ are finite unions of open intervals. 
We can replace them by closed intervals with the same intersection 
with $C$. This produces two new families of disjoint 
closed intervals related by $\phi_1$, as the initial situation.   
This proves the claim.

We obtained that there exist two coverings 
$\{I_1,I_2,\ldots,I_n\}$  and  $\{J_1,J_2,\ldots,J_n\}$ of $C$ by 
disjoint closed intervals and a diffeomorphism
$\phi_1: \sqcup_{j=1}^nI_j\to \sqcup_{j=1}^nJ_j$ 
such that $\phi_1(x)=\varphi(x)$, for any $x\in C$.

Notice  that the sign of $D_a(\pi\circ\varphi)$ 
might not be the same for all intervals.

Every partition of $C$ induced by a covering $\{I_1,I_2,\ldots,I_n\}$
as above is determined by the choice of complementary intervals, namely 
the $n-1$ connected components of  $L\setminus \cup_{j=1}^nI_j$. 
It follows that there are only countably many finite partitions 
of $C$ of the type considered here. 
Next, the set of those elements of $\diff^1_{M}(C)$ which arise from 
partitions induced by the coverings $\{I_1,I_2,\ldots,I_n\}$  and  $\{J_1,J_2,\ldots,J_n\}$ of $C$ is acted upon transitively by the stabilizer of 
one partition. The stabilizer of one partition embeds 
into the product  of $\diff^1_{I_j}(C\cap I_j)$. 
Theorem \ref{th:diff-countable} then  implies that 
$\diff^1_M(C)$ is countable.

\subsubsection{Proof of Theorem \ref{products}}

Before to proceed we need some preparatory material. 
Let $A\subset \R^n$ be a set without isolated points. 
Let $T_p\R^n$ denote the tangent space at $p$  on $\R^n$ and 
$UT_p\R^n\subset T_p\R^n$ the sphere of unit vectors.  
For any $p\in A$ one defines  the {\em unit tangent spread} $UT_pA\subset UT_p\R^n$ at $p$ 
as the set of vectors $v\in UT_p\R^n$ for which there exists 
a sequence of points $a_i\in A$ with 
$\lim_{i\to \infty}a_i=p$ and 
\[\lim_{i\to \infty}\frac{a_i-p}{\parallel a_i-p\parallel}=v.\] 
Vectors in $UT_pA$ will also be called {\em (unit) tangent vectors} at $p$ to $A$. 
We also set $T_pA=\R_+\cdot UT_pA\subset T_p\R^n$ for the {\em tangent spread} at $p$.

A differentiable map $\varphi:(\R^n, A)\to (\R^n,B)$   induces a tangent map 
$T_p\varphi: T_pA\to T_{\varphi(p)}B$. 
Specifically, let  $D_p\varphi:T_p\R^n\to T_{\varphi(p)}\R^n$ 
be the differential of $\varphi$; then we have 
\[ T_p\varphi= U(D_p\varphi),\]
where for a linear map $L:V\to W$ between vector spaces we denoted by 
$U(L):U(V)\to U(W)$ the map induced on  the unit spheres, namely 
\[ U(L)v= \frac{L(v)}{\parallel L(v)\parallel}.\]

As the unit tangent spread $UT_pA$ is a subset of the unit sphere, it inherits 
the spherical geometry and metric. In particular, it makes sense to consider 
the convex hull $Hull(UT_pA)\subset UT_p\R^n$ in the sphere. 

Although tangent spreads to product Cantor sets might depend on the 
particular factors, their convex hulls have a simple description. 
Let $C=C_1\times C_2\times \cdots \times C_n\subset \R^n$ be a product of Cantor sets 
$C_i\subset \R$.  The  usual cubical complex underlying the  
$n$-dimensional cube $[0,1]^n$ will be denoted by  $\Box^n$. 
Let then denote by $Lk(p)$ the spherical link  of $p\in \Box^n$. 
If $p$ belongs to a $k$-dimensional cube but not to a $k+1$-dimensional cube 
of $\Box^n$ then $Lk(p)$ is isometric to the link  $\mathcal L_{k,n}$ of the origin in $\R^k\times \R_+^{n-k}$. 
Thus there are precisely $n+1$ different isometry types of links of points.  

Now a direct inspection shows that  for each $p\in C$ there exists some $k$ so that 
the convex hull $Hull(UT_pA)$ is isometric to $\mathcal L_{k,n}$.

When the diffeomorphism  $\varphi: (\R^n,C)\to (\R^n,C)$ is also 
{\em conformal}, then the tangent maps are isometries 
between the unit tangent spreads,  
because the spherical distance is given by angles between 
the corresponding vectors. However this is not true for general 
diffeomorphisms. 

Nevertheless the spherical links 
$\mathcal L_{k,n}$ are quite particular. There exist $n+k$
vectors along the coordinates axes which are extremal points 
of $UT_pC$, such that  their convex hull is $Hull(UT_pC)$, so isometric to  
$\mathcal L_{k,n}$. These are vectors of the form $e_i, -e_i, e_j$, where 
$e_i$ correspond to the coordinates axes in  $\R^k$ and $e_j$ 
to those in $\R^{n-k}$. Now, any diffeomorphism  $\varphi: (\R^n,C)\to (\R^n,C)$ should send an unit tangent spread of type $\mathcal L_{k,n}$ into 
one of the same type, since  $\mathcal L_{k,n}$ is not affinely equivalent 
to $\mathcal L_{k',n}$, for $k\neq k'$. 
Moreover, the extremal vectors are sent into extremal vectors 
of the same type.

Let further $\varphi\in \Diff^1(\R^n,C)$ such that 
$\| D_a\varphi-\mathbf 1\| \leq \varepsilon$ for all $a\in C$.  
Assume now that the unit tangent spread 
$UT_aC$ is isometric to  $\mathcal L_{0,n}$,
namely it is of {\em corner} type. In this case 
$U(D_a\varphi)$ should permute the 
$n$  coordinate vectors, which are the extremal 
vectors of $\mathcal L_{0,n}$. Therefore either 
$U(D_a\varphi)=\mathbf 1$, or else 
\[  \| U(D_a\varphi)-\mathbf 1\| \geq \sqrt{2},\]
which yields 
\[  \| D_a\varphi-\mathbf 1\| \geq \sqrt{2}.\]
In other words, taking $\varepsilon < \sqrt{2}$ any diffeomorphism 
$\varphi$ as above should satisfy $U(D_a\varphi)=\mathbf 1$.
Now, if $\varphi$ is of class $\mathcal C^1$ then 
$U(D_a\varphi)$ is continuous. Since the set of corner 
points is dense in $C$ we derive  $U(D_a\varphi)=\mathbf 1$, for any $a\in C$.  
This is the same as saying that for any $a\in C$ the linear map 
$D_a\varphi$ is represented by  a diagonal matrix, with respect to the standard 
coordinate system of $\R^n$.

\begin{proposition}\label{stabilizers}
Let $a\in C$ be a corner point. The map 
$\chi: \diff^1_{\R^n, a}(C)\to \left(\R^*\right)^n$, which associates to the germ $\varphi$ the 
 eigenvalues of $D_a\varphi$ is an isomorphism onto a discrete subgroup 
of $ \left(\R^*\right)^n$. 
\end{proposition}
\begin{proof}
Let $\{x_1,x_2,\ldots,x_n\}$ be the standard coordinates functions  on $\R^n$ and 
$\pi_j:\R^n\to\R^{n-1}$ denote the projection onto the hyperplane $H_j=\{x_j=0\}$.  
For the sake of simplicity we assume that $a=(0,0,\ldots,0)$, and that the 
convex hull of the unit tangent spread is the union   
of the sets  $H_j^+=H_j\cap \{x_i\geq 0, i=1,\ldots,n\}$. 
We will use induction on $n$. The claim was proved in Lemma \ref{l:stabilizer} for $n=1$. 
Assume it holds for all dimensions at most $n-1$. 

Let $\varphi\in \Diff^1(\R^n,C)$ such that $\varphi(a)=a$. 
Assume that $ \| D_x\varphi-\mathbf 1\| < \frac{1}{2}\sigma <\frac{1}{2}$  for all $x$ in a neighborhood $V$ of $a$ in $\R^n$. 
We will prove that $\varphi\Bigr|_{C}$ is  a trivial germ at $a$. This shows that the image of $\chi$ is a 
discrete subgroup of $\left(\R^*\right)^n$ and the kernel of $\chi$ is trivial.

Consider the maps $\varphi_j: H_j\to H_j$ given by 
$\varphi_j(x)=\pi_j\circ\varphi(x)$. 
The determinant of $D_a\varphi_j$ is the product of all eigenvalues of $D_a\varphi$ but 
the $j$-th eigenvalue, and hence it is non-zero. Moreover, we have 
  $ \| D_a\varphi_j-\mathbf 1\| <\frac{1}{2}\sigma$.  
  
  We claim that  
\begin{lemma}\label{injectivity}
The map $\varphi_j:H_j\cap V\to H_j$ is injective. 
\end{lemma}
\begin{proof}
  Assume the contrary, namely that there exist two points $p,q\in H_j\cap V$ such that 
  $\pi_j(\varphi(p))=\pi_j(\varphi(q))$.   
  Consider the first non-trivial case $n=2$, when $H_j^+$ are half-lines issued from $a$. 
The mean value theorem and the previous equality prove that there exists 
some $\xi\in H_j^+\cap V$   between $p$ and $q$ so that 
$(\pi_j\circ\varphi)'(\xi)=0$. This amounts to the fact that the image of $D_\xi \varphi$ is contained in the 
kernel of $D_{\varphi(\xi)}\pi_j$, namely that 
\[ \langle D_{\xi} \varphi (v_j), v_j\rangle =0, \]
where $v_j$ is a unit tangent vector to $H_j^+$ at $\xi$. 
We derive $ \| D_{\xi}\varphi_j-\mathbf 1\| \geq 1$, contradicting our assumptions. 

  In the general case $n>2$ we will use a trick to reduce  ourselves to $n=2$, 
  because we lack a multidimensional 
  mean value theorem. Let $P$ a generic affine 2-dimensional half-plane 
  whose boundary line passes through $p$ and $q$.
 We can find arbitrarily small 
$\mathcal C^1$-isotopy deformations  $\psi$ of $\varphi_j$  so that $\psi(H_j)$ is transversal to $P$ and 
  $ \| D_a\psi-\mathbf 1\| < \sigma$.  It follows that 
  $\psi(H_j)\cap P$ is a 1-dimensional manifold $Z$  with boundary containing both $p$ and $q$.
  Now either  there exist two distinct points of   the boundary $\partial Z$ 
  joined by an arc within $Z$, or else there is an arc of $Z$ issued from $p$ which returns to $p$, 
  contradicting the transversality of the intersection $\psi(H_j)\cap P$. 
  In any case  the mean value argument above  shows 
  that  it should exist a point $\psi(\xi)$ of $Z$ for which the tangent vector $v$ is orthogonal  to $H_j$. 
We can write $v= D_{\xi} \psi (w)$, for some  tangent vector $w\in H_j$ at $\xi$. It follows that  
  \[ \langle D_{\xi} \psi (w), w\rangle =0, \]
 which implies 
    $ \| D_{\xi}\psi-\mathbf 1\| \geq 1$, contradicting our assumptions. 
    
\end{proof}
It follows that $\varphi_j:H_j\cap V\to H_j$ is an injective map of maximal rank in a neighborhood $V$ of $a$, and hence a diffeomorphism on its image. The projection  $\pi_j$ sends $C$ into $C\cap H_j$, so that 
\[ \varphi_j(C\cap H_j\cap V)\subset C\cap \varphi(H_j\cap V)\subset C\cap H_j.\] 

Our aim is to use the induction hypothesis for $\varphi_j$. In order to do that we need to show that 
the class of $\varphi_j$ defines indeed an element of $\diff^1_{\R^{n-1},a}(C)$, where 
we identified $H_j$ with $\R^{n-1}$.

We assume  from now on that the neighborhood $V$ is a parallelepiped, all whose vertices being corner points.
Its boundary $\partial V$ will then consist of the union of the faces $V_j=\partial V\cap H_j^+$ with their respective 
parallel faces $V_j'$. The parallelepiped $V$ is surrounded by gaps, whose smaller width is some $\delta>0$. 
Let $V^{\delta}$ be the $\delta$-neighborhood of $V$.   
If $\varphi$ is Lipschitz with Lipschitz constant $1+\varepsilon$ and 
\[ (1+\varepsilon) l_i < \delta +l_i\] 
where $l_i$ are the edges lengths of $V$ then the image $\varphi(V)$ is contained in $V^{\delta}$, 
so that $\varphi_j(V_j)\subset V^{\delta}\cap H_j$. 

Further $\varphi_j(\partial V_j)$ bounds $\varphi_j(V_j)$ and thus there are no points of $C\cap H_j$ 
accumulating on  $\varphi_j(V_j)$, as their unit tangent spread cannot be of the type $\mathcal L_{n-1,n-1}$. 
Thus $C-\varphi_j(V_j)$ is a closed subset of $C$ and hence its distance to $\varphi_j(V_j)$ is 
strictly positive. There exists then an open set 
$U\subset V^{\delta}$ which contains $\varphi_j(V_j)$ such that  
$U\cap (C-\varphi_j(V_j))=\emptyset$. It follows that there exists 
an extension of $\varphi_j$ to a diffeomorphism $\Phi_j$ of $(H_j,C)$  which is identity outside $U$, 
and hence on $(V_{\delta}\cap H_j)\cup (C- \varphi(V_j))$.
 
It only remains to check that $\Phi_j^{-1}(C)$ is also contained in $C$, as needed 
for  $\Phi_j\in \Diff^1(\R^{n-1},C)$.  This follows from the following:

\begin{lemma}\label{surj}
The map $\varphi_j$ has the property   
\[ \varphi_j(C\cap V_j)= C\cap \varphi(V_j).\] 
\end{lemma}
\begin{proof}
Assume that there exists some point $p$ in $\varphi(V_j)\cap C$ which does not belong to $V_j$. 
Then the line issued from $p$ which is orthogonal to $V_j$ intersects 
$\varphi(V_j)$ only once, from Lemma \ref{injectivity}. On the other hand there are points of $C$ 
on this line, as $C$ is a product and $p\not\in V_j$. By Jordan's theorem there exist 
points of $C$ which belong to different components of $R^n-\varphi(\partial V)$ which contradicts 
the fact that $\varphi$ is surjective on $C$.  

Thus $\varphi(C\cap V_j)\subset C\cap V_j$. The same argument for $\varphi^{-1}$ yields the opposite inclusion and hence 
$\varphi(C\cap V_j)=C\cap V_j$. 
Our claim follows. 
\end{proof}

Lemma \ref{surj} tells us that $\varphi_j$ defines a germ in $\diff^1_{H_j,a}(C\cap H_j)$, namely both 
$\varphi_j$ and $\varphi_j^{-1}$ sends $C\cap H_j$ into itself. 
By the induction hypothesis $\varphi_j\Bigr|_{C\cap H_j}$ must be identity in a neighborhood of $a$ within $H_j$.

Notice that this implies already that $D_a\varphi=\mathbf 1$, and hence establishing the first claim of Proposition \ref{stabilizers}. 

For the second claim we consider the distance $d(C-V, V)=\mu >0$, as $V$ is surrounded by gaps. We suppose further that 

\[ \|D_x\varphi -\mathbf 1\| < \min\left(\frac{\sigma}{2}, \frac{\mu}{1+\sigma}\right).\]

We know that $\varphi(y,0)=(y, u(y))$, for $y\in C\cap V\cap H_n$ and some function 
$u\geq 0$. The next step is to show that $u\Bigr|_{C\cap V\cap H_n}=0$. 

Assume that there exists some $x\in C\cap V\cap H_n$ so that $u(x) >0$. 
Observe that $u(x)\in C_n$, since $\varphi(C)\subset C$.  Since points of $C_n$ which are 
not endpoints are dense in $C_n$ there should exist 
$x\in C$ for which $u(x)$ is not an endpoint of $C_n$. Set $z=(x,u(x))\in C$. 

Then for each $\nu>0$ there exist points $z_+, z_-\in C$ with $\pi_n(z_+)=\pi_n(z_-)=x$, 
so that the distances $d(z_+,z),d(z_-,z) <\nu$. 

Observe that the segment $z_+z_-$ intersects just once $\varphi(H_n^+)$, namely at $z$. 
One might expect to use Jordan's theorem in order to derive that $z_+\in C$ and $z_-\in C$ could not belong to the same 
connected component of $\varphi(\partial V)$.  This is not exactly true, as the segment $z_+z_-$ could possibly intersect 
other sheets like $\varphi(H_j^+)$ which are part of $\varphi(\partial V)$. 

Set $r$ for the distance between $x\in H_n^+$ and the union of the other $2n-1$ faces 
of $\partial V$. By the induction hypothesis 
we can assume that $r>0$.  Choose now $\nu$ so that 
$\nu < \min((1-\sigma)r/2, \mu(1-\sigma)/2)$. 

Suppose that there exist $x_+,x_-\in C\cap V$ such that $\varphi(x_+)=z_+$ and $\varphi(x_-)=z_-$. 
By Jordan's theorem the segment $z_+z_-$ intersects at least once 
$\varphi(\partial V-H_n^+)$, say in a point $\tilde{z}=\varphi(\tilde{x})$. 

We have then $ d(x, \tilde{x}) \geq r$, while 
\[ d(\varphi(x),\varphi(\tilde{x}))\leq d(z_+,z_-)\leq 2\nu. \]

On the other hand the $\mathcal C^1$-diffeomorphism $\varphi^{-1}$ is Lipschitz with 
Lipschitz constant bounded by $\sup_{x\in V}\|D_x\varphi^{-1}\| $.
Now, by standard functional calculus we have: 
\[ \|D_x\varphi^{-1}\| \leq \sum_{k=0}^\infty \|\mathbf 1- D_x\varphi\|^k<\frac{1}{1-\sigma}.\]
 Therefore the Lipschitz constant  of $\varphi^{-1}$ is bounded by above by 
 $\frac{1}{1-\sigma}$ so that 
 \[ d(x, \tilde{x}) \leq \frac{1}{1-\sigma} d(\varphi(x),\varphi(\tilde{x}))\leq \frac{2\nu}{1-\sigma}.\]
 This contradicts our choice of $\nu$. 
 
Furthermore if  one of $x_+,x_-$, say $x_+$ belongs to $C-V$ then we have 
$d(x,x_+)\geq \mu$  while 
\[ d(\varphi(x),\varphi(x_+))\leq \nu \]
and the argument above still leads to a contradiction.

This shows that $\varphi$ cannot be surjective on $C$. 
On the other hand a diffeomorphism of $\R^n$ which preserves $C$  
restricts to a bijection on $C$. If it were not surjective then its inverse would send 
points of $C$ outside.  

In particular $u(x)\Bigr|_{C\cap H_n^+}=0$ and so 
$\varphi\Bigr|_{C\cap H_n^+}$ is identity. 
The same proof shows that $\varphi\Bigr|_{C\cap H_j^+}$ is identity, for all $j$. 

By using the same argument when $a$  runs over the  points of $V\cap C\cap \cup_{j=1}^n H_j^+$ we derive that 
$\varphi\Bigr|_{C\cap V}$ is identity, as claimed. 
\end{proof}

{\em End of the proof of theorem \ref{products}}. 
The proof is by induction on $n$. For $n=1$ this was already proved above. 
Let  $V$ denote now the smallest parallelepiped containing $C$, in order to match previous notations and constructions.  
Suppose that $\varphi\in \Diff^1(\R^n,C)$ is such that 
$\|D_x\varphi -\mathbf 1\| < \varepsilon$, for all $x\in V^{\delta}$.
Then $\varphi(\partial V)$ surrounds $C$ and the proof of Lemma \ref{surj} gives us 
$\varphi(C\cap \partial V)=C\cap \partial V$. Moreover, each $\varphi_j$ preserves 
the associated face $V_j$. By the induction hypothesis  $\varphi_j$ is the identity. It follows that 
$\varphi\Bigr|_{C\cap \partial V}$ is the identity. We can therefore use Proposition \ref{stabilizers} 
to derive that around every corner point of $C\cap \partial V$ the map $\varphi\Bigr|_{C}$ is identity.
The same argument works for all corner points of $V$.

\begin{remark}\label{stab2}
If $C=C_{\lambda}^n$, then 
$\diff^{1,+}_a(C_{\lambda})$ 
is isomorphic to $\Z^{r(a)}$, where 
$r(a)$ is the number of coordinates of $a$ which are 
$\lambda$-rational (compare with \cite{BL}). 
\end{remark}

\section{Diffeomorphisms groups of specific Cantor sets}

\subsection{Proof of Theorem \ref{thompIFS}}\label{thompIFSsec}
Observe first that $C_{\Phi}$ is a Cantor set. 
Indeed the contractivity assumption implies that 
an infinite intersection  
$\lim_{p\to \infty}\phi_{i_1}\phi_{i_2}\cdots\phi_{i_p}(M)$ cannot contain 
but a single point.  Two such points which are distinct are separated 
by some smoothly embedded sphere, which is the image of 
$\partial M$ by an element of the semigroup generated by $\Phi$, so that the set $C_{\Phi}$ is totally disconnected. The perfectness follows 
the same way.

We will draw a rooted $(n+2)$-valent tree $\mathcal T$  
with edges directed downwards. 
When $M=[0,1]$ there is an extra structure on  $\mathcal T$, as all edges 
issued from a vertex are enumerated from left to the right. 

There is a  one-to-one correspondence between 
the points of the boundary at infinity of the tree and the points 
of the Cantor set $C=C_{\Phi}$ associated to the  
invertible IFS $(\Phi, M)$.
To  each point  $\xi\in C$ we can assign an infinite 
sequence $I=i_1 i_2\ldots i_p\ldots$, so that $\xi=\xi(I)$ where 
we denoted:
\[
\xi(I)=\bigcap_{p=1}^{\infty} \phi_{i_1}\phi_{i_2}\cdots\phi_{i_p}(M).
\]
The vertices of the tree are endowed with a compatible labeling by means 
of finite multi-indices $I$, where the root has associated the empty index
and the vertex $v_I$ is the one reached after traveling along the 
edges labeled $i_1,i_2,\ldots,i_p$. We also put 
\[ \phi_{I}(x)=\phi_{i_1}\phi_{i_2}\cdots\phi_{i_p}(x)\]
for finite $I$. This extends obviously to the case of infinite 
multi-indices $I$.

We further need to introduce a 
special class of germs, as follows:

\begin{definition}\label{germ}
The standard germ associated to the finite multi-indices 
$I$ and $J$ is the diffeomorphism 
$\phi_{I/J}: \phi_I(M)\to \phi_J(M)$ given by 
\begin{equation}
\phi_{I/J}(\phi_I(x))=\phi_J(x). 
\end{equation}
\end{definition}
Standard germs preserve the Cantor set $C$ as germs, namely 
$\phi_{I/J}(C\cap \phi_I(M))\subset C\cap  \phi_J(M)$. 
In fact if $S$ is an infinite multi-index then 
\[ \phi_{I/J}(\xi(IS))=\xi(JS).\]
Graphically we can realize this map as a partial  
isomorphism of the tree $\mathcal T$  which maps 
the subtree hanging at $v_I$ onto the subtree hanging at $v_J$. 

Consider a pair $(t_1,t_2)$ of finite labeled subtrees 
of the same degree of $\mathcal T$ both 
containing the root, and whose leaves are enumerated 
$v_{I_1},v_{I_2},\ldots, v_{I_p}$ and 
$v_{J_1},v_{J_2},\ldots, v_{J_p}$. 

\begin{lemma}
Assume that $\phi_j$ are orientation preserving diffeomorphisms of $M$. Then the map 
\begin{equation}
\phi(x)= \phi_{I_k/J_k}(x), \quad {\rm if }\quad x\in \phi_{I_k}(C)
\end{equation}
defines an element $\phi_{(t_1,t_2)}\in \diff^{1,+}(C)$. 
\end{lemma}
\begin{proof}
We know that $C=\cup_{i=0}^n \phi_i(C)$, since $C$ is the attractor of $\Phi$. By recurrence on the number of 
leaves we show that 
\[ C= \cup_{i=0}^n \phi_{I_i}(C)\]
for any finite subtree $t$ of $\mathcal T$ containing the root and having 
leaves $v_{I_i}$, $i=0,n$.
Now  $\phi$ is a smooth orientation preserving map 
defined on $\cup_{i=0}^n \phi_{I_i}(M)$, and so its domain 
of definition contains $C$. 
 
When the dimension $d=1$, the complementary 
$M\setminus\cup_{i=0}^n \phi_{I_i}(M)$ is the union of 
finitely many intervals, which we call gaps and there exists 
by orientability assumption an extension of $\phi$ to  a diffeomorphism of 
$M=[0,1]$ sending gaps into gaps. 

When the dimension $d >1$, the complementary gap 
$M\setminus\cup_{i=0}^n \phi_{I_i}(M)$ is now connected and diffeomorphic 
to the standard disk with $(n+1)$ holes. Moreover, the restriction of $\phi$ to every sphere $\partial  \phi_{I_i}(M)$ is isotopic to identity since it is orientation preserving and it admits an extension to the ball. 
Taking a suitable smoothing at the singular vertex of the conical 
extension of $\phi\Bigr|_{\cup_{i=0}^n \phi_{I_i}(\partial M)}$ we obtain 
an extension of $\phi$ to a  diffeomorphism of  the ball $M$, 
possibly non-trivial on $\partial M$. 

This extension preserves $C$ invariant as gaps 
are disjoint from $C$ and therefore defines an element 
$\phi_{(t_1,t_2)}\in \diff^{1,+}(C)$. 
\end{proof}

{\em End of the proof of Theorem  \ref{thompIFS}}. 
Let us stabilize the pair of trees $(t_1,t_2)$ to a 
pair $(t_1',t_2')$, where $t_j'$ is obtained from $t_j$ by 
adding the first descendants at vertex $v_{I_s}$, for $j=1$ and 
$v_{J_s}$, when $j=2$. The new vertices come with a compatible 
labeling. Moreover, an orientation preserving diffeomorphism of $C$ 
induces a monotone map of the boundary of the tree, when $d=1$.   

By direct inspection using the explicit form of $\phi$ we find that: 
\[ \phi_{(t_1,t_2)}=\phi_{(t_1',t_2')}.\]
Thus the map which associates to the pair $(t_1,t_2)$ of labeled trees 
the element $\phi_{(t_1,t_2)}$ factors through a map  
$F_{n+1}\to \diff^{1,+}(C)$, for $d=1$, and 
$V_{n+1}\to \diff^{1,+}(C)$, for $d=2$, respectively. 
This is easily seen to be a homomorphism. 
When $I\neq J$ the map $\varphi_{I/J}\Bigr|_{C}$ is 
not identity since $\varphi_I(M)\cap \varphi_J(M)=\emptyset$. 
This proves that the homomorphism  defined above is injective, 
 thereby ending the proof of Theorem \ref{thompIFS}.

\begin{remark}
There is a more general setting in which we allow basins to have 
boundary fixed points.  We say that the compact 
submanifold $M$ is an {\em attractive basin} for $\Phi=(\phi_0,\phi_1,\ldots,\phi_n)$ if, 
for all $j\in\{0,1,\ldots,n\}$ we have: 
\begin{enumerate}
\item $\phi_j({\rm int}(M))\subset {\rm int}(M)$; 
\item ${\rm int}(\phi_j^{-1}(\phi_j(\partial M)\cap \partial M))\supset  
{\rm int}(\phi_j(\partial M)\cap \partial M)$; 
\item $\phi_i(M)\cap \phi_j(M)=\emptyset$, for any $i\neq j\in\{0,1,\ldots,n\}$; 
\item ${\rm int}(\phi_j(\partial M)\cap \partial M)\subset 
{\rm int}(\phi_j^{-1}(\phi_j(\partial M)\cap \partial M))$.
\end{enumerate}
Using similar arguments one can show that $\diff^{1}(C_{\Phi})$ contains $F_{n+1}$ whenever $\Phi$ has an attractive basin. 
\end{remark}

\begin{remark}
If the Cantor set $C$ is invertible, namely there exists an orientation reversing diffeomorphism $\phi$  
of $M$ preserving $C$, then we can replace the homeomorphisms $\phi_j$ which reverse the orientation by $\phi\circ\phi_j$.  However, there exist non invertible Cantor subsets, for instance the union of two copies $C_{\lambda}\cup (1+C_{\mu})$, 
for $\lambda\neq \mu$. 
\end{remark}

\subsection{Proof of Theorem \ref{thompgen} for $C=C_{\lambda}$}
Our strategy is to give first a detailed proof of Theorem \ref{thompgen} in the case when $C=C_{\lambda}$ and then to explain the necessary changes needed to achieve the general case in the next section. 
 
We first need the following:
\begin{lemma}\label{chi}
If $a$ is a left (or right) point of $C_{\lambda}$, then $\chi(\Diff^{1}_a)$ is the subgroup $\langle \lambda \rangle\subset \R^*$. 
\end{lemma}
\begin{proof}
Recall from Remark \ref{decrease} that $\Diff^{1}_a(C_{\lambda})=\Diff^{1,+}_a(C_{\lambda})$.  The set $L(C_{\lambda})$  of left points of $C_{\lambda}$ is affinely locally homogeneous, namely for any two left points $a$ and $b$ there exists an affine germ sending a neighborhood of $a$ in $C_{\lambda}$  into a neighborhood of $b$ in $C_{\lambda}$.  
Therefore it suffices to analyze $\Diff^{1,+}_0(C_{\lambda})$.  Moreover, $0$ is the minimal element of $C_{\lambda}$ and 
therefore it should be fixed by any element of $\Diff^{1,+}(C_{\lambda})$. 

Elements of $L(C_{\lambda})$ can be described explicitly, as: 
\begin{equation}
L(C_{\lambda})=\bigcup_{n=1}^{\infty} \{x\in [0,1]; x=\sum_{j=1}^n a_j\lambda^{-j}, \quad 
{\rm where} \quad a_j\in\{0,\lambda-1\}\}.
\end{equation}

Therefore there exists $\delta$ such that the 
multiplication by $\lambda\in\R^*$ 
sends $C_{\lambda}\cap \cN_{\delta}(0)$ into $C_{\lambda}$. 
This easily implies that $\chi(\Diff^{1,+}_a)$ contains 
the subgroup $\langle \lambda \rangle$. 

For the reverse inclusion we need a sharpening of Lemma \ref{nonflat}. 
Note first that the set of lengths of gaps 
in $C_{\lambda}$ is $\{(\lambda-2)\lambda^{-n}, n\in \Z_+\setminus\{0\}\}$. 
In particular, the quotients of the lengths of any two gaps 
belong to $\langle \lambda \rangle$.

Let $\alpha >1$ be a minimal element occurring in $\chi(\D^{1,+}_0(C_{\lambda}))\subset \R_+^*$. 
By Lemma \ref{l:stabilizer} there exists $k\in\Z_+$ such that $\lambda^{-1}=\alpha^k$.  Let 
$\varphi\in \Diff^{1,+}_0$ be such that $\varphi'(0)=\lambda^{-1/k}$. 

For every gap $I$, the image $J=\varphi(I)$ is another gap and the ratio $|J|/|I|$ is an element in $\langle \lambda\rangle$, 
hence of the form $\lambda^{i(I)}$, for some integer $i(I)$. Further, there is a point $x_I\in I$ for which $\varphi'(x_I)=\lambda^{i(I)}$. 
Letting $I_n$ be a sequence of gaps converging to the origin, we have that 
$\varphi'(x_{I_n})$ converges to $\varphi'(0)=\lambda^{-1/k}$. 
The sequence of integers $i(I_n)$ hence converges to $-\frac{1}{k}$, which forces $k=1$. 
\end{proof}

We next observe that for each left point $a$ of $C_{\lambda}$ there exists 
a small neighborhood  $U_a$ of $a$ such that the affine map 
$\psi_a=a+\lambda(x-a)$ sends $U_a\cap C_{\lambda}$ into $C_{\lambda}$, defining 
therefore a germ in $\diff^{1,+}_a$. 
Then Lemmas \ref{chi} and \ref{l:stabilizer} imply 
together that $\diff^{1,+}_a$ is generated by 
$\psi_a=a+\lambda(x-a)$.

Let  $a$ and $b$ be two left  points of $C_{\lambda}$. Denote by $D(a,b)$ 
the set of germs at $a$ of classes of local diffeomorphisms $\varphi$ of $(\R,C_{\lambda})$ such that $\varphi(a)=b$. 
Then $D(a,b)$  is acted upon  transitively by $\diff^{1,+}_a$. Using an argument similar  
to the one from above concerning stabilizers, $D(a,b)$ consists of 
germs of maps of the form $\psi_{a,b,k}=b+\lambda^k(x-a)$, with $k\in\Z$.

Let  now $\varphi\in\Diff^{1,+}(\R,C_{\lambda})$ such 
that $\varphi(a)=b$. From above there exists $\delta >0$ such that 
$\varphi\Bigr|_{C_{\lambda}\cap \cN_{\delta}(a)}$ coincides with  $\psi_{a,b,k}\Bigr|_{C_{\lambda}\cap \cN_{\delta}(a)}$ and hence $\varphi'(a)\in \langle \lambda\rangle$. 
Therefore, for any left point $a\in C_{\lambda}$  we have $\varphi'(a)\in \langle \lambda\rangle$.
Now,  left points of $C_{\lambda}$ are dense in $C_{\lambda}$, $\varphi'$ 
is continuous  and $\langle \lambda\rangle$ has no other accumulation points 
in $\R^*$. It follows that  $\varphi'(a)\in \langle \lambda\rangle$, for any $a\in C_{\lambda}$ and 
any $\varphi\in \Diff^{1,+}(\R,C_{\lambda})$.  

For a given $\varphi\in\Diff^{1,+}(\R,C_{\lambda})$ 
its derivative $\varphi'$ is continuous on the whole 
interval $[0,1]$ and hence is bounded. Moreover,  the same argument 
for $\varphi^{-1}$ shows that $\varphi'$ is also bounded from below 
away from $0$, so that $\varphi'\Bigr|_{C_{\lambda}}$ can only take 
finitely many values of the form $\lambda^n$, $n\in \Z$.

The following is a key ingredient in the description of the group 
 $\diff^{1,+}(C_{\lambda})$:

\begin{lemma}\label{finiteness}
Let $\varphi\in\diff^{1,+}(C_{\lambda})$. There is a covering of $C_{\lambda}$ by a finite 
collection of disjoint  closed intervals $I_k$,  such that
$\varphi\Bigr|_{C_{\lambda}\cap I_k}$ is the restriction 
of an affine function to $I_k\cap C_{\lambda}$. Specifically, 
\begin{equation}
\varphi(x)=\varphi(c_k)+ \lambda^{j_k}(x-c_k), \qquad {\rm for }\quad x\in I_k\cap C_{\lambda}, 
\end{equation}
where $c_k$ is a left  point  of $C_{\lambda} \cap I_k$.
\end{lemma}
\begin{proof}
For  $c\in C_{\lambda}$ there is some $m\in \Z$ such that 
$\varphi'(c)=\lambda^{m}$. We want to prove that 
there exists an open neighborhood $U$ of $c$ such that: 
\begin{equation}
\varphi(x)=\varphi(c)+ \lambda^{j_k}(x-c), \qquad {\rm for }\quad x\in U\cap C_{\lambda}.
\end{equation}
Then such neighborhoods will cover $C_{\lambda}$ and we can  
extract a finite subcovering by clopen (closed and open) subsets with 
the same property. 

This claim is true for any left (and by similar arguments for right) end points $c$ of $C_{\lambda}$. It is then 
sufficient to prove that whenever we have a sequence 
of left points $a_n\to a_{\infty}$ contained in a closed interval 
$U\subset [0,1]$ and a $\mathcal C^1$-diffeomorphism   
$\varphi: U\to \varphi(U)\subset [0,1]$ with $\varphi(C\cap U)\subset C$, 
there exists a neighborhood $U_{a_{\infty}}$ of $a_{\infty}$ and an affine 
function $\psi$ such that for large enough $n$ the following holds:
\[\varphi(x)=
\psi(x), \; {\rm for } \, x\in C_{\lambda}\cap U_{a_{\infty}}.\]

Around each left point $a_n$ there are affine maps  
$\psi_{a_n,k_n}:U_{a_n,k_n}\to [0,1]$ defining germs in $D(a_n,c_n)$, 
where $c_n=\varphi(a_n)$,  such that $c_n$ converge to $c_{\infty}=\varphi(a_{\infty})$ and 
\[\varphi(x)=
\psi_{a_n,k_n}(x), \; {\rm for } \, x\in C_{\lambda}\cap U_{a_n,k_n}.\]

We can further assume that $U_{a_n,k_n}\cap C_{\lambda}$ are clopen sets and we can take 
$U_{a_n,k_n}=[a_n,b_n]$ where $b_n$ are right points of $C_{\lambda}$, 
and the sequence $a_n$ is monotone, say increasing.

There is no loss of generality to assume that 
$\psi'_{a_n,k_n}\Bigr|_{C\cap U_{a_n,k_n}}$ is independent on $n$, 
say it equals $\lambda^m$, namely $k_n=m$. Replacing $\varphi$ by its inverse 
$\varphi^{-1}$ we can also assume that $m\leq 0$. 
Since $C_{\lambda}$ is invariant by the homothety of factor 
$\lambda$ and center $0$, we can further reduce the problem to the case 
where $m=0$. We have then $\varphi'(a_{\infty})=1$, by continuity. 

Choose $n$ large enough so that 
$|\varphi'(x)-1|<\varepsilon$, for any $x\in [a_n,a_{\infty}]$, 
where  the exact value of $\varepsilon$ will be chosen later.   
Let now consider the maximal interval of the form  
$[a_n, b]$ to which we can extend 
$\psi_{a_n,0}$ to an affine function which coincides 
with $\varphi$ on $C\cap [a_n,b]$. 

If $b=a_{\infty}$, then the Lemma follows. 
Otherwise, it is no loss of generality in assuming that 
$b=b_n$ and thus $b$ is a right point of $C_{\lambda}$. 
Then $b_n$ is adjacent to some gap $(b_n, d)$. 
Since $d$ is a left point of $C_{\lambda}$ and 
$\varphi'(d)=1$, we can suppose that $d=a_{n+1}$. 

Since $\varphi$ preserves $C\cap U$, it should send the gap 
$(b_n,a_{n+1})$ into some gap contained into 
$[\varphi(a_n), \varphi(b_{n+1})]$. 
Recall from above that the ratios of lengths of gaps of $C_{\lambda}$ is the discrete 
subset $\langle \lambda\rangle\subset \R^*$. 
When $|\varphi'(x)-1|<\varepsilon$, we derive that 
the ratio of the lengths of the gaps 
$\varphi(b_n,a_{n+1})$ and $(b_n,a_{n+1})$ is bounded by 
$1+\varepsilon$. By taking $\varepsilon< 1-\lambda$
we see that the only possibility is that 
the lengths of these two gaps coincide, namely that 
\[ \varphi(a_{n+1})=\varphi(b_n)+a_{n+1}-b_n.\]
This implies that there is a smooth extension of 
$\psi_{a_n,0}$ to an affine function on 
$[a_n, b_{n+1}]$ which coincides with $\varphi$ 
on points of $C_{\lambda}$, contradicting the maximality of $b=b_n$.
This proves that $b=a_{\infty}$, proving the claim. 

When $a_{\infty}$ is not a right point we also have an affine 
extension of $\varphi$ to a right neighborhood of $a_{\infty}$, by the same argument. 
\end{proof}

Consider the rooted binary tree $\mathcal T$ embedded in the plane so that 
its ends abut on the interval $[0,1]$. We label each  
edge $e$ by $l(e)\in\{0, \lambda-1\}$, such that the leftmost edge 
is always labeled $0$. Let $v$ be a vertex of $\mathcal T$ 
and $e_1,e_2,\ldots, e_n$ the sequence of edges representing the geodesic 
which joins the root to $v$. To the vertex $v$ one associates 
then the number 
\begin{equation}
r(v)=\sum_{j=1}^nl(e_j) \lambda^{-j}. 
\end{equation}
Denote by $D(v)$ the set of all descendants of the vertex $v$.
If $I$ is a closed interval in $[0,1]$ we claim that 
$L(C_{\lambda})\cap I$ coincides with the set $r(D(v_I))$, for some unique vertex 
$v_I\in \mathcal T$. 
Furthermore, if $I_1,I_2,\ldots, I_k$ 
is a set of disjoint standard intervals covering $C_{\lambda}$ then 
$v_{I_1},v_{I_2},\ldots,v_{I_k}$ are the leaves of a finite binary 
subtree $T(I_1,I_2,\ldots,I_k)$ of $\mathcal T$ containing the root. 
In particular, if $J_1, J_2,\ldots, J_k$ is another covering of $C_{\lambda}$ 
by standard intervals 
then we have two finite trees 
$T(I_1,I_2,\ldots,I_k)$ and $T(J_1,J_2,\ldots,J_k)$. 
Further,  we also have affine bijections 
$\varphi_j:I_j\to J_j$ which are of the form 
$\varphi_j(x)=b_j+\lambda^{k_j}(x-a_j)$, where $a_j,b_j\in L(C_{\lambda})$. 
It is clear that $\varphi_j(I_j\cap L(C_{\lambda}))=J_j\cap L(C_{\lambda})$. 
The explicit form of $\varphi_j\Bigr|_{I_j\cap L(C_{\lambda})}$ actually can 
be interpreted in terms of $r(v_{I_j})$, as follows. 
Let $\mathcal D(v)$ be the planar rooted subtree of $\mathcal T$ of vertices 
$D(v)$ and root $v$. There is a natural identification 
$\iota_{v,w}$ of the planar binary rooted trees 
$D(v)$ and $D(w)$, for any $v,w\in \mathcal T$.  
When we further identify  
$L(C_{\lambda})\cap I_j$ with the set $r(D(v_{I_j}))$
the induced action of $\varphi_j$  on $w\in D(v_{I_j})$ 
coincides with $\iota_{v_{I_{j}},v_{J_{j}}}$.

Consider now the operation of replacing an interval $I_j$  
by two disjoint intervals $I_j'$ and $I_j"$ 
whose union is disjoint from the other intervals $I_k$. 
Correspondingly we replace $J_j$ by  the couple $\{J_j', J_j"\}=\{\varphi_j(I_j'), \varphi_j(I_j")\}$ and $\varphi_j$ by its restrictions to these smaller intervals. 
This operation does not change the element in $\diff^1(C_{\lambda})$. 
 The immediate consequence of the description of $\varphi_j$ 
is that the pairs of trees 
$T(I_1,\ldots,I_j',I_j",\ldots,I_k)$ and 
$T(J_1,\ldots,J_j',J_j",\ldots,J_k)$ are both obtained from 
$T(I_1,I_2,\ldots,I_k)$ and $T(J_1,J_2,\ldots,J_k)$ 
by adding one caret at the $j$-th leaf. This proves that 
this pair of trees is a well-defined element of 
the standard Thompson group $F$. It is rather clear that 
the map defined this way $\diff^{1,+}(C_{\lambda})\to F$
is an isomorphism. 

In a similar way  we define an isomorphism 
$\diff^{1,+}_{S^1}(C_{\lambda})\to T$, when we work with 
the infinite unrooted binary tree $\mathcal T$ embedded 
in the plane so that its ends abut to $S^1$. 

In the case of $\diff^1_{S^2}(C_{\lambda})$ we use the proof of Theorem \ref{V-type} 
and the infinite unrooted binary tree $\mathcal T$ without any planar structure. 
The only difference is that the restrictions $\varphi\Bigr|_{I_j}$ 
are not having anymore a coherent orientation. Some of them might 
be orientation preserving while the others not. 
This explains  the isomorphism  between $\diff^1_{S^2}(C_{\lambda})$ and the signed Thompson group  $V^{\pm}$.
This ends the proof of Theorem \ref{thompgen} in the case of $C=C_{\lambda}$.

\subsection{Proof of  the general case of Theorem \ref{thompgen}} 
The only missing ingredient is the result generalizing Lemma \ref{finiteness} 
to the more general self-similar sets considered here, as follows: 

\begin{lemma}\label{finitenessgen}
Let $\varphi\in \Diff^{1,+}(\R,C)$, where $C=C_{\Phi}$ is a self-similar Cantor set satisfying the genericity condition (C). Then there is a covering of $C$ by a finite collection of disjoint intervals $I_k$,  such that
$\varphi\Bigr|_{C\cap I_k}$ is the restriction 
of an affine function to $I_k\cap C$. 
\end{lemma}

The proof of this lemma for incommensurable parameters will occupy 
section \ref{prooffiniteness2}. In the case when gaps and homothety factors are respectively 
equal the proof given above extends word by word.

Now, any $\varphi$ in the group $\diff^{1,+}(C)$ 
corresponds to a pair of coverings of $C$ by intervals 
$(I_1,I_2,\ldots,I_k)$ and $(J_1,J_2,\ldots,J_k)$ so that 
$\varphi$ sends affinely $I_j$ into $J_j$, for all $j$. 
These intervals could be chosen to be of the form $[a,b]$, 
where $a$ is a left point of $C$ and $b$ is a right point of $C$.
We call them {\em clopen} intervals.  
  
Particular  examples of clopen intervals are the images of 
$[0,1]$ by the semigroup generated by $\Phi$, 
which will be called {\em standard (clopen)} intervals. 
Each standard clopen interval corresponds to a 
finite geodesic path descending from the root in the 
(regular rooted) tree of valence $n+2$ associated 
to $\Phi$. Thus standard intervals are associated to vertices of the $(n+2)$-valent tree, and one says that they belong to  
the $k$-th {\em generation of standard intervals} if the associated 
vertex is at distance $k$ from the root. The complementary intervals 
to the union of all $k$-th generation of standard intervals will be the 
$k$-th {\em generation of gaps}.  Moreover, given a standard interval 
$I$ of the $k$-th generation, the gaps of the $(k+1)$-th generation lying in $I$ will also be called the first generation 
of gaps in $I$. Notice that, conversely, every gap is a first generation gap for some uniquely determined 
standard interval, to be called its {\em antecedent} standard interval.

Note that the intervals obtained in the previous lemma were not necessarily standard intervals. 
It then remains to prove the following enhancement of Lemma \ref{finitenessgen}: 

\begin{lemma}\label{standard}
We assume that  $C=C_{\Phi}$, where $\Phi$ verifies the genericity condition (C) from Definition \ref{genericity}.
Then any $\varphi\in \diff^{1,+}(C)$ 
corresponds to a pair of coverings of $C$ by standard intervals 
$(I_1,I_2,\ldots,I_k)$ and $(J_1,J_2,\ldots,J_k)$ so that 
$\varphi$ sends affinely $I_j$ into $J_j$, for all $j$. 
\end{lemma}
\begin{proof}
Every clopen interval is the disjoint union of finitely many 
standard intervals and open gaps.  We can therefore 
suppose that all  $I_j$ are standard intervals. 

Note that for any two standard intervals $I, J\subset [0,1]$ there exists an affine 
bijection $(I, I\cap C) \to (J, J\cap C)$, because this holds when $I=[0,1]$. 

We now claim that, conversely, if there exists  an affine 
bijection $\varphi: (I, I\cap C) \to (J, J\cap C)$ and $I$ is standard then 
its image $J$ is also a standard interval. This will prove Lemma \ref{standard}.

Consider a  maximal standard interval $I'\subset J$.  Composing $\varphi$ with 
the affine map in $\diff^{1,+}(C)$ sending  bijectively $(I', I'\cap C)$ onto $(I, I\cap C)$  we can assume that 
$I=I'$.  In particular, the homothety factor $\mu$ of the affine map $\varphi:I\to J$ 
is at least $1$.

Assume first that all homothety ratios are equal to $\lambda$ and all  initial gaps have the same length $g$, 
as in Definition \ref{genericity}.(1).  Observe that all gaps will have sizes of the form $\lambda^mg$, for some $m\in \Z_+$. 
Moreover, if  $I$ is a standard interval of the $k$-th generation, then 
the set of largest gaps in $I$ consists of $n$ equidistant gaps of size $\lambda^kg$.  
Their image by the affine map $\varphi$ is the set of largest gaps in $J$, so that the latter are  
also $n$ equidistant equal gaps in $J$, necessarily of size $\lambda^{n+k}g$, for some $n\in\Z_-$. 
In particular the homothety factor is  $\mu=\lambda^{n}$.  

We  now consider the antecedent standard intervals associated to the largest gaps of $J$.  
If such a gap had size $\lambda^{n+k}g$, its antecedent interval would have size $\lambda^{n+k}$. 
If two of the largest gaps in $J$ have distinct antecedent intervals, then they would be separated 
by another gap of size $\lambda^{n-1+k}g$, contradicting their maximality in $J$. 
Therefore all but possibly the leftmost and rightmost intervals of  the complement 
of these $n$ gaps in $J$  are standard. 

Now, the interval between two consecutive gaps in $I$ 
is a standard interval of length $\lambda^{k+1}$, whose image by the affine map $\varphi$ 
has length $\lambda^{n+k+1}$. 
This shows that the leftmost and the rightmost intervals also 
should be standard intervals, as they have the same size as the remaining 
$(n-1)$ standard intervals between consecutive image gaps. 
This proves that $J$ is a standard interval.

Consider now the case when homothety factors and gaps lengths are incommensurable, as in Definition \ref{genericity}.(2). 
The set of gaps of the same generation 
is totally ordered from the leftmost gap towards the right.  
The sequence of lengths of $(k+1)$-th generation gaps within a 
standard interval of the $k$-th generation is of the form 
$(\Lambda_{\mathbf k}g_1,\Lambda_{\mathbf k}g_2, \ldots, \Lambda_{\mathbf k}g_n)$, for some 
$\mathbf k$. 
Consider now a gap of the first generation, say 
$\Lambda_{\mathbf k}g_{\alpha}$, of $I$. Its image by an affine map 
should be a gap of $J$. It follows that there exists some $\sigma(\alpha)\in \{1,2,\ldots,n\}$ and $\mathbf k_{\alpha}\in \Z_+^{n+1}$, 
so that: 
\[ \mu \Lambda_{\mathbf k} g_{\alpha} =\Lambda_{\mathbf k_{\alpha}} g_{\sigma(\alpha)},\] 
where $\mu$ is the homothety factor of the map $\varphi$. 
Conversely, any gap of $I\subset J$ is the image 
by $\varphi$ of some gap of $I$, and hence there exists some 
 $\tau(\alpha)\in \{1,2,\ldots,n\}$ and $\mathbf l_{\alpha} \in \Z_+^{n+1}$, so that:
\[ \frac{1}{\mu} \Lambda_{\mathbf k} g_{\alpha} = \Lambda_{\mathbf l_{\alpha}} g_{\tau(\alpha)}.\] 
Getting rid of $\mu$ in the two equalities above we obtain the following identities, for all $\alpha,\beta$:
\[ \Lambda_{\mathbf k_{\alpha} +\mathbf l_{\beta}-2\mathbf k }  \; g_{\sigma(\alpha)}g_{\tau(\beta)}=
g_{\alpha} g_{\beta}.\]
By taking $\beta=\sigma(\alpha)$ we derive: 
\[  \Lambda_{\mathbf k_{\alpha} +\mathbf l_{\sigma(\alpha)}-2\mathbf k } \; g_{\tau(\sigma(\alpha))}= g_{\alpha}.\]

If $g_{\alpha} $ and $\lambda_j$ satisfy the genericity condition (C)  the last equality implies 
$\tau(\sigma(\alpha))=\alpha$ and 
$\mathbf k_{\alpha} +\mathbf l_{\sigma(\alpha)}=2\mathbf k$, for every $\alpha$. 
A symmetric argument yields $\sigma(\tau(\alpha))=\alpha$, so that 
$\sigma$ and $\tau$ are bijections inverse to each other. 
Furthermore we derive: 
\[ \mu^n= \prod_{\alpha=1}^n \Lambda_{\mathbf k_{\alpha} -\mathbf k} \frac{g_{\sigma(\alpha)}}{g_{\alpha} }=\Lambda_{\sum_{\alpha=1}^n(\mathbf k_{\alpha} -\mathbf k)},\]
so that 
\[ \mu=\Lambda_{-\mathbf k +\frac{1}{n}\sum_{\alpha=1}^n\mathbf k_{\alpha}}. 
\]
Therefore, for each $\beta$ we have: 
\[  \frac{g_{\sigma(\beta)}}{g_{\beta} }= \Lambda_{-\mathbf k_{\beta}+\frac{1}{n}\sum_{\alpha=1}^n\mathbf k_{\alpha}}.\]
Then our assumptions of genericity imply that $\sigma$ must be identity. 
It turns that all $\mathbf k_{\alpha}$ are equal to some $\overline{\mathbf k}$ and hence $\mu=\Lambda_{\overline{\mathbf k}-
\mathbf k}$. 

Observe that there exists  a standard interval $J'$ and an affine bijection $\psi:(I,I\cap C)\to (J',J'\cap C)$ with homothety factor $\Lambda_{\overline{\mathbf k}-\mathbf k}$. 
Therefore $\varphi\circ \psi^{-1}:J'\to J$ is a translation.  
Moreover,  as $\sigma$ was identity the sequence of first generation gaps in $J'$ is sent by $\varphi\circ \psi^{-1}$ into the sequence of first generation gaps of some standard interval $J''$. 
By induction, the ordered sequence of the $k$-th generation of gaps in $J'$ is sent by $\varphi\circ \psi^{-1}$ into the sequence of the $k$-th generation gaps of $J''$. Since $J$ and $J''$ have the same length it follows that $J=J''$ and hence $J$ is a standard interval, as claimed. 
\end{proof}

Now, it is immediate that 
$\langle \lambda_0\rangle \subset \chi(\Diff^{1,+}_0)$, and by Lemma \ref{l:stabilizer}
there exists some $N\in\Z_+$ so that 
 $\chi(\Diff^{1,+}_0)=\langle \lambda_1^{1/N}\rangle$. 
Since $L(C)$ is affinely locally homogeneous this holds for any left point $a$ of $C$. 

Then, the general form of an affine germ locally preserving $C$ around a left point $c_k\in C\cap I_k$ is:  
\begin{equation}\label{nongeneric}
\varphi(x)=\varphi(c_k)+ \Lambda_{{\bf j}_k,N}(x-c_k), \qquad {\rm for }\quad x\in I_k\cap C, 
\end{equation}
where, for each multi-index $\mathbf k=(k_0,k_1,\ldots,k_n)$ we put:
\begin{equation}
 \Lambda_{\mathbf k,N}= \lambda_0^{k_0/N}\prod_{i=1}^n\lambda_i^{k_i}. 
 \end{equation} 

Furthermore,  we can modify any germ in $\Diff^{1,+}_0$ by using homotheties 
of ratios $\lambda_0^k$, $k\in \Z$ in order to obtain a diffeomorphism  
$\varphi:[0,1]\to [0,r]$ sending $C$ into $C$. By Lemma \ref{finitenessgen} 
we can assume that $\varphi$ is an affine map, and by Lemma  \ref{standard} 
$[0,r]$ must be a standard interval. It follows that the homothety factor of $\varphi$ is 
a power of $\lambda_0$. This implies that $N=1$.  

Pairs of coverings by standard clopen intervals of $C$  correspond 
to pairs of finite rooted subtrees. Subdividing the covering 
by standard subintervals is then equivalent to stabilizing the trees. 
This provides isomorphisms with the Thompson groups $F_{n+1}$, $T_{n+1}$ and the signed Thompson group $V^{\pm}_{n+1}$,  respectively, 
ending the proof of Theorem \ref{thompgen}.

\subsubsection{Proof of Lemma \ref{finitenessgen} for incommensurable parameters}\label{prooffiniteness2} 
We will use a 
much weaker restriction than the total incommensurability, see the conditions used below.
 
Recall from section \ref{thompIFSsec} that the rooted $(n+2)$-valent tree associated to the IFS has 
the edges issued from a vertex labeled by integers from  $0$ to $n$ (from left to right). Then left points of $C$ correspond to 
sequences which eventually end in $0$, namely of the form 
\[ L(i_1\ldots i_p)=i_1i_2\ldots i_p 000000\ldots, \]
while right points correspond to sequences 
which eventually end in $n$:  
\[  R(i_1\ldots i_p)=i_1\ldots i_p nnnnnn\dots.\]

Consider two finite multi-indices $I=i_1\ldots i_p$ 
and $J=j_1\ldots j_q$ and set 
$a=L(i_1\ldots i_p)$, $b=R(i_1\ldots i_p)$, 
$\alpha=L(j_1\ldots j_q)$, $\beta=R(j_1\ldots j_q)$.
Following Definition \ref{germ} the {\em standard germ} $\psi_{I,J}$  is the affine map 
$\psi_{I,J}: [a,b]\to [\alpha,\beta]$ given by the formula: 
\[
\psi_{I,J}(x)= a+ \left(\frac{\prod_{m=1}^q \lambda_{j_m}}{\prod_{k=1}^p \lambda_{i_k}^{-1}}\right) (x-a). 
\]

Each multi-index $I$ determines a vertex $v_I$ of the tree, 
which is the endpoint of the geodesic issued from the 
root which travels along the edges labeled $i_1,i_2,\ldots,i_p$.  
Then, at the level of trees a standard germ corresponds 
to a combinatorial map sending the subtree hanging at 
the vertex $v_I$  onto the subtree issued from the vertex $v_J$, as in the figure below:

\begin{center}
\includegraphics[width=0.5\textwidth]{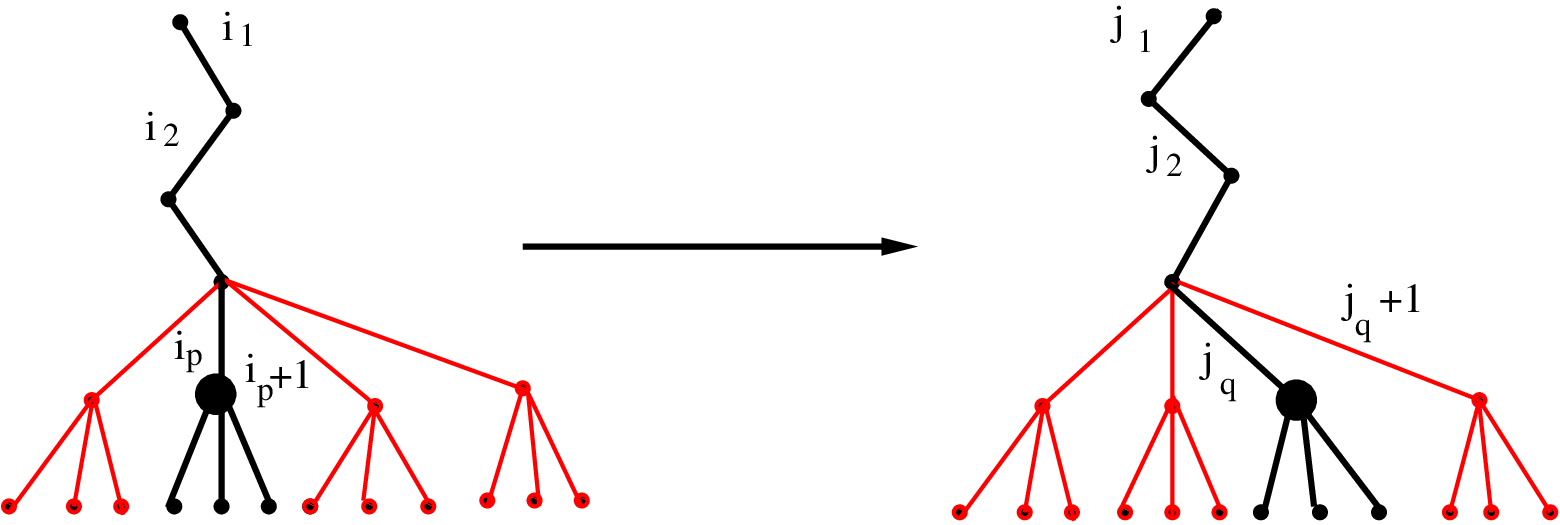}
\end{center}

An {\em  extension} of the standard germ $\psi:[a,b]\to [\alpha,\beta]$ 
is a standard germ defined on $[c,d]\supset[a,b]$ whose restriction to $[a,b]$ coincides with $\psi$ such that 
$[c,d]$ corresponds to a vertex $v_{I'}$ of the tree 
whose multi-index $I'$ is a prefix of $I$, namely 
$I'=i_1i_2\ldots i_{r}$ with $r\leq p$. Note  
that a non-trivial extension of $\psi$ exists only if $i_p=j_q$.

A {\em multi-germ} is a finite collection of standard germs
$\psi_j:[a_j,b_j]\to [\alpha_j,\beta_j]$
such that: 
\[
a_1<b_1<a_2<b_2<c\dots < a_k<b_k, \quad 
\alpha_1< \beta_1<\alpha_2<\beta_2<\cdots <\alpha_k <\beta_k
\]
and  $[b_j, a_{j+1}]$ and $[\beta_j,\alpha_{j+1}]$ are gaps of $C$, for all $j$.

Eventually an {\em extension of a multi-germ} $\{\psi_j\}_{j=1,k}$ 
is a multi-germ $\{\theta_j\}_{j=1,m}$ such that every standard 
germ $\psi_j$ is extended by some $\theta_i$. Notice that  
several elements of the multi-germ $\{\psi_j\}_{j=1,k}$
might have the same extension $\theta_i$.

\begin{lemma}
 Let $\{\psi_j\}_{j=1,m}$ be a multi-germ with the property that there exist
constants $\mu,\nu >0$ satisfying
\[\frac\mu\nu > \frac{1}{\max(\lambda_1,\lambda_2,\ldots,\lambda_n)},\]
such that: 
 \begin{equation}\label{eq:small-1}
 \mu\leq\psi_j'(x)\leq\nu,  \quad {\rm for \;\; every }\; \; x. 
\end{equation}
If  one standard germ $\psi_i$, for some $i\in \{1,2,\ldots,m\}$, admits an extension $\chi$, then 
there exists an extension of  the multi-germ $\{\psi_j\}_{j=1,m}$ 
containing the standard germ $\chi$.
 
Moreover, if a diffeomorphism $\varphi\in \Diff^1(\R,C)$ 
whose derivative $\varphi'$ verifies the condition for 
derivative (\ref{eq:small-1}) coincides with the 
multi-germ $\{\psi_j\}_{j=1,m}$ 
on $[a_1,b_m]$, then it coincides with $\chi$ on its domain of 
definition.
\end{lemma}
\begin{proof}
The standard germ $\psi_j$ is of the form $\psi_j=\psi_{I,J}$, 
with $i_p=j_q=k$.  
We want to construct an increasing  function extending
the standard germ $\psi_{I,J}$ which satisfies the 
condition (\ref{eq:small-1}) for the derivative.  
Such a function will be called a {\em continuation} of $\psi_j$. 
Moving one step upward on the tree (i.e. the ancestor vertices) 
we arrive at the vertices $v_{I'}$ and $v_{J'}$, where 
$I=I'k$, $J=J'k$. 
 
Consider first $k < n$ and seek for a 
continuation on the right side of the interval on which $\psi_{I,J}$ 
is defined. Therefore the continuation must have form drawn below, 
where points marked by squares correspond to each other:

\begin{center}
\includegraphics[width=0.5\textwidth]{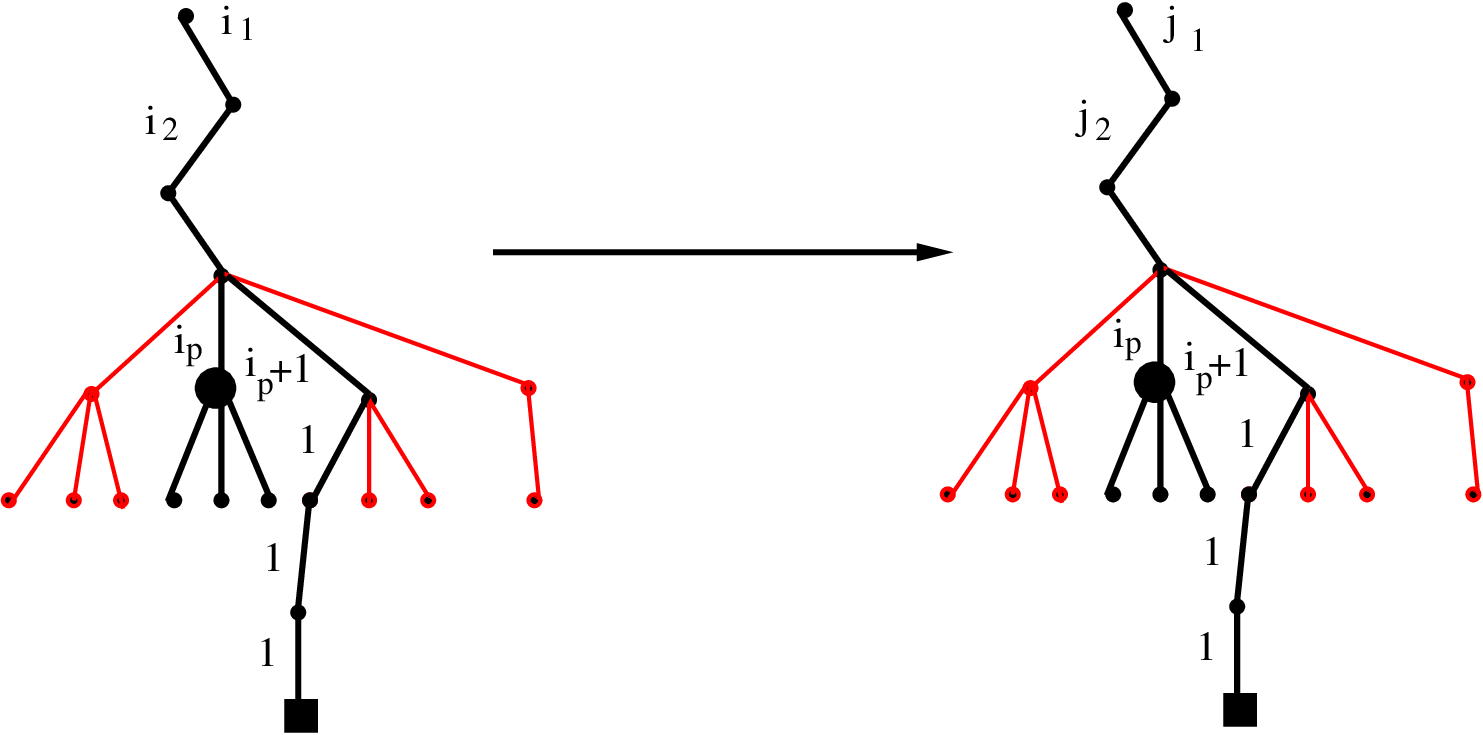}
\end{center}

Since the ratio of the derivatives is uniformly bounded, 
the vertices corresponding to squares should be on the same level, 
namely at equal distance from the vertices $v_{I'}$ and $v_{J'}$, 
respectively. Consider the highest possible level of such squares 
for which the extended map is compatible with the 
standard germ $\psi_{j+1}$. 
We claim that this continuation has the following form, namely that 
squares sit on the vertices $v_{I'k+1}$ and $v_{J'k+1}$:  

\begin{center}
\includegraphics[width=0.5\textwidth]{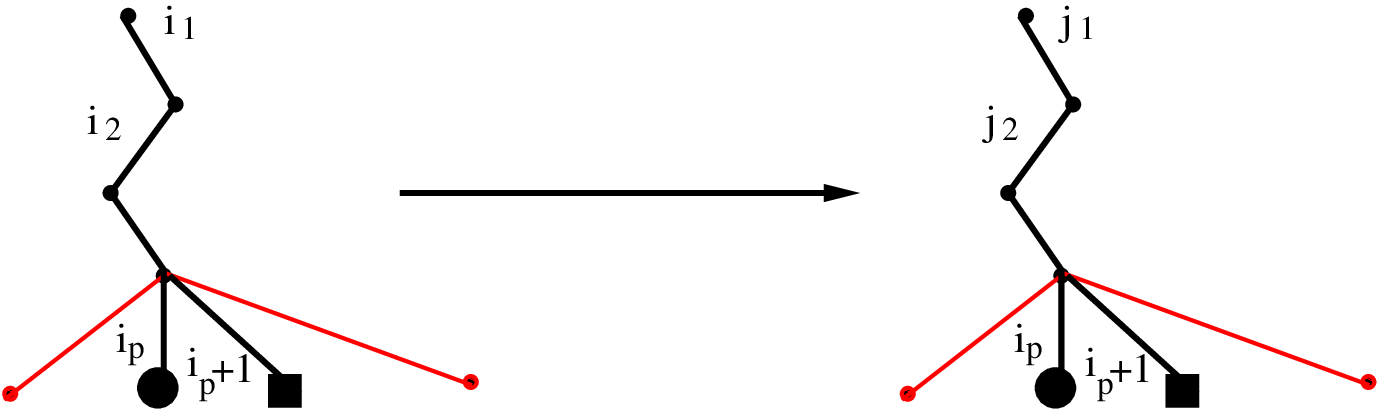}
\end{center}

Assume the contrary holds, namely that the squares sit on lower levels, 
as in the figure below:

\begin{center}
\includegraphics[width=0.5\textwidth]{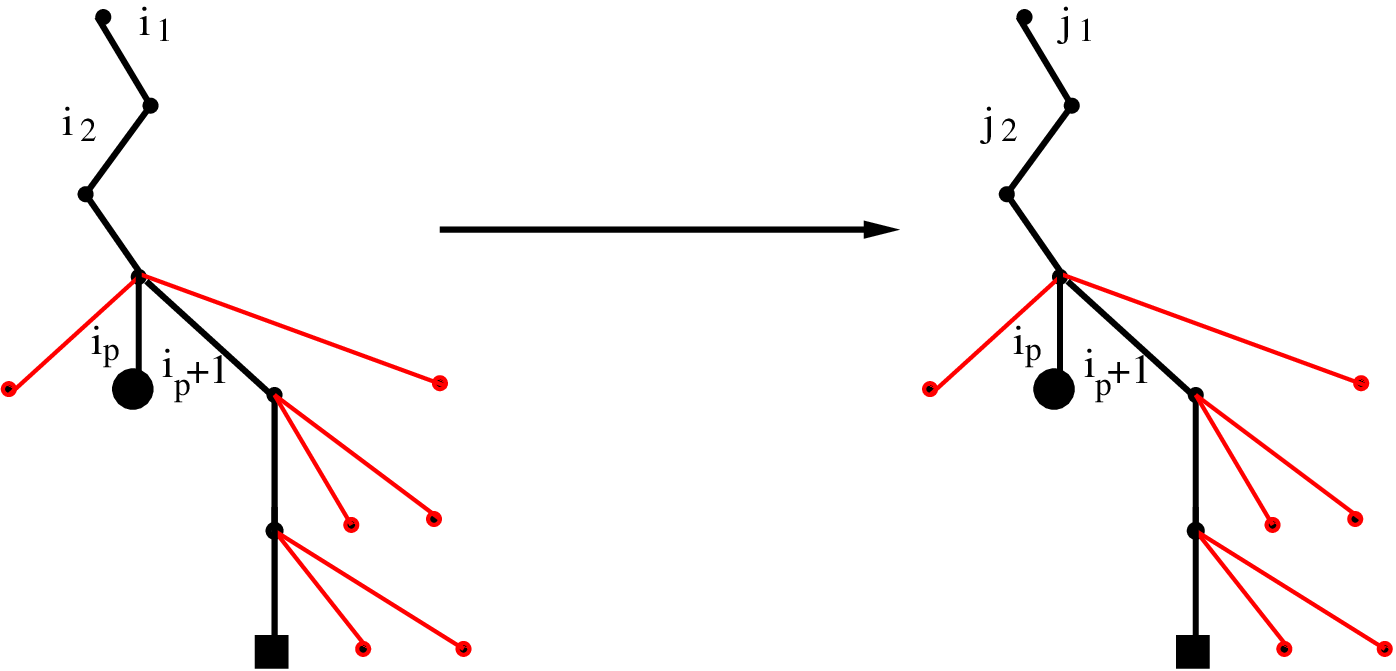}
\end{center}

Consider further continuation to the right of this 
increasing function. We label points on the next branch issued 
from the ancestor of squares vertices by triangles and further by 
hexagons etc. Consider further the highest levels for which continuation 
is compatible. Then the  picture
\begin{center}
\includegraphics[width=0.17\textwidth]{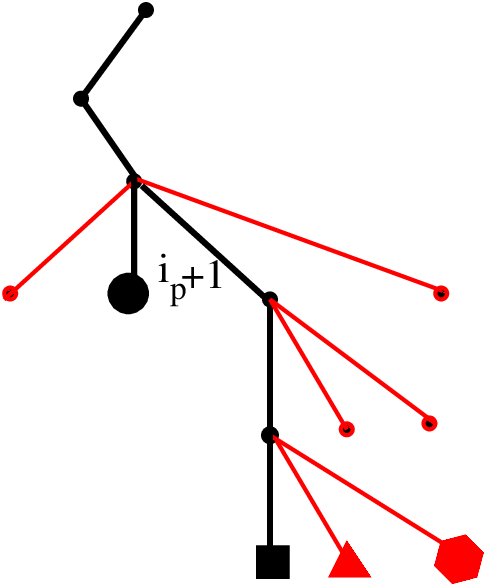}
\end{center}
is impossible, since then the ancestor of the 
square vertex also should  have been labeled by a square. 
Therefore we must continue along an infinite path down to a 
boundary point of the tree, as in the figure: 

\begin{center}
\includegraphics[width=0.2\textwidth]{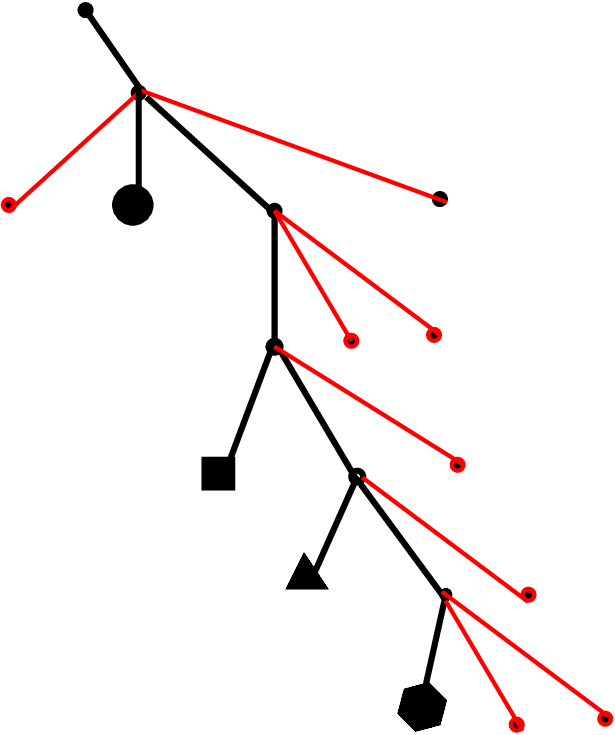}
\end{center}

The boundary point corresponds to an infinite multi-index 
$I$. Then $\xi=\xi(I)\in[0,1]$ cannot be a right point of the Cantor set, since 
this would give a continuation to a whole subtree issued from 
$v_{I'}$, contradicting the form of our path.  

Now our continuation coincides with the multi-germ $\{\psi_j\}_{j=1,m}$ 
for  values $x\in[a_j,\xi]$. Since $\xi$ is not a right point,
they coincide in a right semi-neighborhood of $\xi$ 
and this contradicts the choice of our infinite path.

We summarize the discussion above as follows. Let $k_r<n$;  
then the only possible right continuation  (which satisfies the condition 
(\ref{eq:small-1})) of  
$\psi_{Ik_1\ldots k_r,Jk_1\ldots k_r}$ is by the germ 
$\psi_{Ik_1\ldots k_{r-1}k_{r}+1,Jk_1\ldots k_{r-1}k_r+1}$. A similar argument shows that  whenever $k_r >0$ 
the only possible left continuation  (which satisfies the condition 
(\ref{eq:small-1})) of  $\psi_{Ik_1\ldots k_r,Jk_1\ldots k_r}$ is by the germ 
$\psi_{Ik_1\ldots k_{r-1}k_{r}-1,Jk_1\ldots k_{r-1}k_r-1}$.

Repeating the same argument, we get the desired statement.
\end{proof}

\begin{lemma}
There exists $\epsilon>0$  with the following property.
Consider a   standard  germ
$\psi_{I,J}$ with $i_p\neq j_q$ and $j_q\neq n\neq i_p$.

Then any  continuation of $\psi_{I,J}$ to 
a standard germ $\theta$ sending $L(i_1i_2\ldots i_{p-1}i_{p}+1)$ to 
$L(j_1j_2\ldots j_{q-1}j_{q}+1)$ which is defined in a right semi-neighborhood of $L(i_1i_2\dots i_{p-1}i_{p}+1)$  
is either an extension of the standard germ $\psi_{I,J}$, or else it verifies:
 \[
 \left|\frac {\psi_{I,J}'}{\theta'}-1\right| > \epsilon.
 \]  
\end{lemma}
Notice that $\theta$ is locally affine and hence 
we don't need to specify the point (of the corresponding domain of definition) in which we consider the derivative. 
\begin{proof}
The ratio of the derivatives of the standard germs  
$\psi_{I,J}$ and $\theta=
\psi_{i_1i_2\ldots i_{p-1}i_{p}+1, \; j_1j_2\ldots j_{q-1}j_{q}+1}$
is given by: 
\[\frac {\psi'}{\theta'}=\frac
{\lambda_{i_p}^{-1}\lambda_{i_q}}
{\lambda_{i_p+1}^{-1}\lambda_{i_q+1}} \lambda_1^m, 
\]
where $m\in\Z$. This is a discrete subset of $\R^*$ and hence the claim.
\end{proof}

We can apply the same arguments when $i_p\neq 0\neq j_q$.
Specifically, we have: 
\begin{lemma}
Let $n\ge 2$. Then there exists $\epsilon>0$ such that any 
multi-germ $\{\psi_j\}_{j=1,m}$ with the property: 
 \[
\left|\frac{\psi_i'}{\psi_j'}-1\right|<\epsilon
 \]
 admits an extension containing with at most two elements. 
\end{lemma}
\begin{proof}
It remains to examine the standard germs 
$\psi_{I,J}$ in following two cases: 
\[(I,J)\in\{(i_1\ldots i_{p-1} 0, \;  J=j_1\dots j_{q-1} n), \;
(i_1\ldots i_{p-1} n, \; j_1\dots j_{q-1} 0)\}.
\]
The corresponding picture depends on 
the number $s$ of occurrences of $n$ in the tail of  
$j_1\dots j_{q-1} n$ and the positions of the the square vertices 
(having $r$ and $m$ respectively ancestors labeled $0$) as below:

\begin{center}
\includegraphics[width=0.5\textwidth]{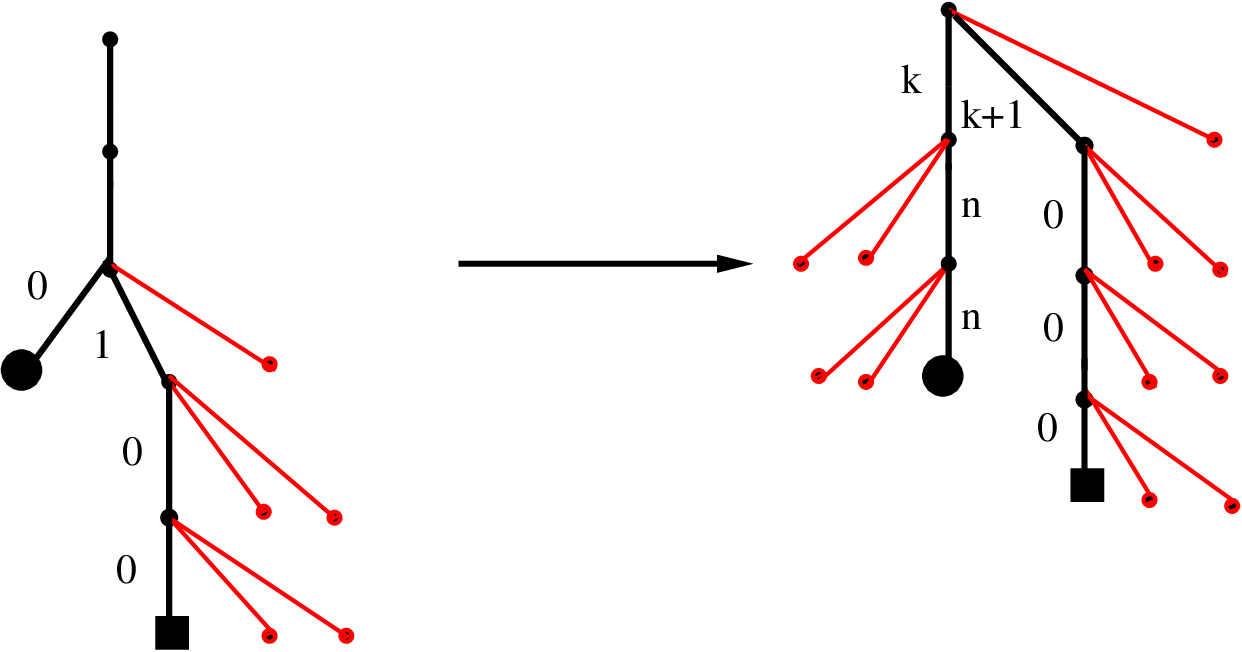}
\end{center}

The ratio of derivatives is
\[
\frac{\lambda_n^s\lambda_k\lambda_{k+1}^{-1}\lambda_0^{-r}}
{\lambda_0\lambda_1^{-1}\lambda_0^{-m} }=
\frac{\lambda_k\lambda_1}{\lambda_{k+1}}\cdot\frac{\lambda_n^{s}}{\lambda_0^{r-m+1}}
\]
Letting $s$ and $\mu=r-m+1$ be large enough we can insure 
that  $\lambda_n^s/\lambda_0^{\mu}$ is arbitrarily 
close to $\lambda_{k+1}/\lambda_k\lambda_1$. 
In this case $\mu >0$,  so that we 
can automatically extend the new standard germ  obtained this 
way and get the figure below, where the position of the squared
vertex is the highest possible: 

\begin{center}\label{fig15}
\includegraphics[width=0.5\textwidth]{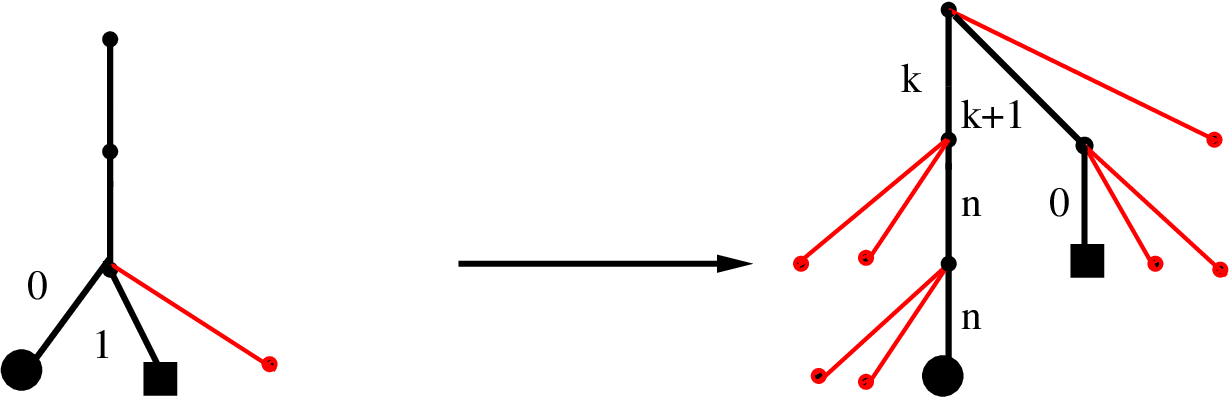}
\end{center}

Now, as $n\geq 2$ we cannot find a non-trivial extension of 
the two standard germs corresponding to the labeled vertices.
This means that there is an extension with at most two elements, thereby 
proving our statement. 
\end{proof}

\begin{lemma}
 Let $n=1$. Then there exists $\epsilon>0$ such that any multi-germ $\{\psi_j\}_{j=1,m}$ verifying the condition:  
 \[
\left|\frac{\psi_i'}{\psi_j'}-1\right|<1+\epsilon
 \]
 admits an extension containing  at most $4$ elements. 
\end{lemma}
\begin{proof}
The only possible situation is that pictured below:

\begin{center}
\includegraphics[width=0.4\textwidth]{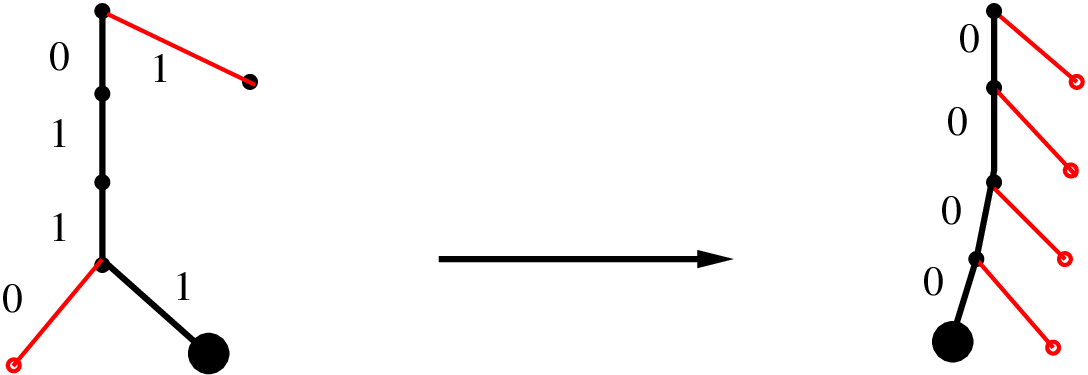}
\end{center}

Consider a right continuation as follows:

\begin{center}
\includegraphics[width=0.4\textwidth]{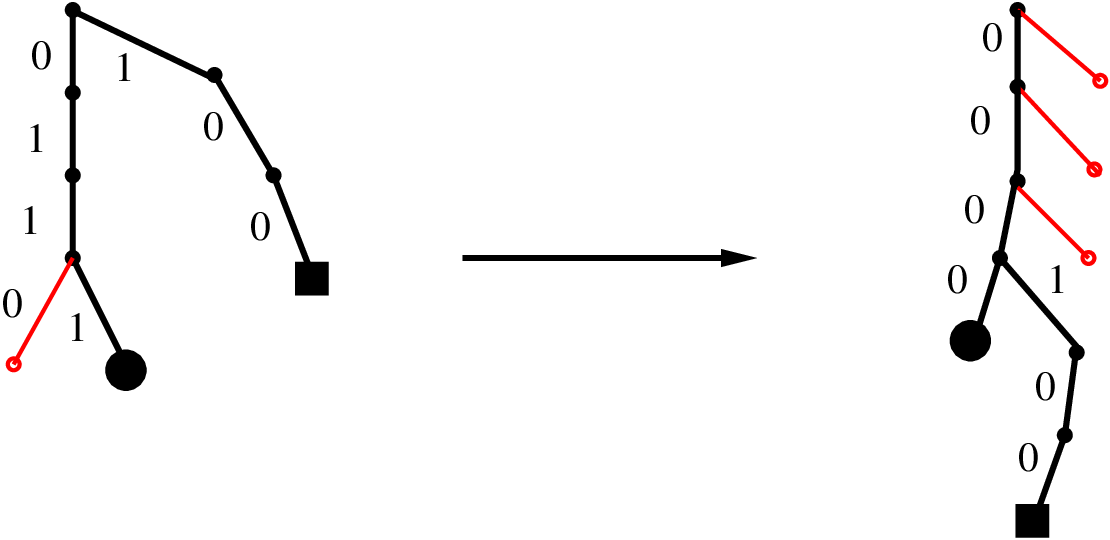}
\end{center}
In the left hand side picture we have $r+1$ 
ancestors of the fat dotted vertex which are labeled $1$ and 
$s$ ancestors of the square vertex labeled $0$, while in the 
right picture there are $v$ ancestors of the square vertex labeled $0$.  
Then the ratio of derivatives of the two standard germs is: 
\[
\frac{\lambda_1^{-r-1}\lambda_1\lambda_0^s}
{\lambda_1\lambda_0^v}=\frac{\lambda_0^{s-v}}{\lambda_1^{r+1}}
\]
We can approximate arbitrarily close $1$ by 
$\lambda_0^{s-v}/\lambda_1^{r+1}$, but then $s-v$ must be large, and in particular positive. This implies that we can automatically extend 
this to a standard germ as follows:

\begin{center}
\includegraphics[width=0.4\textwidth]{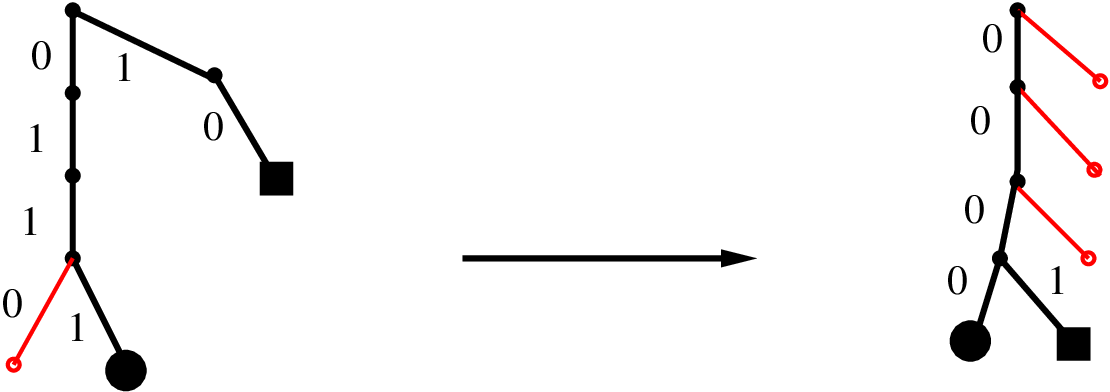}
\end{center}

or, after removing nonessential information:

\begin{center}
\includegraphics[width=0.33\textwidth]{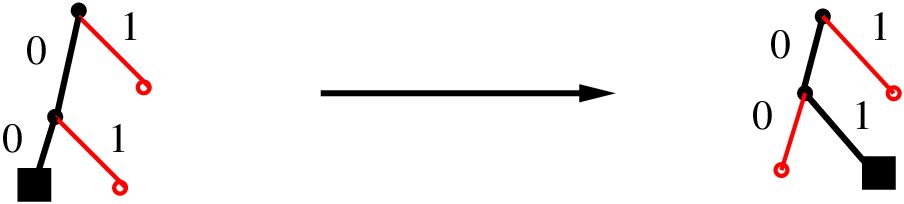}
\end{center}

And we now see that a right continuation is impossible. Thus we get our claim. 
\end{proof}

\subsubsection{Cantor sets with commensurable parameters}\label{prooffiniteness1} 
The genericity condition (C) could be extended to include also the case when all homothety factors $\lambda_i$ are commensurable. We skip the details and present instead an example of an asymmetric Cantor set 
$AC$ which is the attractor of the IFS: 
\[ \phi_0(x)=\frac{1}{4}x, \; \phi_1(x)=\frac{1}{2}x+\frac{1}{2}.\]
For each finite multi-index $I=i_1i_2\ldots i_k$, with $i_j\in\{0,1\}$  we 
set $l_{\emptyset}=0$ and define by induction:
\[l_{0I}=\frac14 l_{I}, \; l_{1I}=\frac 12 l_I +\frac12.\]
Then $L(AC)=\{l_{I}; I\;  
{\rm finite \;  and \; admissible}\}$, where $I=i_1i_2\ldots i_k$ is admissible if it is either empty or else $i_k=1$. 
If we set 
$ l_I =\lim_{k\to \infty}l_{i_1i_2\ldots i_k}$ for infinite $I$, then $AC$ is the union of 
$L(AC)$ and the set of $l_I$, with infinite $I$. Further, $L(AC)$  is identified with the set of those 
$l_I$ for which $I$ is infinite and eventually $0$.  It follows as in the case of $C_{\lambda}$ that $\chi(\Diff^1_0(AC))=\langle 4\rangle$. 
Further, for $a,b\in L(AC)$,  we obtain:  
\begin{equation}
D(a,b)=\left\{\psi_{a,b,k}= b+\frac{1}{2^{n(a,b)}}4^{-k}(x-a), k\in \Z\right\},
\end{equation}
where $n(a,b)\in\{0,1\}$ is the parity of the length 
of the geodesic joining $a$ to $b$ in the reduced binary tree associated to the IFS.

The previous arguments show that any element $\varphi$ of $\diff^{1,+}(AC)$ 
determines a {\em finite} covering of $AC$ by  
intervals $I_j$ on which $\varphi|_{I_j}$ is of the form 
$\psi_{a_j,b_j,k_j}$, for some $a_j\in L(AC)$. Moreover $\diff^{1,+}(AC)$ is isomorphic to the Thompson group $F$.

\subsection{Proof of Theorem \ref{highthomp}}
Let $\Diff^1_{a}(\R^n,C)$ denote the stabilizer of 
$a\in C$ in the group  $\Diff^1(\R^n,C)$.
We verify immediately that the 
map $\chi: \Diff^1_{a}(\R^n,C)\to GL(n,\R)$, given by 
$\chi(\varphi)=D_a\varphi$ is a homomorphism. 
In the case when $C$ is a product we can describe
explicitly $\chi(\Diff^1_{a}(\R^n,C))$. 
For the sake of simplicity we restrict ourselves to the 
case $n=2$. Consider $C=C_{\lambda_1}\times C_{\lambda_2}$. 
We say that $a=(a_1,a_2)\in C$ is an {\em end} point of $C$ if both 
$a_i$ are endpoints of $C_{\lambda_i}$.

\begin{lemma}
Suppose that $\lambda_i >2$ and $a$ is an end point of $C$. 
\begin{enumerate}
\item If  $\lambda_1\neq \lambda_2$ then 
\begin{equation}
 \chi(\Diff^{1}_{a}(\R^2,C)) =\langle \lambda_1\rangle \oplus 
\langle \lambda_2\rangle. 
\end{equation}
\item If $\lambda_1=\lambda_2=\lambda$ then 
\begin{equation}
\chi(\Diff^{1,+}_{a}(\R^2,C))= \langle \lambda\rangle \oplus 
\langle \lambda\rangle. 
\end{equation}
\end{enumerate}
\end{lemma}
\begin{proof}
From the first part of the proof of Theorem \ref{products} we infer that 
whenever $C$ is a product and $a\in C$ is fixed by $\varphi$ 
its differential $D_a\varphi$ must send 
both horizontal and vertical vectors 
into horizontal or vertical vectors. 

Moreover, when $\lambda_i$ are distinct the horizontality/verticality 
of the segment should be preserved. Otherwise $\varphi$ 
would induce a germ of $\mathcal C^1$-diffeomorphism 
$\phi:(\R,C_{\lambda_1})\to (\R, C_{\lambda_2})$. By Remark \ref{Hausdorff}  
we need $\lambda_1=\lambda_2$.

Therefore $\varphi$ restricts to germs of diffeomorphisms 
$\phi_i\in \Diff^{1}_{a_i}(\R,C_{\lambda_i})$. By Lemma \ref{chi} and Remark \ref{decrease} 
$\chi(\phi_i)=\langle\lambda_i\rangle$. 
This proves  the first item.  

On the other hand when $\lambda_1=\lambda_2$  we can locally identify $(C_{\lambda_1},a_1)$ and 
$(C_{\lambda_2},a_2)$ by an affine germ. The linear map $R_a=\left(\begin{array}{cc}
0 & 1 \\
1 & 0 \\
\end{array}\right)$  
which exchanges the two orthogonal axes meeting at $a\in C$ belongs  to 
$\Diff^{1}_a(\R^2,C)$.  We can compose $\varphi$ with $R_a$, if needed,  in order to 
have  $D_a\varphi$ diagonal. 
Thus $\chi(\Diff^{1}_{a}(\R^2,C))=\left\langle \langle \lambda\rangle \oplus 
\langle \lambda\rangle, R_a\right\rangle$. Taking into account that  $\det(R_a)=-1$, so $R_a$ is orientation reversing, we  
obtain  the claim. 

Observe that $\diff^{1}_{a, \R^2}(C)$ is either 
isomorphic to $\Z^2$, when $\lambda_i$ are distinct, or 
an extension of $\Z^2$ by  $\Z/2\Z$, 
otherwise. 
 \end{proof}

Consider now that $\lambda_1=\lambda_2$. 
Let now $a$ and $b$ be two end points of $C$. Denote by $D(a,b)$ 
the set of germs at $a$ of classes in $\diff^{1}_{\R^2}(C)$ having representatives 
$\varphi\in\Diff^{1}(\R^2,C)$ such that $\varphi(a)=b$. This set is acted upon  
transitively by $\diff^{1,+}_{a, \R^2}(C)$, so that $D(a,b)$ consists of 
maps  of  the form: 

\begin{equation}
 \psi_{a,b,\mathbf k, n}(x)=(b_{j,i}+ \lambda^{k_j}(x_i-a_{j,i}))_{i=1,2}\circ S_{a}^{n_b}, \qquad {\rm for \quad any }\quad x\in I_j\cap C. 
 \end{equation} 
where $S_{a}^{n_b}$ is an element of the group $D_2$ of orientation 
preserving symmetries of the square, namely $S_{a}$ is a rotation of 
order $4$ fixing $a$ and $n_b\in \{0,1,2,3\}$.

Now  the set of endpoints of $C$ is kept invariant 
by any $\varphi\in\Diff^{1}(\R^2,C)$. 
Therefore, for any endpoint 
$a\in C$ there exists some $k_i,n$ depending on $a$ such that  
$D_a\varphi= 
(\lambda^{k_1}\oplus \lambda^{k_2})\circ S_{a}^n$. The set of possible values 
of $D_a\varphi$ is then a discrete subset of $GL(2,\R)$.  
Since endpoints of $C$ are dense in $C$ and $D\varphi$ is continuous 
we have  $D_a\varphi$ is of the form $(\lambda^{k_1}\oplus \lambda^{k_2})\circ S_{a}^n$, for any $a\in C$ and any $\varphi\in \Diff^{1,+}(\R,C)$.  

For a given $\varphi$  both the norm  $\parallel D\varphi\parallel$ and the 
determinant  $\det(D\varphi)$ of its differential are continuous on  
 $[0,1]\times [0,1]$ and hence  they 
are bounded. Moreover,  the same argument 
for $\varphi^{-1}$ shows that these quantities are also bounded from below 
away from $0$, so that $D\varphi\Bigr|_{C}$ can only take 
finitely many values. 
The next point is the analogue of Lemma \ref{finiteness} 
to this situation: 

\begin{lemma}\label{finitenessproduct}
There is a covering of $C$ by a finite 
collection of disjoint standard rectangles $I_k$  whose images are standard rectangles such that
$\varphi\Bigr|_{C\cap I_k}$ is the restriction of an affine function 
and more precisely we have: 
\begin{equation}
\varphi(x)=(\lambda^{j_{k,1}}\oplus \lambda^{j_{k,2}})\circ S_{b_k}^{m_k}(x-(\alpha_1,\alpha_2))+\varphi(\alpha_1,\alpha_2), \; {\rm for }\;  x\in I_k\cap C, 
\end{equation}
where  $\alpha_i$ are left points of 
$C_i$.  
\end{lemma}
\begin{proof}
We can choose both $I_k$ and their images to be standard rectangles, as in the case of central Cantor sets $C_{\lambda}$. 

Let $c\in C$. Then $D_c\varphi= 
(\lambda^{j_{k,1}}\oplus \lambda^{j_{k,2}})\circ S_{b_k}^{m_k}$, which we denote by $A$ for simplicity in the proof.
We have to prove that there exists a neighborhood $U$ of $c$ such that: 
\begin{equation}
\varphi(x)= A(x-(\alpha_1,\alpha_2))+ \varphi(\alpha_1,\alpha_2), 
\qquad {\rm for }\quad x\in U\cap C.
\end{equation}
Such neighborhoods will cover $C$ and we can extract 
a finite subcovering  by clopen sets to get the statement.

This claim is true for endpoints $a=(\alpha_1,\alpha_2)$ of $C$. It is then 
sufficient to prove that whenever we have a sequence 
of endpoints $a_n\to a_{\infty}$ contained in a closed rectangle 
$U\subset [0,1]$ and a $\mathcal C^1$-diffeomorphism   
$\varphi: U\to \varphi(U)\subset [0,1]$ with $\varphi(C\cap U)\subset C$, 
there exists a neighborhood $U_{a_{\infty}}$ of $a_{\infty}$ and an affine 
function $\psi$ such that for large enough $n$ the following holds:
\[\varphi(x)=
\psi(x), \; {\rm for } \, x\in C_{\lambda}\cap U_{a_{\infty}}.\]

Around each left point $a_n$ there are affine maps  
$\psi_{a_n}:U_{a_n,k_n}\to [0,1]$ defining germs in $D(a_n,c_n)$, 
where $c_n=\varphi(a_n)$,  such that $c_n$ converge to $c_{\infty}=\varphi(a_{\infty})$ and 
\[\varphi(x)=
\psi_{a_n}(x), \; {\rm for } \, x\in C_{\lambda}\cap U_{a_n,k_n}.\]

We can further assume that $U_{a_n}\cap C_{\lambda}$ are clopen sets and we can take 
$U_{a_n}$ to be standard rectangles 
$[\alpha_{n,1},\beta_{n,1}]\times [\alpha_{n,2}, \beta_{n,2}]$ 
where $\beta_{n,i}$ are right points of $C_i$.

There is no loss of generality to assume that 
$D\psi_{a_n}\Bigr|_{C\cap U_{a_n}}$ is independent on $n$, 
say it equals $(\lambda^{m_1}\oplus \lambda^{m_2})S^{j}$. 
Replacing $\varphi$ by its inverse 
$\varphi^{-1}$ we can also assume that $m_1\leq 0$. 
Since $C_1$ is invariant by the homothety of factor 
$\lambda$ and center $0$, we can further reduce the problem to the case 
where $m_1=0$.  We can further assume that $m_2\leq 0$ by the same trick and 
finally get rid of the second diagonal component of the differential.  
Then, by continuity, we have $D_{a_{\infty}}\varphi=S^j$. 

Choose $n$ large enough so that 
$\parallel D_x\varphi(x)-\mathbf 1\parallel<\varepsilon$, for any $x$ in a square 
centered at $a_{\infty}$ and containing all $U_{a_n}$, with $n$ large enough.  
where  the exact value of $\varepsilon$ will be chosen later.   

Let now consider the maximal standard rectangle of the form  
$U'=[\alpha_{n,1}, \beta_1]\times [\alpha_{n,2}, \beta_2]$ 
to which we can extend $\psi_{a_n}$ to an affine function which coincides 
with $\varphi$ on $C\cap U'$. 

The endpoint $(\beta_1,\beta_2)$ belongs to the closure of three 
maximal rectangular gaps: the rectangle $Q$ which is opposite to $U'$ is 
a product of two gaps, while the other two $Q_v$ and $Q_h$  are products 
of gaps with one (vertical or horizontal) side of $U'$. 
Since $D_x\varphi$ is close to identity 
the image of the rectangular gaps are closed to rectangular gaps 
of approximatively the same sizes. 
Now, the images by $\varphi$ of the vertices of $Q$ are 
points of $C$ forming a rectangle, which is itself the product of two gaps.
Thus the sizes of this rectangle belong to the set 
$\{(\lambda-2)\lambda^{-n}; n\in \Z_+\} \times \{(\lambda-2)\lambda^{-n}; n\in \Z_+\}$. 
Since the ratios of two different lengths form a discrete 
set and $D_x\varphi$ is close to identity, the four points in the image form 
a rectangle congruent to $Q$. A similar argument holds now for the 
rectangles $Q_v$ and $Q_h$. This implies that $\psi_{a_n}$ 
can be extended to an affine function on a strictly larger rectangle, 
contradicting our assumptions. This proves the claim. 
\end{proof} 

This description shows that $\diff^{1,+}_{\R^2}(C)$ is isomorphic  
to $2V^{sym}$, namely the Brin's group decorated by $D_2$ (see 
\cite{brin1} for the non-decorated case).  Here $D_2$ is the group of the orientation preserving symmetries of the cube, namely the 
group of orthogonal matrices with integral entries and unit determinant.

 \begin{remark}
We can obtain smaller decoration by choosing self-similar Cantor sets with less symmetries. For instance,  
$\diff^{1,+}_{\R^{n}}(C)$ is isomorphic to Brin's group decorated by the positive isometry group of a 
rectangular parallelepiped with edges of different sizes, if $C=C_{\lambda_1}\times C_{\lambda_2}\times \cdots C_{\lambda_n}$, where $\lambda_i$ are pairwise distinct but commensurable, namely there exists $\alpha\in \R^*_+$ and $k_i\in \Z$ such that $\lambda_i=\alpha^{k_i}$, for all $i$. We expect a similar result when $\lambda_i$ are incommensurable. 
Moreover, by replacing each $C_{\lambda_i}$ by some non-invertible self-similar Cantor set, the corresponding 
group $\diff^{1,+}_{\R^{n}}(C)$ is isomorphic to Brin's group $nV$. 
\end{remark}

\begin{remark}
Following the arguments in the proof of  Theorem \ref{V-type} one shows that 
$\diff^{1,+}_{\R^{n+k}}(C)$, for $k\geq 1$, is Brin's group $nV^{\pm sym}$ decorated by the hyperoctahedral group $O_n$, 
namely the group of symmetries of the cube (possibly reversing the orientation). 
\end{remark}

\section{Examples and counterexamples}\label{examples}
\subsection{Nonsparse Cantor sets with uncountable 
diffeomorphism group}
Let $h:\R_+\to \R$ be a $C^\infty$-function satisfying
the following conditions: 
\begin{align*}
h(x)&=0, \qquad \text{for $0\leq x\leq 1$, $x\geq2$},
\\
h(x)&>0, \qquad\text{for $1<x<2$},
\\
h'(x)&>-1.
\end{align*}
Since the maps $g_j:[0,1]\to[0,1]$ given by:
\begin{equation}
 g_j(x)=x+2^{-2^j} h(2^j x), \;\; j\in \Z^*_+, 
\end{equation}
are strictly increasing,  they are 
smooth diffeomorphisms of the interval.  
The support of $g_j$ is $[2^{-j},2^{-j+1}]$ and hence the 
diffeomorphisms $g_j$ pairwise commute.
Their derivatives are of the form: 
 \[
 g'_j(x)=1+ 2^{j-2^j}h'(2^j x), 
 \]
and respectively 
\[g_j^{(k)}(x)=2^{kj-2^j}h^{(k)}(2^j x), \qquad{\rm for }\; k\geq 2.  \]

 Consider the group $R$ consisting of bounded infinite sequences 
$\mathbf m=m_1, m_2,\ldots$ of integers, endowed with 
the term-wise addition.

 There is a  map $\Theta:R\to \Diff^{0}([0,1])$ given by:
 \begin{equation}
 \Theta(\mathbf m)= \lim_{n\to \infty}g_1^{m_1}\circ g_2^{m_2}\circ\cdots
\circ g_n^{m_n}, 
\end{equation}
where $g_j^m$ is the $m$-fold composition of $g_j$. 
The order in the previous definition does not matter, as the maps 
commute. The limit map $\Theta(m)$ is immediately seen to be
a homeomorphism of $[0,1]$ which is a diffeomorphism outside $0$. 

Let first consider only those $\mathbf m$ where $m_j\in \{0,1\}$. Then we can compute first:
\[
\lim_{x\to 0} \Theta(\mathbf m)'(x)=1,\]
and then 
\[
\lim_{x\to 0} \Theta(\mathbf m)^{(k)}(x)=1, \qquad {\rm for }\; k\geq 2.\]
Therefore $\Theta(\mathbf m)$ is a $\mathcal C^{\infty}$ diffeomorphism
of $[0,1]$.  

Moreover any element of $R$ can be represented as a product of 
$\Theta(\mathbf m)$, with $\mathbf m$ of having only $0$ or $1$ 
entries. This implies that  $\Theta(R)\subset \Diff^{\infty}([0,1])$. 
Furthermore it is clear that $\Theta$ is injective, by looking at the 
factor corresponding to the first place where two sequences disagree. 
This implies that  $\Theta$ provides a faithful $\mathcal C^{\infty}$ action 
of $R$  by $\mathcal C^{\infty}$  diffeomorphisms on $[0,1]$.

The dynamics of each $g_j$ on its support $[2^{-j},2^{-j+1}]$ 
is of type north-south with repelling and attracting fixed points on the 
boundary. Pick up some 
$a_j\in (2^{-j},2^{-j+1})$, so that $b_j=g_j(a_j)>a_j$. 
Then the intervals $g_j^n((a_j,b_j))$ are all pairwise disjoint. 
If $C_{j}^0\subset [a_j,b_j]$ is some Cantor set, then the closure of its orbit, namely  
$C_j=\cup_{k=-\infty}^{\infty}g_j^k(C_j^0)\cup \{2^{-j},2^{-j+1}\}$ is a $g_j$-invariant Cantor subset 
of $[2^{-j},2^{-j+1}]$. Moreover, for any $n\neq 0$ 
the restriction $g_j^n\Bigr|_{C_j}$ cannot be identity, since $g_j^n$ is 
strictly increasing. 

Then their union $C=\cup_{j=1}^{\infty}C_j$ is a Cantor subset of $[0,1]$ 
and for $\mathbf m$ not identically $0$ we also have 
 $\Theta(\mathbf m)\Bigr|_{C}$ is not identity. 
This shows that the diffeomorphism group $\diff^{\infty}(C)$ contains $R$.  
In particular, $\diff^{\infty}(C)$ is uncountable.

This technique of construction is close to that involved in other classical constructions in the field,
as in \cite{Ts,N2,J}. In particular, such a group arises as the group of fragmentations of elements 
of the Grigorchuk-Maki group of intermediate growth acting by $\mathcal C^1$-diffeomorphisms of the interval from 
\cite{N2}.

\subsection{Nonsparse Cantor set with trivial diffeomorphism group}
Let $X$ be obtained by removing a sequence of intervals, as follows. 
At the first step we remove from $[0,1]$ the central interval 
of length $\frac{1}{4}$. At the $m$-th step we have 
$2^m$ intervals which we label, starting from the leftmost to the rightmost 
as $I_1^{(m)},I_2^{(m)}, \ldots, I_{2^m}^{(m)}$. 
We remove then from $I_j^{(m)}$ the central interval 
of length $2^{-2^{2^{m-1}-1+j}}$. 
The result of this procedure is a Cantor set $X$ which is not sparse.

Let $\varphi\in \Diff^{1}(\R,X)$. 
Suppose first that there exists a sequence $I_n$ of gaps approaching $0$ from the right side with the property that 
for every $n$ we have $J_n=\varphi(I_n)\neq I_n$.  Then there exists points $x_{I_n}\in I_n$ such that 
$\varphi'(x_{I_n})=|J_n|/|I_n|$. The sequence $x_{I_n}$ converges to $0$. Now lengths of gaps belong to the discrete set 
$\{ 2^{-2^n}, n\in \Z_+\}$ and there are not two gaps of the same length. Therefore 
any infinite sequence of lengths $|J_n|/|I_n|\neq 1$  has a subsequence which either converges to $0$ or is unbounded. 
This implies that $\varphi'(0)$ is either $0$ or infinite, which contradicts the fact that $\varphi$ was a diffeomorphism. 

It follows that for any sequence $I_n$ as above and large enough $n$ we have $\varphi(I_n)=I_n$. In particular $\varphi(0)=0$.
This holds for any left point of $X$ and hence $\diff^{1}(X)=1$. 

Let now $I,J\subset [0,1]$ be closed intervals intersecting $X$ intersecting along Cantor sets. 
The arguments above also show that there exists a diffeomorphism $\varphi:(I,X\cap I)\to (J,X\cap J)$ if and only 
$I=J$ and $\varphi|_{X\cap I}$ is the identity. From the proof of Theorem \ref{V-type} (see section \ref{section-V-type}) 
we deduce that $\diff^1 _M(X)=1$ for any manifold $M$ containing the $\mathcal C^1$-interval $[0,1]$.

\subsection{Sparse Cantor set with trivial diffeomorphism group}
Start  from the interval $I^{(0)}=[0,1]$
by removing a central gap $J_1^{(1)}$ of size 
$(1-\varepsilon)$. By recurrence at the $n$-th step we have 
$2^n$ intervals $I^{(n)}_j$, $j=1, \ldots, 2^n$, numbered from the left to 
the right. 
To go further we remove a central gap 
$J^{(n+1)}_j$ from $I^{(n)}_j$ of length 
$|J^{(n+1)}_j|=(1-\varepsilon^n)|I^{(n)}_j|$.   
The set so obtained is obviously a sparse Cantor set $C_0$.

Let $a\in C_0$. Let $b_n$ be the right endpoint of 
the interval $I^{(n)}_{j_n}$ to which $a$ belongs. 
Then set $(x_n,y_n)$ for the gap $J^{(n+1)}_{j_n}\subset I^{(n)}_{j_n}$. 
There is no loss of generality of assuming that 
$a<x_n<y_n<b_n$. Given $\varphi\in\Diff^1_a(\R,C_0)$, 
with $\varphi'(a)\neq 1$, there are infinitely many $n$ 
for which the gap $J^{(n)}_{j_n}$ is not fixed by $\varphi$. 
It follows that either $\varphi(y_n)<x_n$, or 
$\varphi(x_n)>y_n$,  for infinitely many $n$. 
By symmetry we can assume that the second alternative holds. 
Then 
\begin{equation} 
\frac{\varphi(x_n)-x_n}{x_n-a}\geq \frac{|y_n-x_n|}{|x_n-a|}
\geq \frac{(1-\varepsilon^n)|b_n-a|}{|x_n-a|}\geq 
\frac{(1-\varepsilon^n)|b_n-a|}{\varepsilon^n|b_n-a|}=
\frac{1-\varepsilon^n}{\varepsilon^n}. 
\end{equation}
Letting $n\to \infty$ we obtain that $\varphi'(a)=\infty$, contradiction. 
This proves that $\diff^{1,+}_a(C_0)=1$. 

From Remark \ref{trivial} we have  $\diff^{1,+}(C_0)=1$. Moreover, 
the proof of Theorem \ref{V-type} implies that $\diff^{1,+}_M(C_0)=1$ for any manifold $M$ containing 
the $\mathcal C^1$-interval $[0,1]$.

\vspace{0.2cm}
{\bf Another potential example}.
In order to convert the nonsparse example above $X$ into a 
sparse Cantor set with the same properties, we have to mix 
ordinary gaps and very small gaps. 
Start as above from the interval $I^{(0)}=[0,1]$
by removing a central gap $LG^{(1)}$ of size 
$\frac{1}{3}$ and two very small gaps 
each one centered within 
an interval component of $I^{(0)}-LG^{(1)}$, namely 
$SG^{(1)}_1$ and $SG^{(1)}_2$ of lengths 
$2^{-2^{\alpha}}$ and $2^{-2^{\alpha+1}}$, respectively. 
Here $\alpha$ is chosen so that 
\[ \frac{1}{3}-2^{-2^{\alpha}} > \frac{1}{6}.\]
We obtain at the next stage 
four intervals $I^{(1)}_1,I^{(1)}_2, I^{(1)}_3, I^{(1)}_4$, labeled from the 
left to the right.  

By recurrence at the $n$-th step we have 
$4^n$ intervals $I^{(n)}_j$, $j=1, \ldots, 4^n$. 
To go further we remove first a central gap 
$LG^{(n+1)}_j$ from $I^{(n)}_j$ of length 
$|LG^{(n+1)}_j|=\frac{1}{3}|I^{(n)}_j|$.  
Further we remove two very small gaps 
each one centered within 
an interval component of $I^{(n)}_j-LG^{(n)}_j$, namely 
$SG^{(n)}_{2j+1}$ and $SG^{(n)}_{2j+2}$ of lengths 
$2^{-2^{\alpha+j+4^n}}$ and $2^{-2^{\alpha+j+1+4^n}}$. 
Letting $n$ go to $\infty$ we obtain a $\frac13$-sparse Cantor set $MC$. 
We believe that  $\diff^{1,+}(MC)=1$.

\subsection{Split Cantor sets}
Two Cantor sets $C_i\subset \R^n$  are 
{\em locally smoothly nonequivalent} if for any $p_i\subset C_i$ there is 
no $\mathcal C^1$-diffeomorphism germ $(\R^n,C_1,p_1)\to (\R^n,C_2,p_2)$.

A Cantor set in $C\subset \R^n$ is said to be {\em smoothly split} if 
we can write $C=C_1\cup C_2$ as a union of two Cantor sets with $C_1$ and $C_2$ 
locally smoothly nonequivalent and contained in disjoint intervals. 

We have the following easy:
\begin{proposition}
Let $n\geq 1$ and $C\subset \R^n$ be a Cantor set which is smoothly split as $C_1\cup C_2$. 
 Then $\diff^{1,+}(C)=\diff^{1,+}(C_1)\times \diff^{1,+}(C_2)$. 
\end{proposition}
\begin{proof}
In this situation $C_i$ are contained into disjoint intervals $U_i$. 
Then diffeomorphisms preserving $C$ should also 
send each $C_i$ into itself. Furthermore all elements from 
$\diff^{1,+}(C_1)\times \diff^{1,+}(C_2)$ can be realized as classes of 
pairs of commuting diffeomorphisms supported in $U_i$. 
\end{proof}

According to Remark \ref{Hausdorff} the central Cantor sets 
$C_{\lambda}$ are pairwise locally smoothly nonequivalent. 
In particular the union $C_{\lambda}\cup 2+C_{\mu}$ 
of two distinct Cantor sets, one of which is translated by $2$ 
is a split Cantor set. It follows that 
\[\diff^{1,+}(C_{\lambda}\cup 2+C_{\mu})=\diff^{1,+}(C_{\lambda})\times \diff^{1,+}(C_{\mu})\cong 
F\times F,\]
for distinct $\lambda$ and $\mu$. 
It is not clear what would be $\diff^{1,+}(C_{\lambda}\cup C_{\mu})$ or $\diff^{1,+}(C_{\lambda}+ C_{\mu})$ 
(for those parameters for which the sum is still a Cantor set).

\subsection{Questions}
The countability of $\diff^1_M(C)$ relies on the description of the 
group of germs of the stabilizer $G_a$ of a point $a\in C$ for a  finitely generated group 
$G\subset \Diff^1(M,C)$: it is {\em cyclic} when 
$\dim M=1$ and $C$ is sparse and  (a finite extension of) a subgroup of $\Z^n$ 
if $\dim M=n$ and $C$ is a product of $n$ 
sparse Cantor subsets of the line. We think that a similar result 
holds for any sparse enough Cantor subset of $M$ in higher dimensions, 
not necessarily a product.  
One should note that if the action of a group $G\subset 
\Diff^{2}(S^1)$ admits a Markov partition (see \cite{DKN3,DKN2}) 
and an exceptional minimal set $C$ then $C$ must indeed be sparse.

It seems presently unknown whether for any  $\mathcal C^1$-locally discrete 
group $G$ of $\Diff^{2}(S^1)$ with an exceptional  
minimal set every point stabilizer should have a cyclic group of germs.  
Recall that a group  $G\subset \Diff^{2}(S^1)$  is $\mathcal C^1$-locally discrete if the restriction of 
the identity to any interval intersecting its minimal set is isolated in the 
$\mathcal C^1$-topology among the restrictions of elements of $G$. In 
particular groups $G\subset \Diff^{\omega}(S^1)$ with an exceptional minimal set 
are $\mathcal C^1$-locally discrete, as well as Fuchsian groups.  
One believes that the action of every  $\mathcal C^1$-locally discrete subgroup of 
$\Diff^{\omega}(S^1)$ has a Markov partition (see \cite{AFK}, Main Conjecture).  
  
A well-known conjecture of Dippolito (see \cite{Di}, 448--449) states  
that for a finitely generated group $G\subset \Diff^{2}(S^1)$ with a 
minimal exceptional $C$ such that the groups of germs of stabilizers are 
cyclic the Radon-Nikodym derivative of every element of $G$ (with respect 
to an invariant measure) should be locally constant. This suggests that 
our results describing the elements of $\diff^1_M(C)$ for Cantor sets associated to generic affine IFS 
might be extended to more general Cantor sets. For instance,  
when $C=\Lambda\Gamma\subset S^1$ is the limit set of a second kind 
Fuchsian group  $\Gamma$ then  $\diff^{1,+}_{M}(C)$ should also be 
isomorphic to one of the Greenberg generalizations  $T_{\Gamma}$ or $V^{\pm}_{\Gamma}$ of 
the Thompson groups associated to $\Gamma$ (see \cite{Gre}).  More generally, if $\Gamma\subset \Diff^{1}(M)$ is 
a finitely generated group having an exceptional minimal set $C$, 
the group $\diff^1_M(C)$ is closely related to the group of piecewise-$\Gamma$ homeomorphisms of $(M,C)$. 
Note that an exceptional minimal set of a  Denjoy $\mathcal C^1$-diffeomorphism (i.e. without periodic points and whose derivative has bounded variation) of the circle is not generic in our sense, as the spectrum of ratios of lengths of gaps contains $1$ in its closure (see \cite{Mcd,Po}). 
On the other hand Triestino conjectured that any finitely generated 
$\mathcal C^1$-locally discrete subgroup of $\Diff^{2}(S^1)$ is $\mathcal C^1$-semi-conjugate to a subgroup of a generalized 
subgroup $T_{\Gamma}$, where now $\Gamma\subset \Diff^{\omega}(S^1)$ is Gromov-hyperbolic and $\mathcal C^1$-locally discrete.
Similar questions arise for $\mathcal C^1$-locally discrete subgroups of 
$\Diff^{2}(M,C)$ in relation with the generalized groups 
$V^{\pm}_{\Gamma}$, associated to $\Gamma\subset \Diff^{\omega}(M,C)$.

Every couple of Cantor sets which are attractors of IFS 
arising from $\mathcal C^1$-diffeomorphisms of the line have arbitrarily small perturbations in the $\mathcal C^1$-topology 
which makes them disjoint and generically their arithmetic difference is still a Cantor set, according to a remarkable result of Moreira (\cite{Mo}). We don't know how the group $\diff^1_M(C)$ varies under a $\mathcal C^1$-perturbation of the IFS and in particular 
whether it might be larger than the Thompson-type group associated to the IFS.

The validity of some version of the Tits alternative for diffeomorphism groups has its counterpart 
both for $\diff^1_M(C)$ and  the smooth mapping class groups: does any finitely generated subgroup contains a free subgroup 
on two generators or else it has a finite orbit on the Cantor set $C$? This was recently 
settled in the affirmative by Hurtado and Militon (see \cite{HuM}) for $\mathcal M^1(M,C_{\lambda})$ 
where $C_{\lambda}$ is the standard ternary Cantor set. Whether a strong Tits alternative could hold for 
the groups associated to some Cantor sets comprises the question on finding the solvable subgroups of 
$\diff^1_M(C)$, which started to be investigated in \cite{BMNR}. 
Recall that the Thompson group $F$ does not contain any free non-abelian group though is 
not virtually solvable. Notice that, by slightly extending \cite{CJN}, every finitely-generated torsion-free nilpotent group can be made act on the interval with an invariant Cantor set, but we don't know whether there exists 
a non-virtually-abelian, nilpotent group of $\mathcal C^1$-diffeomorphisms of a generic Cantor set.

\end{document}